\documentclass[11pt,reqno]{amsart}
\usepackage[utf8]{inputenc}
\usepackage{amsmath,amssymb,amsthm,mathrsfs,mathtools,color,dsfont,enumerate}
\usepackage{a4wide}
\usepackage{graphics}
\usepackage{epsfig}
\usepackage{caption}
\usepackage{subcaption}   
\usepackage[colorlinks=true]{hyperref}
\hypersetup{%
  ,urlcolor=blue
  ,citecolor=red
  ,linkcolor=blue}
\usepackage[]{diagbox}
\usepackage{multirow}

\newtheorem{theorem}{Theorem}[section]
\newtheorem{prop}[theorem]{Proposition}
\newtheorem{lemma}[theorem]{Lemma}
\newtheorem{corollary}[theorem]{Corollary}

\newtheorem{remark}[theorem]{Remark}

\newcommand{\R}{{\mathbb R}}
\newcommand{\N}{{\mathbb N}}
\newcommand{\E}{{\mathbb E}}
\newcommand{\Var}{\mathrm{Var}}
\newcommand{\Cov}{\mathrm{Cov}}

\newcommand{\cN}{{\mathcal N}}
\renewcommand{\P}{{\mathbb P}}
\newcommand{\PRp}{{\mathfrak{P}(\R_+)}}
\newcommand{\PLRp}[1]{{\mathfrak{P}}_{#1}(\R_+)}
\newcommand{\Cpol}{{\mathcal C}^\infty_{\textup{pol}}(\R_+)}
\newcommand{\CpolK}[1]{{\mathcal C}^{#1}_{\textup{pol}}(\R_+)}
\newcommand{\CpolKL}[2]{{\mathcal C}^{#1,#2}_{\textup{pol}}(\R_+)}

\def\bbar{\overline}

\def\cL{\mathcal L}

\def\sjzL{\sum_{j= 0}^L}

\def\sioj{\sum_{i= 1}^j}
\def\i{\textup{I}}
\def\d{\textup{D}}

\def\blue{}%\textcolor{blue}}
\makeatletter
 \@addtoreset{equation}{section}
 \makeatother
 
\allowdisplaybreaks
\usepackage{appendix}
\title[High order approximations for the CIR process using random grids]{High order approximations of  the Cox-Ingersoll-Ross process semigroup using random grids}
\date{\today}
\author{Aurélien Alfonsi and Edoardo Lombardo}
\address{Aurélien Alfonsi, CERMICS, Ecole des Ponts, Marne-la-Vall\'ee, France. MathRisk, Inria, Paris, France.}
\email{aurelien.alfonsi@enpc.fr}
\address{Edoardo Lombardo, CERMICS, Ecole des Ponts, Marne-la-Vall\'ee, France. MathRisk, Inria, Paris, France. Universit\`a degli Studi di Roma Tor Vergata, Rome, Italy. }
\email{edoardo.lombardo@enpc.fr}
\thanks{This work benefited from the support of the ``chaire Risques financiers'', Fondation du Risque. Edoardo Lombardo is partially supported by the MIUR Excellence Department Project MatMod@TOV awarded to the Department of Mathematics, University of Rome Tor Vergata.}



\begin{document}

\begin{abstract}
  We present new high order approximations schemes for the Cox-Ingersoll-Ross (CIR) process that are obtained by using a recent technique developed by Alfonsi and Bally (2021) for the approximation of semigroups. The idea consists in using a suitable combination of discretization schemes calculated on different random grids to increase the order of convergence. This technique coupled with the second order scheme proposed by Alfonsi (2010) for the CIR leads to weak approximations of order $2k$, for all $k\in\N^*$. Despite the singularity of the square-root volatility coefficient, we show rigorously this order of convergence under some restrictions on the volatility parameters. We illustrate numerically the convergence of these approximations for the CIR process and for the Heston stochastic volatility model and show the computational time gain they give. 
\end{abstract}

\maketitle
%%%%%%%%%%%%%%%%%%%%%%%%%%%%%%%%%%%%%%%%%%%%%%%%%%%%%%%%%%%
%%%%%%%%%%%%%%%%%%%%%%%%%%%%%%%%%%%%%%%%%%%%%%%%%%%%%%%%%%%

%\hypersetup{pageanchor=true}

\section{Introduction}

The present paper develops approximations, of any order, of the semigroup $P_t f(x):=\E[f(X^x_t)]$ associated to the following Stochastic Differential Equation (SDE) known as the Cox-Ingersoll-Ross (CIR) process
\begin{equation}\label{CIR_SDE}
  X^x_t =x+ \int_0^t (a-kX^x_s)ds+ \int_0^t \sigma\sqrt{X^x_s}dW_s, \quad t\ge 0,
\end{equation}
where $W$ is a Brownian motion, $x,a\ge 0$, $k\in \R$ and $\sigma>0$. Let us recall that the process~\eqref{CIR_SDE} is   nonnegative and  the semigroup $(P_t)_{t\ge 0}$ is well defined on the space of functions $f:\R \to \R$ with polynomial growth. The diffusion~\eqref{CIR_SDE} is widely used in financial mathematics, in particular because of its simple parametrisation and the affine property that enables to use numerical methods based on Fourier techniques. We mention here the Cox-Ingersoll-Ross model~\cite{CIR} for the short interest rate and the Heston stochastic volatility model~\cite{Heston}, that have been followed by many other ones.   Developing efficient numerical methods for the process~\eqref{CIR_SDE} is thus of practical importance. 

To deal with the approximation of SDE's semigroups, a common approach is to consider stochastic approximations and the most standard one is the Euler-Maruyama scheme. The error between the approximated semigroup and the exact one is called the weak error, as opposed to the strong error that quantifies the error "omega by omega" on the probability space. 
The seminal work of Talay and Tubaro~\cite{TT} shows, under regularity assumptions on the SDE coefficients, that the weak error given by the Euler-Maruyama scheme is of order one, i.e. is proportional to the time step. They also obtain an error expansion that enables to use Richardson-Romberg extrapolations as developed by Pagès~\cite{Pages}. Higher order schemes for SDEs and related extrapolations have been proposed by Kusuoka~\cite{Kusuoka}, Ninomiya and Victoir~\cite{NV}, Ninomiya and Ninomiya~\cite{NiNi} and Oshima et al.~\cite{OTV} to mention a few. Recently, Alfonsi and Bally~\cite{AB} have given a method to construct weak approximation of general semigroups of any order by using random time grids. 

These general results on weak approximation of SDEs do not apply to the CIR~\eqref{CIR_SDE} process. This is due to the diffusion coefficient, namely the singularity of the square-root at the origin. Besides this, classical schemes such as the Euler-Maruyama scheme are not well-defined for~\eqref{CIR_SDE}, and one has to work with dedicated schemes.   Under some restrictions on the parameters, the weak convergence of order one for some discretization schemes of the CIR process has been obtained by Alfonsi~\cite{AA_MCMA}, Bossy and Diop~\cite{BoDi}, and more recently by Briani et al.~\cite{BrCaTe} who also study the weak convergence of a semigroup approximation for the Heston model. \blue{We also mention the earlier work by Altmayer and Neuenkirch~\cite{AlNe} that precisely studies the weak error for the Heston model. } 
 Adapting ideas from Ninomiya and Victoir~\cite{NV} who developed a second order scheme for general SDEs, Alfonsi~\cite{AA_MCOM} has introduced second order and third order schemes for the CIR and proved their weak order of convergence, without any restriction on the parameters.

 The goal of the present paper is to boost the second order scheme developed in~\cite{AA_MCOM} and get approximations of any order. To do so, we rely on the method developed recently by Alfonsi and Bally~\cite{AB} to construct approximation of semigroups of any order. Roughly speaking, this method allows to get, from an elementary weak approximation scheme of order~$\alpha>0$, approximation schemes of any order by computing the elementary scheme on appropriate random grids. The method is illustrated in~\cite{AB} on the case of the Euler-Maruyama scheme for SDEs, under regularity assumptions on the coefficients that do not hold for the CIR process~\eqref{CIR_SDE}. This method is presented briefly in Section~\ref{Sec_nutshell}. It relies on an appropriate choice of a function space endowed with a family of seminorms. Section~\ref{Sec_O2CIR} then presents the second order scheme that is used as an elementary scheme to get higher order approximation. It states in Theorem~\ref{thm_main} the main result of this paper: we prove, when $\sigma^2\le 4a$, that we get weak approximations of any orders for smooth test functions~$f$ with derivatives having at most a polynomial growth.  
 Section~\ref{Sec_pol_fct} illustrates the boosting method when considering the space of polynomials function with their usual norm. In this simple case, proofs are quite elementary so that the method can be followed easily. Section~\ref{Sec_main} is more involved: it  first defines the appropriate family of seminorms on the space of smooth functions with derivative of polynomial growth and then proves Theorem~\ref{thm_main}. Last, we illustrate in Section~\ref{Simulations} the convergence of the high order approximations for different parameter sets. It validates our theoretical results and shows important computational gains given by the new approximations. We also test the method on the Heston model and obtain similar convincing results.

\section{High order schemes with random grids: the method in a nutshell}\label{Sec_nutshell}

In this paragraph, we recall briefly the method developed by Alfonsi and Bally in~\cite{AB} to construct approximations of any order from a family of approximation schemes. We consider $F$ a vector space endowed with a family of seminorms $(\| \|_k)_{k\in\N}$ such that $\|f\|_k\leq \|f\|_{k+1}$. We consider a time horizon $T>0$ and set, for $n\in \N^*$ and $l\in \N$,
\begin{equation}\label{def_hl}
  h_l=\frac{T}{n^l}.
\end{equation}
To achieve this goal, we consider a family of linear operators $(Q_l)_{l \in \N}$ \blue{on $F$.  For $l\in \N$,  we note $Q^{[0]}_l=I$ the identity operator and, for $j\in \N^*$, $Q^{[j]}_l = Q^{[j-1]}_l Q_l$ the operator obtained by composition. We} suppose that the two following conditions are satisfied. The first quantifies how $Q_l$ approximates~$P_{h_l}$:
\blue{\begin{equation}\tag{$\bbar{H_1}$}\label{H1_bar}
  \begin{array}{c}\text{there exists } \alpha>0 \text{ such that for any }  l,k\in\N, \text{ there exists } C>0,\text{ such that }\\ \|(P_{h_l}-Q_l)f\|_k \leq C\|f\|_{\psi_{Q}(k)} h_l^{1+\alpha} \text{ for all } f\in F,
  \end{array}
\end{equation}}
where $\psi_Q: \N \to \N$ is a function\footnote{Note that in~\cite{AB}, it is taken $\psi_Q(k)=k+\beta$ for some $\beta \in \N$, but is can be easily generalized to any function~$\psi_Q$. In this paper, we will work with a doubly indexed norm and take $\psi_Q(m,L)=(2(m+3),L-1)$.}. The second one is a uniform bound with respect to all the seminorms:
\blue{\begin{equation}\tag{$\bbar{H_2}$}\label{H2_bar}
  \begin{array}{c}
    \text{for all } l,k\in \N, \text{ there exists } C>0 \text{ such that } \\ \max_{0\leq j\leq n^l}\|Q^{[j]}_l f\|_k + \sup_{t\leq T}\|P_t f\|_k\leq C\|f\|_k  \text{ for all } f\in F.
  \end{array}
\end{equation}}
Then, for any $\nu \in \N^*$, Alfonsi and Bally~\cite{AB} show how one can construct, by mixing the operators $Q_l$,  a linear operator $\hat{P}^{\nu,n}_T$ for which there exists $C>0$ and $k\in \N$  such that
\begin{equation}\label{const_Pnu}
   \|P_Tf-\hat{P}^{\nu,n}f\|_0\leq C\|f\|_k n^{-\nu \alpha} \text{ for all } f\in F.
\end{equation}

Let us explain how it works for $\nu=1$ and $\nu=2$. For $\nu=1$, we mainly repeat the proof of Talay and Tubaro~\cite{TT} for the weak error of the Euler scheme. From the semigroup property, we have
\begin{equation}\label{dev_1}P_Tf-Q^{[n]}_1f=P_{nh_1}f-Q^{[n]}_1f=\sum_{k=0} ^{n-1}P_{(n-(k+1))h_1}[P_{h_1}-Q_1]Q_1^{[k]}f.
\end{equation}
We get by using~\eqref{H2_bar}, then~\eqref{H1_bar} and then again~\eqref{H2_bar}
\begin{align}
  \|P_Tf-Q^{[n]}_1f\|_0 & \le\sum_{k=0} ^{n-1} C \|[P_{h_1}-Q_1]Q_1^{[k]}f\|_0 \le \sum_{k=0} ^{n-1} C \|Q_1^{[k]}f\|_{\psi_Q(0)}h_1^{1+\alpha} \notag \\
                        & \le  C \|f\|_{\psi_Q(0)} n(T/n)^{1+\alpha}= C \|f\|_{\psi_Q(0)} T^{1+\alpha}n^{-\alpha}.\label{nu_egal_1}
\end{align}
Here, and through the paper, $C$ denotes a positive constant that may change from one line to another. So, $\hat{P}^{1,n}=Q^{[n]}_1$ satisfies~\eqref{const_Pnu}  with $\nu=1$, $k=\psi_Q(0)$. The approximation scheme simply consists in using $n$ times the scheme~$Q_1$, which can be seen as a scheme on the regular time grid with time step~$h_1$.

We now present the approximation scheme~\eqref{const_Pnu} for $\nu=2$. To do so, we use again~\eqref{dev_1} to get $P_{(n-(k+1))h_1}-Q_1^{[n-(k+1)]}=\sum_{k'=0}^{n-(k+2)}P_{(n-(k+k'+2))h_1}[P_{h_1}-Q_1]Q_1^{[k']}$ and then expand further~\eqref{dev_1}:
\begin{align}
  P_Tf-Q^{[n]}_1f            & =\sum_{k=0}^{n-1}Q_1^{[n-(k+1)]}[P_{h_1}-Q_1]Q_1^{[k]}f + R_2^{h_1}(n)f, \label{dev_2}                          \\
  \text{ with } R_2^{h_1}(n) & =\sum_{k=0}^{n-1}\sum_{k'=0}^{n-(k+2)}P_{(n-(k+k'+2))h_1}[P_{h_1}-Q_1] Q_1^{[k']}[P_{h_1}-Q_1] Q_1^{[k]} \nonumber
\end{align}
Using~\eqref{H1_bar} three times  and~\eqref{H2_bar} twice, we obtain $$\| R_2^{h_1}(n)f\|_0 \le C \| f\|_{\psi_Q(\psi_Q(0))}\frac{n(n-1)}2 h_1^{2(1+\alpha)}\le C \| f\|_{\psi_Q(\psi_Q(0))}\frac{T^{2(1+\alpha)}}{2} n^{-2\alpha}.$$
Thus, $Q^{[n]}_1 +\sum_{k=0}^{n-1}Q_1^{[n-(k+1)]}[P_{h_1}-Q_1]Q_1^{[k]}f $ is an approximation of order $2\alpha$, but it still involves the semigroup through $P_{h_1}$. To get an approximation that is obtained only with the operators $Q_l$, we use again~\eqref{dev_1} with time step~$h_2$ and final time~$h_1=nh_2$:
$$P_{h_1}f-Q_2^{[n]}f=\sum_{k=0} ^{n-1}P_{(n-(k+1))h_2}[P_{h_2}-Q_2]Q_2^{[k]}f. $$
We have $\|P_{h_1}f-Q_2^{[n]}f\|_0\le C \|f\|_{\psi_Q(0)}n h_2^{1+\alpha}$ by using again \eqref{H1_bar} and~\eqref{H2_bar}. We get from~\eqref{dev_2}
\begin{equation}\label{devt_erreur}
  P_Tf-Q^{[n]}_1f=\sum_{k=0}^{n-1}Q_1^{[n-(k+1)]}[Q_2^{[n]}-Q_1]Q_1^{[k]}f+ \sum_{k=0}^{n-1}Q_1^{[n-(k+1)]}[P_{h_1}-Q_2^{[n]}]Q_1^{[k]}f+ R_2^{h_1}(n)f,
\end{equation}
with  $\|\sum_{k=0}^{n-1}Q_1^{[n-(k+1)]}[P_{h_1}-Q_2^{[n]}]Q_1^{[k]}f\|_0\le C  \|f\|_{\psi_Q(0)}n^2 h_2^{1+\alpha}=C  \|f\|_{\psi_Q(0)}T^{1+\alpha}n^{-2\alpha}$. Therefore, the approximation
\begin{equation}\label{def_P_2}
  \hat{P}^{2,n}f:=Q^{[n]}_1f+\sum_{k=0}^{n-1}Q_1^{[n-(k+1)]}[Q_2^{[n]}-Q_1]Q_1^{[k]}f
\end{equation}
satisfies~\eqref{const_Pnu} with $\nu=2$ and is obtained only with the approximating operators~$Q_l$.  The first term~$Q_1^{[n]}$ corresponds to apply the scheme~$Q_1$ on the regular time grid with time step~$h_1$, while each term  $Q_1^{[n-(k+1)]}[Q_2^{[n]}-Q_1]Q_1^{[k]}$ is the difference between this scheme and the one where $Q_2^{[n]}$ is used instead of $Q_1$ for the $(k+1)$-th time step. This amounts to refine this time step and split it into $n$ time steps of size~$h_2$, and to use the scheme $Q_2$ on this time grid.

In practice, it is inefficient to calculate one by one the terms in $\hat{P}^{2,n}f$. In fact, each term requires a number of calculations that is proportional to~$n$, and the overall computation cost would be of the same order as~$n^2$. Since  the convergence is in $O(n^{-2\alpha})$ it would not be better asymptotically than using $\hat{P}^{1,n^2}f$. To avoid this, we use randomization. We sample a uniform random variable~$\kappa$ on $\{0,\dots,n-1\}$ and calculate $n\E[Q_1^{[n-(\kappa+1)]}[Q_2^{[n]}-Q_1]Q_1^{[\kappa]}f]=\sum_{k=0}^{n-1}Q_1^{[n-(k+1)]}[Q_2^{[n]}-Q_1]Q_1^{[k]}f$. This amounts to consider the regular time grid with time step~$h_1$, to select randomly one time step and to refine it, and then to compute the difference between the approximations on the (random) refined time-grid and on the regular time-grid. \blue{To be more precise, let us consider the case of an approximation scheme defined by $\varphi(x,h,V)$ where $\varphi$ is a measurable function, $x$ is the starting point, $h$ the time step and $V$  a random variable. The associated operators are   
 $Q_lf(x)=\E[f(\varphi(x,h_l,V))]$, $l\in \N$. For a time-grid $\Pi=\{0=t_0<\dots<t_n=T\}$, we define $X^\Pi_0(x)=x$ and $X^\Pi_{t_{i}}(x)=\varphi(X^\Pi_{t_{i-1}}(x),t_i-t_{i-1},V_i)$ for $1\le i\le n$, where $(V_i)_{i\ge 1}$ is an i.i.d. sequence. Thus, we get on the uniform time grid $\Pi^0=\{ kT/n, 0\le k\le n \}$ $\E[f(X^{\Pi^0}_{T}(x))]=Q_1^{[n]}f(x)$. By taking the random grid  $\Pi^1 = \Pi^0 \cup  \{ \kappa T/n + k'T/n^2 , 1 \leq k' \leq n-1 \}$, where $\kappa$ is an independent uniform random variable on $\{0,\ldots,n-1\}$, we also get $\E[f(X^{\Pi^1}_{T}(x))]=\E[Q_1^{[n-(\kappa+1)]}[Q_2^{[n]}-Q_1]Q_1^{[\kappa]}f(x)]$, and then $\E[n(f(X^{\Pi^1}_{T}(x))-f(X^{\Pi^0}_{T}(x)))]=\sum_{k=0}^{n-1}Q_1^{[n-(k+1)]}[Q_2^{[n]}-Q_1]Q_1^{[k]}f(x)$. When using a Monte-Carlo estimator of this identity, one has thus to draw as many $\kappa$'s as trajectories. 
}

We have presented here how to construct~$\hat{P}^{\nu,n}$ for $\nu=1$ and $\nu=2$, and it is possible by repeating the same arguments to construct by induction approximations of any order. Unfortunately, the induction is quite involved. It is fully described in~\cite[Theorem 3.10]{AB}. We do not reproduce it in this paper because it would require much more notation, and we will mainly use the scheme~\eqref{def_P_2}. Here, we give in addition the explicit form of   $\hat{P}^{3,n}$, $n\ge 2$:
\begin{align}\label{def_P_3}
  \hat{P}^{3,n}f:= & \hat{P}^{2,n}+  \sum_{0\le k_1<k_2<n}^{n-1}Q_1^{[n-(k_2+1)]}[Q_2^{[n]}-Q_1]Q_1^{[k_2-k_1-1]}[Q_2^{[n]}-Q_1]Q_1^{[k_1]}f     \\
                   & +\sum_{k=0}^{n-1}Q_1^{[n-(k+1)]}\left[\sum_{k'=0}^{n-1}Q_2^{[n-(k'+1)]}[Q_3^{[n]}-Q_2]Q_2^{[k']} \right] Q_1^{[k]}f. \notag
\end{align}
By similar arguments, it satisfies~\eqref{const_Pnu} with $\nu=3$.

\section{Second order schemes for the CIR process and main result}\label{Sec_O2CIR}

In this section, we focus on the approximation of the semigroup of the CIR process $P_tf(x)=\E[f(X^x_t)]$, where
$$X_t^x=x + \int_0^t(a-kX^x_s)ds+ \sigma \int_0^t\sqrt{X^x_s}dW_s, \ t\ge 0.$$
Equation~\eqref{nu_egal_1} shows that, necessarily, approximating operators $Q_l$ that satisfy both~\eqref{H1_bar} and~\eqref{H2_bar} lead to a weak error of order~$\alpha$. Therefore, we are naturally interested in approximation schemes of the CIR for which we know the rate of convergence~$\alpha$ for the weak error. \cite[Proposition 4.2]{AA_MCMA} gives a rate $\alpha=1$ for a family of approximation schemes that are basically obtained as a correction of the Euler scheme. Ninomiya and Victoir~\cite{NV} have developed a generic method to construct second order schemes ($\alpha=2$) for Stochastic Differential Equations with smooth coefficients. Applied to the Cox-Ingersoll-Ross process, their method leads to the following approximation scheme
\begin{equation}\label{NV_scheme}
  \hat{X}^x_t=\varphi(x,t,\sqrt{t}N),
\end{equation}
where  $N\sim \mathcal{N}(0,1)$ and  $\varphi:\R_+\times\R_+\times\R\to \R_+$ is defined by
\begin{align}\label{def_varphi}
  \varphi(x,t,w) & = e^{-kt/2}\left(\sqrt{(a-\sigma^2/4)\psi_k(t/2)+e^{-kt/2}x}+\sigma w/2 \right)^2 +(a-\sigma^2/4)\psi_k(t/2) \\
                 & =X_0(t/2,X_1(w,X_0(t/2,x))), \text{ with} \notag
\end{align}
\begin{align}\label{def_X0_and_psi_k}
  X_0(t,x) & = e^{-kt}x+\psi_k(t)(a-\sigma^2/4), \quad \psi_k(t)=\frac{1-e^{-kt}}k, \\ X_1(t,x)&=(\sqrt{x}+ t\sigma/2)^2, \label{def_X1}
\end{align}
with the convention that $\psi_0(t)=t$. This scheme corresponds to approximate $P_tf(x)$ by $\hat{P}_tf(x)=\E[f( \hat{X}^x_t)]$ for $x,t\ge 0$, and then to set $Q_l=\hat{P}_{h_l}$. Its construction comes from the splitting of the infinitesimal generator of the CIR process
\begin{equation}\label{def_LCIR}
  \cL f(x)=(a-kx)f'(x)+\frac 12 \sigma^2 x f''(x),\ f \in \mathcal{C}^2, x\ge 0,
\end{equation}
as $\cL=V_0+\frac 12 V_1^2$ with
\begin{equation}\label{def_V0_V1}
  V_0f(x)=\left(a-\frac{\sigma^2}4 -kx\right)f'(x) \text{ and } V_1f(x)=\sigma \sqrt{x} f'(x).
\end{equation}
The function $t\mapsto X_0(t,x)$ is the solution of the ODE $X_0'(t,x)=a-\frac{\sigma^2}4 -kX_0(t,x)$ such that $X_0(0,x)=x$, while $X_1(W_t,x)$ solves the SDE associated to the infinitesimal generator $V_1^2/2$.

The scheme~\eqref{NV_scheme} is well defined for $\sigma^2\le 4a$. Instead, for $\sigma^2> 4a$, it is not well defined for any $x\ge 0$ since the argument in the square-root is negative when $x$ is close to zero. To correct this, Alfonsi~\cite{AA_MCOM} has proposed the following scheme
\begin{align}\label{Alfonsi_scheme}
  \hat{X}^x_t =  (\mathds{1}_{x\geq K^Y_2(t)}\varphi(x,t,\sqrt{t}Y) + \mathds{1}_{x< K^Y_2(t)} \hat{X}^{x,d}_t),
\end{align}
where $Y$ is a random variable with compact support on $[-A_Y,A_Y]$ for some $A_Y>0$ such that $\E[Y^k]=\E[N^k]$ for $k\le 5$, and  $\hat{X}^{x,d}_t$ is a nonnegative random variable such that $\E[(\hat{X}^{x,d}_t)^i]=\E[(X^{x}_t)^i]$ for $i\in \{1,2\}$ and $K^Y_2(t)$ is a nonnegative threshold defined by
\begin{equation}\label{threshold_gen}
  K^Y_2(t) =\mathds{1}_{\sigma^2>4a} \left[ e^{\frac{kt}{2}}\left( (\sigma^2/4-a)\psi_k(t/2)+ \bigg( \sqrt{e^{\frac{kt}2}(\sigma^2/4-a)\psi_k(t/2)} + \frac \sigma 2 A_Y\sqrt{t} \bigg)^2 \right)\right].
\end{equation}
Note that when $\sigma^2\le 4a$, we have $K^Y_2(t)=0$ and thus $\hat{X}^x_t = \varphi(x,t,\sqrt{t}Y)$.
\blue{  In~\cite{AA_MCOM}, it is taken $Y$ such that $\P(Y=\sqrt{3})=\P(Y=-\sqrt{3})=1/6$ and $\P(Y=0)=2/3$, and a discrete  random variable 
$\hat{X}^{x,d}_t$ such that $\P(\hat{X}^{x,d}_t=\frac{1}{2 \pi(t,x)})=\pi(t,x)$, $\P(\hat{X}^{x,d}_t=\frac{1}{2 (1-\pi(t,x))})=1-\pi(t,x)$ where $\pi(t,x)=\frac{1-\sqrt{1-\E[(X^{x}_t)]^2/\E[(X^{x}_t)^2]}}{2} \in (0,1/2)$. 
}

We now restate~\cite[Theorem 2.8]{AA_MCOM} that analyses the weak error. \blue{We introduce $\CpolK{k}$, the set of $\mathcal{C}^k$ functions $f:\R \to \R_+$ such that all its derivatives have polynomial growth. More precisely, this means that for all $k'\in\{0,\dots,k\}$, there exists $C_{k'},E_{k'}\in \R_+$ such that
$$|f^{(k')}(x)|\le C_{k'}(1+x^{E_k'}), \ x\ge 0.$$
We also set $\Cpol=\cap_{k\in \N}\CpolK{k}$. }

\begin{theorem}\label{thm_MCOM}
  Let $\hat{X}^x_t$ be the scheme defined by~\eqref{NV_scheme} for $\sigma^2\le 4a$ or by~\eqref{Alfonsi_scheme} for any $\sigma>0$. %Let $f:\R_+\to \R$ be a $\mathcal{C}^\infty$ function such that:
  % $$\forall k \in \N, \exists C_k, E_k \in \R_+, \forall x\ge 0, |f^{(k)}(x)|\le C_k(1+|x|^{E_k}).$$
  \blue{Then, for all $f\in \Cpol$, we have $Q_1^{[n]}f(x)-P_Tf(x)=O(1/n^2)$ where $Q_1f(x)=\E[f( \hat{X}^x_{h_1})]$.}
\end{theorem}
The goal of this paper is to extend this result and prove the estimates~\eqref{H1_bar} and~\eqref{H2_bar} for a suitable space of functions and a suitable family of seminorms. We are able to prove such results only in the case $\sigma^2\le 4a$: the indicator function in~\eqref{Alfonsi_scheme} creates a singularity that is difficult to handle in the analysis. In Section~\ref{Sec_pol_fct}, we first prove~\eqref{H1_bar} and~\eqref{H2_bar} for polynomial test functions. Then, we deal in Section~\ref{Sec_main} with the much technical case of smooth test functions with derivatives of polynomial growth. We state here our main result, the proof of which is given in Section~\ref{Sec_main}.
\begin{theorem}\label{thm_main}
  Let $\hat{X}^x_t$ be the scheme defined by~\eqref{NV_scheme} for $\sigma^2\le 4a$ and $Q_lf(x)=\E[f(\hat{X}^x_{h_l})]$, for $l\ge 1$. %Let $f:\R_+\to \R$ be a $\mathcal{C}^{18}$ function such that:
  %$$\forall k \le 18, \exists C_k, E_k \in \R_+, \forall x\ge 0, |f^{(k)}(x)|\le C_k(1+|x|^{E_k}).$$
  Then, for all $f\in \CpolK{18}$, we have $\hat{P}^{2,n}f(x)-P_Tf(x)=O(1/n^4)$ as $n\to \infty$.\\
 \blue{Besides, for  $f\in \Cpol$, we have $\hat{P}^{{\nu,n}}f(x)-P_Tf(x)=O(1/n^{2\nu})$.}
\end{theorem}
%{\bf REMARK: in fact, it is sufficient to have $f \in \mathcal{C}^{6(2^\nu-1)}$.}
\blue{Let us stress here that Theorem~\ref{thm_main}  gives an asymptotic result as $n\to \infty$. It thus might happen that for small values of $n$,  $\hat{P}^{2,n}$ is less accurate than $\hat{P}^{1,n}=Q_1^{[n]}$ for some $f\in \Cpol$ and $x\ge 0$. In practice, we have always noticed in our numerical experiments that $\hat{P}^{2,n}$ is more accurate than $\hat{P}^{1,n}$. However, the estimated rates of convergence obtained from relatively small values of $n$ may be different from the theoretical asymptotic ones, see Figures~\ref{CIR_Plot1},\ref{CIR_Plot2} and~\ref{CIR_Plot3} where are given the estimated rates for $\hat{P}^{1,n}$, $\hat{P}^{2,n}$ and $\hat{P}^{3,n}$.
}

\section{The case of polynomial test functions}\label{Sec_pol_fct}

In this section, we want to illustrate the method and consider test functions that are  polynomial test functions. We define for $L\in \N$
$$\PLRp{L}=\{ f: \R_+\to \R,  f(x)=\sum_{j=0}^L a_j x^j \text{ for some } a_0,\dots,a_L \in \R \},$$
the vector space of polynomial functions over $\R_+$ with degree less or equal to~$L$. We also define  $\PRp= \cup_{L\in \N} \PLRp{L}$ the space of polynomial functions.
We endow $\PRp$ with the following norm:
\begin{equation}\label{norm_pol}
  \|f\| = \sjzL |a_j|, \text{ for } f(x)=\sjzL a_j x^j.
\end{equation}

We consider the case $\sigma^2\le 4a$ and consider the scheme~\eqref{Alfonsi_scheme} for the CIR process with a time step~$t>0$, $\hat{X}^x_t =  \varphi(x,t,\sqrt{t}Y)$.  The approximation scheme $Q_l$ is then defined by $Q_lf=\E[f(\hat{X}^x_{h_l})]$. The goal of this section is to prove~\eqref{H1_bar} and~\eqref{H2_bar} for the norm~\eqref{norm_pol}. We make the following assumption on~$Y$.

\noindent {\bf Assumption~$(\mathcal{H}_Y)$:} $Y:\Omega\to \R$ is a symmetric random variable such that $\E[|Y|^k]<\infty$ for all $k\in \N$, and $\E[Y^k]=\E[N^k]$ for $k\in \{2,4\}$ with $N\sim \mathcal{N}(0,1)$.

We now state two lemmas that will enable us to prove that~\eqref{H2_bar} is satisfied by the scheme~\eqref{Alfonsi_scheme}. Lemma~\ref{estimates_pol} shows that polynomials functions are preserved by the approximation scheme, and gives short time estimate for the polynomial norm. Lemma~\ref{Moments_Formula_CIR} gives similar results for the CIR diffusion. The proofs of these lemmas are quite elementary and are postponed to Appendix~\ref{App_proof_sec_pol}.

\begin{lemma}\label{estimates_pol}
  Let $T\ge 0$, $t\in[0,T]$, $f\in\PLRp{L}$ and assume~$(\mathcal{H}_Y)$ and $\sigma^2\le 4a$.  Then, we have $f(X_0(t,\cdot)),\E[f(X_1(\sqrt{t}Y, \cdot ))] \in \PLRp{L}$ where $X_0$ and $X_1$ are defined by~\eqref{def_X0_and_psi_k} and~\eqref{def_X1}, and
  \begin{enumerate}
    \item $\|f(X_0(t,\cdot))\| \leq (1\vee e^{-kLt})(1+C_{X_0}^Lt) \|f\|$,
    \item $\| \E[f(X_1(\sqrt{t}Y, \cdot ))] \| \leq (1+\E[Y^{2L}]C_{X_1}^Lt) \| f \|,$
  \end{enumerate}
  for some constants $C_{X_0},C_{X_1}$ depending only on~$(a,\sigma,T)$.
\end{lemma}

\begin{lemma}\label{Moments_Formula_CIR}
  Let $(X^x_t,t\ge 0)$ be the CIR process starting from $x\in\R_+$. For $m\in \N$, we define $\tilde{u}_m(t,x):= \E[(X^x_t)^m]$. There exists $C^\infty$ functions $\tilde{u}_{j,m}:\R_+\to \R$ that depend on $(k,a,\sigma)$ such that:
  \begin{equation}
    \tilde{u}_m(t,x) = \sum_{j=0}^m \tilde{u}_{j,m}(t) x^j.
  \end{equation}
  If $f \in \PLRp{L}$, then we have $\E[f(X^\cdot_t)]\in \PLRp{L}$ and for $t\in [0,T]$,
  \begin{equation}\label{Pol_Estimate_CIR}
    \|\E[f(X^{\cdot}_t)]\| \leq C_{\text{cir}}(L,T) \|f\|,
  \end{equation}
  with $C_{\text{cir}}(L,T) = \max_{t\in[0,T],m\in \{0,\dots,L\} } \sum_{j=0}^m |\tilde{u}_{j,m}(t)|$.
\end{lemma}
We are now in position to prove the main result of this section, which is a weaker (but easier to prove) version of our main Theorem~\ref{thm_main}, since it only applies to polynomial test functions. \blue{Let us point however that it applies to a larger family of schemes, namely to the schemes $\varphi(x,t,\sqrt{t}Y)$ with $Y$ satisfying~$(\mathcal{H}_Y)$, while Theorem~\ref{thm_main} requires to take $Y\sim \mathcal{N}(0,1)$.}

\begin{prop}
  Let $\sigma^2\le 4a$ and assume that $Y$ satisfies~$(\mathcal{H}_Y)$.  For any $L \in \N$, the properties~\eqref{H1_bar} and \eqref{H2_bar} are satisfied by the scheme~\eqref{Alfonsi_scheme} $\hat{X}^x_t=\varphi(x,t, \sqrt{t}Y)$ for $F=\PLRp{L}$ and the norm~\eqref{norm_pol}. Then, we have for any $f\in \PLRp{L}$,
  $$\|\E[f(X^x_T)-\hat{P}^{\nu,n}f\| \le C_L \|f\| n^{-2\nu},$$
  for some constant $C_L$.
\end{prop}

\begin{proof}
  We first prove~\eqref{H2_bar}. The property $\sup_{t\in[0,T]}\|P_tf\|$ is given by Lemma~\ref{Moments_Formula_CIR}.
  Since $\hat{X}^x_t=X_0(t/2,X_1(\sqrt{t}Y,X_0(t/2,x)))$, we get by Lemma~\ref{estimates_pol}
  $$
    \|\E[f(\hat{X}^\cdot_t)]\|\le  [(1\vee e^{-kLt/2})(1+C_{X_0}^Lt/2)]^2(1+\E[Y^{2L}]C_{X_1}^Lt)  \|f\|.
  $$
  We now use that $1+x\le e^x$ to get
  \begin{equation}\label{maj_scheme}
    \| \E[f(\hat{X}^\cdot_t)] \| \leq e^{((-k)^+L+C_{X_0}^L+\E[Y^{2L}]C_{X_1}^L)t} \|f\|.
  \end{equation}
  Since $Q_lf(x)=\E[f(\hat{X}^x_{T/n^l})]$, this yields to $\max_{0\le j \le n^l}\|Q^{[j]}_lf\| \le  e^{((-k)^+L+C_{X_0}^L+\E[Y^{2L}]C_{X_1}^L)T} \|f\|$.

  We now prove~\eqref{H1_bar}. Let $m\in \{0,\dots,L\}$ and $0<x_0<\dots<x_L$ be fixed real numbers (one may take for example $x_\ell=\ell+1$).  Lemmas~\ref{estimates_pol} and~\ref{Moments_Formula_CIR} give that $v_m(t,x)=\E[(\hat{X}^x_t)^m]-\E[(X^x_t)^m]=\sum_{j=0}^m v_{j,m}(t)x^j $. By~\cite[Proposition 2.4]{AA_MCOM}, we know that there exists $C'_m,E'_m$ such that for all $t \in (0,1)$, $|v_m(t,x)|\le C'_m t^3(1+|x|^{E'_m})$. Therefore, there exists  $\tilde{C}_m\in \R_+$ such that for all $ \ell \in \{0,\dots,L\}$, $|v_m(t,x_\ell)|\le \tilde{C}_m t^3$. By using the invertibility of the Vandermonde matrix, we get the existence of $C_m \in \R_+$ such that
  $$ |v_{j,m}(t)|\le C_m t^3,  \ j\in \{0,\dots,m\}.$$
  Therefore, we get for $f \in \PLRp{L}$
  $$\|\E[f(\hat{X}^\cdot_t)]-\E[f(X^\cdot_t)]\| \le \sum_{m=0}^L|a_m|\sum_{j=0}^m C_m t^3\le L \max_{m\in \{0,\dots,L\}}C_m \|f\| t^3,$$
  that gives~\eqref{H1_bar}. We conclude by applying~\cite[Theorem 3.10]{AB}.
\end{proof}

\section{Proof of Theorem~\ref{thm_main}}\label{Sec_main}

In Section~\ref{Sec_pol_fct}, we have obtained the convergence for test functions that are polynomial functions. For these test functions, the choice of the norm is straightforward and the proofs are not very technical and quite easy. However, one would like to obtain the convergence result for a much larger class of test functions. This is the goal of this section.

We consider test functions that are smooth with polynomial growth, whose derivatives have a polynomial growth. Namely, we introduce for $m,L \in \N$,
\begin{equation}
  \CpolKL{m}{L}= \left\{f:\R_+ \to \R \text{ of class } \mathcal{C}^m \ :  \ \max_{j\in\{0,\ldots,m\}} \sup_{x\geq 0}\frac{|f^{(j)}(x)|}{1+x^L}<\infty \right\},
\end{equation}
which we endow with the norm
\begin{equation}\label{mLnorm}
  \|f\|_{m,L} = \max_{j\in\{0,\ldots,m\}} \sup_{x\geq 0} \frac{|f^{(j)}(x)|}{1+x^L}.
\end{equation}

To prove Theorem~\ref{thm_main}, we need to prove the estimates~\eqref{H1_bar} and~\eqref{H2_bar} for this family of norms. This is the goal of the two next subsections. \blue{More precisely, we will show respectively the estimates
  $$\|(P_{h_l}-Q_l)f\|_{m,L+3}\le C h_l^3  \|f\|_{2(m+3),L}, \  m\le L+3, f\in  \CpolKL{2(m+3)}{L} $$
in Proposition~\ref{H1_bar_mL}  and 
  $$\sup_{t\ge T} \|P_{t}f\|_{m,L}+ \max_{0\le j \le n^l} \|Q^{[j]}_l f\|_{m,L} \le \|f\|_{m,L} C h_l^3, \  m\le L, f\in  \CpolKL{m}{L} $$
in Proposition~\ref{prop_H2_NV} for $Q_l$ as in Theorem~\ref{thm_main}. Note that $L$ has to be large enough: this is not an issue for our purpose since $\CpolKL{m}{L}\subset \CpolKL{m}{L+1}$, and we can work with $L$ as large as needed. We refer to the proof of Theorem~\ref{thm_main} in Subsection~\ref{Subsec_thm_main} for further details. }

Before, we summarize in the next lemma some properties of the norms defined in Equation~\eqref{mLnorm} that we will use later on. Its proof is postponed to Appendix~\ref{App_proof_sec_5}
\begin{lemma} \label{lem_estimnorm}
  Let $m,L\in\N$. We have the following basic properties:
  \begin{enumerate}
    \item $\|f\|_{m',L}= \max_{j\in\{0,\ldots, m'\}}\|f^{(j)}\|_{0,L}$ for $f\in\CpolKL{m}{L}$ and $m'\in\{0,\ldots,m\}$.
    \item $\CpolKL{m+1}{L}\subset \CpolKL{m}{L}$ and  $\|f\|_{m,L} \leq \|f\|_{m+1,L}$ for $f\in\CpolKL{m+1}{L}$.
    \item $\|f^{(i)}\|_{m,L} \leq \|f\|_{m+i,L}$ for  $i\in \N$ and $f\in \CpolKL{m+i}{L}$.
    \item $\CpolKL{m}{L}\subset \CpolKL{m}{L+1}$ and $\|f\|_{m,L+1} \leq 2\|f\|_{m,L}$ for $f\in\CpolKL{m}{L}$.
    \item Let $\mathcal{M}_1$ be the operator defined by $f\mapsto \mathcal{M}_1 f$, $\mathcal{M}_1 f(x)=xf(x)$. Then, $\mathcal{M}_1f\in  \CpolKL{m}{L+1}$ for $f\in  \CpolKL{m}{L}$ and $\|\mathcal{M}_1 f\|_{m,L+1} \leq (2m+3)\|f\|_{m,L}$.
    \item Let $\cL f (x)=(a-kx)f'(x)+\frac 12 \sigma^2 x f''(x)$ be the infinitesimal generator of the CIR process. Then, we have for $f \in \CpolKL{m+2}{L}$,
          $$\|\cL f\|_{m,L+1}\le  \left(2a  + (2m+3)(|k| +\sigma^2/2) \right) \|f\|_{m+2,L}.  $$
          We also have $\|(V_1^2/2) f\|_{m,L+1}\le  \sigma^2 (m+2) \|f\|_{m+2,L}$ and $\|V_0 f\|_{m,L+1}\le [2|a-\sigma^2/4|+(2m+3)|k|] \|f\|_{m+1,L}$, where $V_0$ and $V_1$ are defined by~\eqref{def_V0_V1}.
  \end{enumerate}
\end{lemma}

We also state the following elementary lemma that will be useful to prove both~\eqref{H1_bar} and~\eqref{H2_bar}.

\begin{lemma}\label{X0_inequalities}
  Let $T>0$, $\sigma^2\le 4a$ and $X_0$ be defined by~\eqref{def_X0_and_psi_k}. Then, there exists a constant $K\ge 0$ such that
  for any function $f\in\CpolKL{m}{L}$, we have
  $$ \|f(X_0(t,\cdot))\|_{m,L} \leq  e^{Kt} \|f\|_{m,L}, \ t \in [0,T].$$
\end{lemma}
\begin{proof}
  We first prove the following inequality
  $$1 + X_0(t,x)^L \leq  (1\vee e^{-Lkt})(1+\tilde{C}_{X_0}t)(1+x^L),$$
  for some constant $\tilde{C}_{X_0}$. To do so,  we develop the term $X_0(t,x)^L$ and get
  \begin{align*}
    1 + X_0(t,x)^L & = 1+\sum_{j=0}^L {L\choose j} e^{-(L-j)kt} x^{(L-j)}(\psi_k(t)(a-\sigma^2/4))^{j}                              \\
                   & = 1+e^{-Lkt}x^L +\psi_k(t)\sum_{j=1}^{L} {L\choose j} e^{-(L-j)kt} x^{(L-j)}\psi_k(t)^{j-1}(a-\sigma^2/4)^{j}.
  \end{align*}
  We remark that for $k\ge 0$, $0\le \psi_k(t)\leq t \leq 1\vee T$  for all $t\in[0,T]$. For $k<0$, we have $\psi_k(t)=e^{-kt}\psi_{-k}(t)$ and thus $\psi_k(t)\le e^{(-k)^+t}t$  for all $t\in[0,T]$ and $k\in \R$. Using $x^j\leq 1+x^L$ for all $j\in\{1,\ldots,L\}$, we can rewrite the previous identity as
  \begin{align*}
    1 + X_0(t,x)^L & \leq (1\vee e^{-Lkt})(1+x^L) \\
    & \quad + t  e^{(-k)^+t} (1\vee e^{-Lkt})(1+x^L)  \sum_{j=0}^L {L\choose j} ( e^{(-k)^+T}(1\vee T)(a-\sigma^2/4))^{j} \\
                   & \leq (1\vee e^{-Lkt})(1+\tilde{C}_{X_0}t)(1+x^L),
  \end{align*}
  where $\tilde{C}_{X_0}=e^{(-k)^+T}(1+e^{(-k)^+T}(1\vee T)(a-\sigma^2/4))^L$.

  We are now in position to prove the claim. For $i\le m$, we have:
  \begin{align*}
    |\partial_x^i f(X_0(t,x))| =|e^{-ikt}f^{(i)}(X_0(t,x))| & \leq e^{-ikt}\|f\|_{m,L}(1+X_0(t,x)^L)                                              \\
                                                            & \leq \|f\|_{m,L} (1 \vee e^{-mkt})  (1 \vee e^{- Lkt})  (1+\tilde{C}_{X_0}t)(1+x^L) \\
                                                            & \leq \|f\|_{m,L} e^{[\tilde{C}_{X_0}+(L+m)(-k)^+] t}(1+x^L).
  \end{align*}
  This gives $\|f(X_0(t,\cdot))\|_{m,L}\le \|f\|_{m,L} e^{[\tilde{C}_{X_0}+(L+m)(-k)^+] t}$.
\end{proof}

\subsection{Proof of~\eqref{H1_bar}} In this subsection, we prove the following result which is a direct consequence of Propositions~\ref{LCIR_Expansion} \blue{(with $\nu=2$)} and~\ref{prop_H1sch} that are stated below.
\begin{prop}\label{H1_bar_mL}
  Let $Y$ satisfy~$(\mathcal{H}_Y)$, $\sigma^2\le 4a$ and $\hat{X}^x_t=\varphi(x,t,\sqrt{t}Y)$ be the scheme~\eqref{Alfonsi_scheme}. \blue{Let $m,L\in \N$ such that $L+3\ge m$ and} $f \in \CpolKL{2(m+3)}{L}$. Then, there exists a constant $C\in \R_+^*$ such that for $t\in[0,T]$,
  $$ \|\E[f(\hat{X}^\cdot_t)]- \E[f({X}^\cdot_t)]\| _{m,L+3}\le C t^3 \|f\|_{2(m+3),L}.$$
\end{prop}

To prove this result, we compare each term with the expansion $f(x)+t\cL f(x)+\frac{t^2}2\cL^2 f(x)$ of order two. The next proposition analyses the difference between such expansion and the semigroup of the CIR process.
\begin{prop}\label{LCIR_Expansion}
  Let \blue{ $m,\nu,L\in \N$ such that $L+\nu+1\ge m$}, $T>0$ and  $f\in\CpolKL{m+2(\nu+1)}{L}$. Let $X^x$ be the CIR process and $\cL$ its infinitesimal generator. Then, for  $t\in [0,T]$,  we have
  \begin{equation}\label{dev_funct_CIR}
    \E[f(X^x_t)] =\sum_{i=0}^\nu \frac{t^i}{i!}\cL^if(x) + t^{\nu+1}\int_0^1 \frac{(1-s)^\nu}{\nu!}\E[\cL^{\nu+1}f(X^x_{ts})]ds
  \end{equation}
  where the function $x\mapsto \int_0^1 \frac{(1-s)^\nu}{\nu!}\E[\cL^{\nu+1}f(X^x_s)]ds$ belongs to $\CpolKL{m}{L}$ and we have the following estimate for all $t\in [0,T]$,
  \begin{equation}\label{rem_CIR_estimate}
    \left\|\int_0^1 \frac{(1-s)^\nu}{\nu!}\E[\cL^{\nu+1}f(X^{\cdot}_{ts})]ds\right\|_{m,L+\nu+1} \le C\|f\|_{m+2(\nu+1),L},
  \end{equation}
  for some constant $C\in \R_+$ depending on $(a,k,\sigma,\nu,m,L,T)$.
\end{prop}

\begin{proof}
  Let $f\in\CpolKL{m+2(\nu+1)}{L}$. Since the coefficients of the CIR SDE have sublinear growth, we have bounds on the moments of $X^x_s$: for any $q\in\N^*$, there exists $C_q>0$ such that $\E[|X^x_s|^q]\leq C_q(1+x^q)$ for $s\in[0,T]$. Using iterations of Itô’s formula and a change of variable (in time), we then easily get \eqref{dev_funct_CIR} for $t \in [0, T]$.  To get the estimate \eqref{rem_CIR_estimate}, we first use Lemma~\ref{lem_estimnorm} and obtain $$\| \cL^{\nu+1} f \|_{m,L+\nu+1}\le K_{ cir}(m,\nu)^{\nu+1} \|f\|_{m+2(\nu+1),L},$$ with $K_{\bf cir}(m,\nu)=2a+(2m+4\nu+3)(|k|+\sigma^2/2)$.   By the triangle inequality, we have
  $$\left\|\int_0^1 \frac{(1-s)^\nu}{\nu!}\E[\cL^{\nu+1}f(X^{\cdot}_{ts})]ds\right\|_{m,L+\nu+1}\le \int_0^1 \frac{(1-s)^\nu}{\nu!} \left\|\E[\cL^{\nu+1}f(X^{\cdot}_{ts})]\right\|_{m,L+\nu+1}ds. $$
  Since $t\le T$, we have $\left\|\E[\cL^{\nu+1}f(X^{\cdot}_{ts})]\right\|_{m,L+\nu+1}\le C_{cir}(m,L+\nu+1,T)\left\|\cL^{\nu+1}f\right\|_{m,L+\nu+1}$ by  Proposition~\ref{deriv_CIR_functionals} \blue{using that $L+\nu+1\ge m$}. This gives by Lemma~\ref{lem_estimnorm}
  $$\left\|\int_0^1 \frac{(1-s)^\nu}{\nu!}\E[\cL^{\nu+1}f(X^{\cdot}_{ts})]ds\right\|_{m,L+\nu+1}\le \frac {C_{cir}(m,L+\nu+1,T)} {(\nu+1)!} K_{ cir}(m,\nu)^{\nu+1} \|f\|_{m+2(\nu+1),L}. $$
\end{proof}

We now focus on the approximation scheme. The main difficulty comes from the  differentiation of the square-root that may lead to derivatives that blow up at the origin. Here, we exploit the fact that $Y$ is a symmetric random variable to cancel these blowing terms. More precisely, we will then need to differentiate in $x$ the following quantity
$$g(X_1(s \sqrt{t},x))+g(X_1(-s\sqrt{t},x))=g(x+\sigma s \sqrt{t} \sqrt{x} + \frac{\sigma^2}4 ts^2)+g(x-\sigma s \sqrt{t} \sqrt{x} + \frac{\sigma^2}4 t s^2),$$
and the next lemma enables us to have a sharp estimate of the derivatives.

\begin{lemma}\label{regular_rep}
  Let $g:\R_+\to \R$ be a $\mathcal{C}^{2n}$ function, $\beta\in \R_+$ and $\gamma \ge \beta^2/4$. Then,
  the function  $\psi_g(x):=g(x+\beta \sqrt{x}+\gamma)+g(x-\beta \sqrt{x}+\gamma)$, $x\ge 0$   is $\mathcal{C}^n$ with derivatives
  \begin{equation}\label{formula_psign}
    \psi_g^{(n)}(x)=\psi_{g^{(n)}}(x)+\sum_{j=1}^n\binom{n}{j}\beta^{2j}\int_0^1g^{(n+j)}(x+\beta(2u-1)\sqrt{x}+\gamma) \frac{(u-u^2)^{j-1}}{(j-1)!}du
  \end{equation}
\end{lemma}
The proof of this lemma and of the next corollary are postponed to Appendix~\ref{App_proof_sec_5}.

\begin{corollary}\label{cor_psign} Let $m,L\in \N$, $\beta\ge 0$ and $g \in \CpolKL{2m}{L}$. Then, $\psi_g(x)=g((\sqrt{x}+\beta/2)^2)+ g((\sqrt{x}-\beta/2)^2)$ belongs to $\CpolKL{m}{L}$, and for all $n\in\{0,\ldots,m\}$ we have the following estimates
  \begin{equation}
    \|\psi_g\|_ {n,L} \leq C_{\beta,m,L} \|g\|_{2n,L},
  \end{equation}
  with $C_{\beta,m,L}= \big((1+\beta/2)^{2L}+(1-\beta/2)^{2L}+2(1+\beta^2/2)^{L}(1+\beta^2/2)^{m})$.
\end{corollary}

\begin{lemma}\label{lem_expan_V}
  Let $m,\nu,L\in \N, \ T>0,\ t \in [0,T]$ \blue{and $N\sim \mathcal{N}(0,1)$}. Let $Y$ be a symmetric random variable such that $\E[Y^k]=\E[N^k]$ for $k\le 2\nu$ and $\E[Y^{2k}]<\infty$ for all $k\in \N$.  We have, for $f\in \CpolKL{m+\nu+1}{L}$,
  \begin{equation}
    f(X_0(t,x)) = \sum_{i=0}^\nu \frac{t^i}{i!} V^i_0f(x) +t^{\nu+1} \int_0^{1} \frac{(1-u)^{\nu}}{\nu!}V^{\nu+1}_0 f(X_0(u t,x))du,  \label{expan_X0}
  \end{equation}
  with $\|\int_0^{1} \frac{(1-u)^{\nu}}{\nu !}V^{\nu+1}_0 f(X_0(u t,\cdot))du\|_{m,L+\nu+1}\le C_0 \|f\|_{m+\nu+1,L}$; and for $f\in\CpolKL{2(m+\nu+1)}{L} $,
  \begin{align}\E[f(X_1(\sqrt{t}Y,x))] =& \sum_{i=0}^\nu \frac{t^i}{i!} \left(\frac{1}{2}V^2_1\right)^if(x) \label{expan_X1} \\&+ t^{\nu+1}\E\left[Y^{2\nu +2}\int_0^{1} \frac{(1-u)^{2\nu+1}}{(2\nu+1)!}V^{2\nu+2}_1 f(X_1(u \sqrt{t} Y,x))du \right],  \notag
  \end{align}
  with   $\left\|\E\left[Y^{2\nu +2}\int_0^{1} \frac{(1-u)^{2\nu+1}}{(2\nu+1)!}V^{2\nu+2}_1 f(X_1(u \sqrt{t} Y,\cdot))du \right]\right\|_{m,L+\nu+1}\le  C_1 \|f\|_{2(m+\nu+1),L}$, for some constants $C_0,C_1\in \R^+$ depending on $(a,k,\sigma)$, $T$, $m$, $M$ and $\nu$.
\end{lemma}
\begin{proof}
  Equation~\eqref{expan_X0} holds by using Taylor formula since $\frac{d}{dt}f(X_0(t,x))=V_0f(X_0(t,x))$. We have by Property~(6) of Lemma~\ref{lem_estimnorm} $\|V_0f\|_{m,L+1}\le |a-\frac {\sigma^2} 4|\|f'\|_{m,L+1}+|k|(2m+3) \|f'\|_{m,L}\le (2|a-\frac {\sigma^2} 4|+|k|(2m+3))\|f\|_{m+1,L}$ and thus $\|V_0^{\nu+1}f\|_{m,L+\nu+1}\le C \|f\|_{m+\nu+1,L}$ for some constant $C$ depending on $a,\sigma,k,\nu,m$.  Using the triangular inequality and Lemma~\ref{X0_inequalities}, we get the result.

  \blue{ We now prove the second part of the claim. We first show Equation~\eqref{expan_X1}. Since $\frac{d}{dt}f(X_1(t,x))=V_1f(X_1(t,x))$, we get by Taylor formula
    \begin{align*}
      f(X_1(t,x))&=\sum_{i=0}^{2\nu+1} \frac{t^i}{i!} V_1^if(x) +  \int_0^t \frac{(t-s)^{2\nu+1}}{(2\nu +1)!} V_1^{2\nu+2}f(X_1(s,x))du \\
      &=\sum_{i=0}^{2\nu+1} \frac{t^i}{i!} V_1^if(x) + t^{2\nu+2} \int_0^1 \frac{(1-u)^{2\nu+1}}{(2\nu +1)!} V_1^{2\nu+2}f(X_1(ut,x))du, \ t\in \R.
    \end{align*}
We apply this formula at $\sqrt{t}Y$ and take the expectation. Since $\E[Y^{2i+1}]=0$ by symmetry and $\E[Y^{2i}]=\E[N^{2i}]=\frac{(2i)!}{i!2^i}$ for $i\le \nu$, we get~\eqref{expan_X1}.  We now analyze the norm of the remainder.}   We have $\|\frac 12 V_1^2f\|_{m,L+1}\le \sigma^2 (m+2)\|f\|_{m+2,L}$ by using Lemma~\ref{lem_estimnorm}~(6). Then, we observe that by symmetry of~$Y$,
  \begin{align*}
     & \E\left[Y^{2\nu +2}\int_0^{1} \frac{(1-u)^{2\nu+1}}{(2\nu+1)!}V^{2\nu+2}_1 f(X_1(u \sqrt{t} Y,x))du \right]
    \\& = \frac 12 \E\left[Y^{2\nu +2}\int_0^{1} \frac{(1-u)^{2\nu+1}}{(2\nu+1)!}[V^{2\nu+2}_1 f(X_1(u \sqrt{t} Y,x)) + V^{2\nu+2}_1 f(X_1(-u \sqrt{t} Y,x))]du  \right]
  \end{align*}
  By Corollary~\ref{cor_psign}, we have
  \begin{align*}
    \|V^{2\nu+2}_1 f(X_1(u \sqrt{t} Y,\cdot))+V^{2\nu+2}_1 f(X_1(-u \sqrt{t} Y,\cdot))\|_{m,L+\nu+1} & \le C_{\sigma u \sqrt{t} Y,m,L} \|V^{2\nu+2}_1 f \|_{2m,L+\nu+1} \\&\le C' C_{\sigma u \sqrt{t} Y,m,L} \| f \|_{2(m+\nu+1),L},
  \end{align*}
  with $C'=(4\sigma^2 (m+\nu+1))^{\nu+1}$.
  The conclusion follows by using the triangle inequality, the polynomial growth of the constant $C_{\sigma u \sqrt{t} Y,m,L}$ given by~Corollary~\ref{cor_psign} and the finite moments $\E[Y^{2k}]$ for $k$ sufficiently large.
\end{proof}

We are now in position to prove the estimate for the approximation scheme~\eqref{Alfonsi_scheme}. Since this scheme is obtained as the composition of the schemes $X_0$ and $X_1$, the proof consists  is using iteratively the estimates of Lemma~\ref{lem_expan_V}.
\begin{prop}\label{prop_H1sch}
  Let $Y$ be a symmetric random variable such that $\E[Y^k]=\E[N^k]$ for $k\le 4$ and $\E[Y^{2k}]<\infty$ for all $k\in \N$. Let $\sigma^2\le 4a$ and $\hat{X}^x_t$ be the scheme~\eqref{Alfonsi_scheme}.  Let $m\in \N,L \in \N^*$ and $f \in \CpolKL{2(m+3)}{L}$. Then, we have for $t\in[0,T]$,
  $$ \E[f(\hat{X}^x_t)]=f(x) +t \cL f(x)+\frac{t^2}{2} \cL^2f(x) +\bar{R}f(t,x),$$
  with $\|\bar{R}f(t,\cdot)\|_{m,L+3}\le C t^3 \|f\|_{2(m+3),L}$.
\end{prop}
\begin{proof}
  We use $\hat{X}^x_t=X_0(t/2,X_1(\sqrt{t}Y,X_0(t/2,x)))$
  and apply first~\eqref{expan_X0}:
  \begin{align*}
     & \E[f(X_0(t/2,X_1(\sqrt{t}Y,X_0(t/2,x))))]=\E\left[(f+\frac{t}{2} V_0 f +\frac{t^2}{8} V_0^2 f )(X_1(\sqrt{t}Y,X_0(t/2,x)))\right] + R_If(t,x), \\
     & \text{ with } R_If(t,x) = \left(\frac t 2\right)^3\int_0^{1}\frac{(1-u)^2}{2} \E[V_0^3f(X_0(u t/2,X_1(\sqrt{t}Y,X_0(t/2,x))))]ds.
  \end{align*}
  We get by using Lemma~\ref{X0_inequalities}, Corollary~\ref{cor_psign} (using the symmetry and the finite moments of~$Y$), again Lemma~\ref{X0_inequalities} and then Lemma~\ref{lem_estimnorm}~(6):
  \begin{align*}&\|\E[V_0^3f(X_0(u t/2,X_1(\sqrt{t}Y,X_0(t/2,\cdot))))]\|_{m,L+3}\le C\|\E[V_0^3f(X_0(u t/2,X_1(\sqrt{t}Y,\cdot)))]\|_{m,L+3} \\&\le C\|V_0^3f(X_0(u t/2,\cdot))\|_{2m,L+3}\le   C\|V_0^3f\|_{2m,L+3} \le C  \|f\|_{2m+3,L}
  \end{align*}
This gives $\| R_If(t,x)\|_{m,L+3}\le  C t^3 \|f\|_{2m+3,L}$, for $t\in [0,T]$.

  We now expand again and get from~\eqref{expan_X1}
  \begin{align*} & \E\left[(f+\frac{t}{2} V_0 f +\frac{t^2}{8} V_0^2 f )(X_1(\sqrt{t}Y,X_0(t/2,x)))\right]                               \\
     & =f(X_0(t/2,x))+\frac t2 V_1^2 f(X_0(t/2,x))+ \frac {t^2}{2} (V_1^2/2)^2 f(X_0(t/2,x)) + \frac{t}{2} V_0 f(X_0(t/2,x)) \\
     & \phantom{=} + \frac{t^2}{2} (V_1^2/2)V_0 f(X_0(t/2,x)) + \frac{t^2}{8} V_0^2 f(X_0(t/2,x)) +R_{II}f(t,x),
  \end{align*}
  with
  \begin{align*}
    R_{II}f(t,x)=  t^3 \E\Bigg[&Y^6 \int_0^{u} \frac{(1-u)^{5}}{5!}V^{6}_1 f(X_1(u \sqrt{t} Y,X_0(T/2,x)))du \\
    &+  \frac {Y^4} 2 \int_0^{u} \frac{(1-u)^{3}}{ 3!}V^{4}_1V_0 f(X_1(u \sqrt{t} Y,X_0(T/2,x)))du \\
                  & + \frac{Y^2}{8} \int_0^{u}  (1-u)V^{2}_1V_0^2 f(X_1(u \sqrt{t} Y,X_0(T/2,x)))du \Bigg].
  \end{align*}
  We use Lemmas~\ref{lem_expan_V},~\ref{X0_inequalities} and~\ref{lem_estimnorm} to get, for $t\in [0,T]$,
  $$\| R_{II}f(t,\cdot)\|_{m,L+3}\le Ct^3(\|f\|_{2(m+3),L}+\|V_0f\|_{2(m+2),L+1}+\|V_0^2f\|_{2(m+1),L+2})\le Ct^3 \|f\|_{2(m+3),L}.$$
  Last, we use again~\eqref{expan_X0} to get 
  \begin{align*}
     & f(X_0(t/2,x))+\frac t2 [V_1^2 +V_0] f(X_0(t/2,x))+ \frac {t^2}{2} [(V_1^2/2)^2+(V_1^2/2)V_0+V_0^2/4] f(X_0(t/2,x))     \\
     & = f(x) +\frac t2 V_0f(x)+\frac{t^2}{8}V_0^2 f(x) + \frac t2 [V_1^2 +V_0] f(x) + \frac {t^2}{4} [V_0 V_1^2 +V_0^2] f(x) \\&\phantom{=}+ \frac {t^2}{2} [(V_1^2/2)^2+(V_1^2/2)V_0+V_0^2/4] f(x)+ R_{III}f(t,x),
  \end{align*}
  where again by Lemma~\ref{lem_expan_V}  and~\ref{lem_estimnorm}, we have
  \begin{align*}
    \|R_{III}f(t,\cdot)\|_{m,L+3} & \le Ct^3( \|f\|_{m+3,L}+\|[V_1^2 +V_0]f\|_{m+2,L+1}\\& \quad +\|[(V_1^2/2)^2+(V_1^2/2)V_0+V_0^2/4]f\|_{m+1,L+2}) \\&\le Ct^3 \|f\|_{m+5,L}.
  \end{align*}
  Finally, we get $ \|\bar{R}f(t,\cdot)\|_{m,L+3}\le Ct^3 \|f\|_{2(m+3),L}$ with $\bar{R}f:=R_{I}f+R_{II}f+R_{III}f$ and
  \begin{align*}
    \E[f(X_0(t/2,X_1(\sqrt{t}Y,X_0(t/2,x))))] & = f(x)+t[V_0+V_1^2/2]f(x)                                \\&+\frac{t^2}{2}[V_0^2+V_0V_1^2/2+(V_1^2/2)V_0+ (V_1^2/2)^2]f(x) +\bar{R}f(t,x)\\
                                              & =f(x)+t \cL f(x)+\frac{t^2}{2} \cL^2f(x) +\bar{R}f(t,x).\qedhere
  \end{align*}
\end{proof}

\subsection{Proof of~\eqref{H2_bar}}

In this section, we mainly prove the following result.
\begin{prop}\label{prop_H2_NV}
  Let $\sigma^2\le 4a$ and $\hat{X}^x_t=\varphi(x,t,\sqrt{t}N)$ be the scheme~\eqref{NV_scheme} with $N\sim \mathcal{N}(0,1)$. Let $T>0$ and  $m,L\in\N$ such that $L\ge m$. We define for $n\ge 1$ and $l\in \N$, $Q_lf(x)=\E[f(\hat{X}^x_{h_l})]$with $h_l=\frac{T}{n^l}$. Then, there exists a constant $C\in \R_+^*$ such that for any $f \in \CpolKL{m}{L}$, $l\in \N$ and $t \in [0,T]$, 
  \begin{equation}\label{H2_NV}   \left\|\E[f(X^{\cdot}_t)]\right\|_{m,L} + \max_{0\le k\le n^l} \left\|Q_l^{[k]}f \right\|_{m,L} \le C \|f\|_{m,L}.
  \end{equation}
\end{prop}
We split the proof in two parts. The first one deals with the semigroup of the CIR process, for which the assumption~$\sigma^2\le 4a$ is not needed. This is stated in Proposition~\ref{deriv_CIR_functionals}, whose proof exploits the particular form of the density of~$X^x_t$. The second part that deals with the approximation scheme is quite technical. We prove in fact in  Proposition~\ref{prop_H2_sch} a slightly more general result for the scheme  $\hat{X}^x_t=\varphi(x,t,\sqrt{t}Y)$, when $Y$ is a symmetric random variable with a smooth density. However, the conditions needed on the density are quite restrictive. These conditions are satisfied by the standard normal variable by Lemma~\ref{gaussian_eta_positivity}. If we want besides to have~\eqref{H2_NV} for any $m$ and in addition to match the moments $\E[Y^2]=\E[N^2]$ and $\E[Y^4]=\E[N^4]$ -- which is required to have a second-order scheme --, then Theorem~\ref{thm_carac_gauss} shows that we necessarily have $Y\sim \mathcal{N}(0,1)$. This is why we directly state here, for sake of simplicity, Proposition~\ref{prop_H2_NV} with $Y=N\sim \mathcal{N}(0,1)$.

\subsubsection{Upper bound for the semigroup}
We first prove the estimate~\eqref{H2_bar} for the semigroup of the CIR process. To do so, we take back the arguments of~\cite[Proposition 4.1]{AA_MCMA} that gives polynomial estimates for~$(t,x)\mapsto P_tf(x)$. First we remove the polynomial Taylor expansion of the function~$f$ at~$0$, which enables then to do an integration by parts and to get the remarkable formula in Eq.~\eqref{formula_derivative} below for the iterated derivatives of $P_tf$ that gives then the desired estimate. The polynomial part is analyzed separately in Lemma~\ref{Moments_Formula_CIR2} below with standard arguments.

\begin{prop}\label{deriv_CIR_functionals}
  Let $f\in\CpolKL{m}{L}$, $L\ge m$, $T>0$ and $t\in (0,T]$. Let  $X^x$ be the CIR process starting from $x\ge 0$. Then, $\E[f(X^{\cdot}_t)]\in\CpolKL{m}{L}$ and  we have the following estimate for some constant $C_\text{cir}(m,L,T)\in \R_+$:
  \begin{equation}\label{H0_estimate_CIR}
    \left\|\E[f(X^{\cdot}_t)]\right\|_{m,L} \le C_\text{cir}(m,L,T)\|f\|_{m,L}.
  \end{equation}
\end{prop}

\begin{proof}
  Let $f\in\CpolKL{m}{L}$ and $T_m(f)(x)=\sum_{j=0}^m \frac{f^{(j)}(0)}{j!} x^j$ its Taylor polynomial expansion at~$0$ of order $m$. We define $\hat{f}_m = f-T_m(f)\in\CpolKL{m}{L}$, so we have $f=\hat{f}_m +T_m(f)$.
  By Lemma~\ref{Moments_Formula_CIR2} below, one gets $\|\E[T_m(f)(X^{\cdot}_t)]\|_{m,L} \leq C_{cir}(m,T) \|T_m(f)\|_{m,L}$ and then
  \begin{equation}\label{estim_taylor_trunc}\|\E[T_m(f)(X^{\cdot}_t)]\|_{m,L}\leq e C_{cir}(m,T) \|f\|_{m,L},
  \end{equation}
  since for all $i\in\{0,\ldots,m\}$ and $x\geq0$
  $$ \bigg|\frac{(T_m(f))^{(i)}(x)}{1+x^L}\bigg| = \bigg|\sum_{j=0}^{m-i} \frac{f^{(i+j)}(0)}{j!} \frac{x^{j}}{1+x^L}\bigg| \leq \sum_{j=0}^{m-i} \frac{1}{j!} \|f\|_{m,L} \leq  e \|f\|_{m,L}. $$

  We now focus on~$\E[\hat{f}_m(X^{\cdot}_t)]$. We recall the density of $X^x_t$~(see e.g. \cite[Proposition 1.2.11]{AA_book})\footnote{In the case $a=0$, $X^x_t$ is distributed according to the probability measure $e^{-d_t x/2}\delta_{0}(dx)+\sum_{i=1}^\infty \frac{e^{-d_t x/2}(d_t x/2)^i}{i!} \frac{c_t/2}{\Gamma(i)}\left(\frac{c_t z}{2}\right)^{i-1}e^{-c_t z/2}$. The proof works the same since $\hat{f}_m(0)=0$, so that $\E[\hat{f}_m(X^x_t)]$ only involves the absolutely continuous part of the distribution. }
  \begin{equation}\label{CIR_density}
    p(t,x,z)=\sum_{i=0}^\infty \frac{e^{-d_t x/2}(d_t x/2)^i}{i!} \frac{c_t/2}{\Gamma(i+v)}\left(\frac{c_t z}{2}\right)^{i-1+v}e^{-c_t z/2}
  \end{equation}
  where $c_t=\frac{4k}{\sigma^2(1-e^{-kt})}$, $v=2a/\sigma^2$ and $d_t=c_te^{-kt}$. Let us remark that
  \begin{equation*}
    c_t\geq c_\text{min}:=\begin{cases} \,\,\quad\frac{4k}{\sigma^2},      & k>0  \\
      \quad\frac{4}{\sigma^2T},          & k=0  \\
      \frac{4|k|}{\sigma^2(e^{|k|T}-1)}, & k<0.
    \end{cases}
  \end{equation*}
  We have
  \begin{equation*}
    \E[\hat{f}_m(X^x_t)] = \sum_{i=0}^\infty \frac{e^{-d_t x/2}(d_t x/2)^i}{i!} I_i(\hat{f}_m,c_t),\ t>0,
  \end{equation*}
  where
  \begin{equation*}
    I_i(\hat{f}_m,c_t)=\int_0^\infty \hat{f}_m(z) \frac{c_t/2}{\Gamma(i+v)}\left(\frac{c_t z}{2}\right)^{i-1+v}e^{-c_t z/2}dz.
  \end{equation*}
  Differentiating successively, we get that for $j\le m $, $t\in(0,T]$ and $x\in\R_+$
  \begin{equation*}
    \partial^x_j\E[\hat{f}(X^x_t)]= \sum_{i=0}^\infty \frac{e^{-d_t x/2}(d_t x/2)^i}{i!} \Delta^j_t(I_i(\hat{f}_m,c_t)),
  \end{equation*}
  where $\Delta_t$ : $\R^\N \rightarrow \R^\N$ is an operator defined on sequences $(I_i)_{i\geq0}\in\R^\N$ by $\Delta_t(I_i) = \frac{d_t}{2} (I_{i+1}-I_i)= \frac{e^{-kt}}{2}c_t(I_{i+1}-I_i)$. An integration by parts gives for $i\geq 1$
  \begin{align*}
    I_i(\hat{f}^{(j)}_m,c_t) & = \int_0^\infty \hat{f}^{(j-1)}_m(z) \frac{(c_t/2)^2}{\Gamma(i+v)}\left(\frac{c_t z}{2}\right)^{i-1+v}e^{-c_t z/2}dz             \\
                             & \quad- \int_0^\infty \hat{f}_m^{(j-1)}(z) \frac{(c_t/2)^2(i-1+v)}{\Gamma(i+v)}\left(\frac{c_t z}{2}\right)^{i-2+v}e^{-c_t z/2}dz \\
                             & =\frac{c_t}{2}(I_i(\hat{f}_m^{(j-1)},c_t)-I_{i-1}(\hat{f}_m^{(j-1)},c_t))=e^{kt}\Delta_t (I_{i-1}(\hat{f}_m^{(j-1)},c_t)),
  \end{align*}
  since $\hat{f}^{(j)}_m(0)=0$ for all $1\leq j\leq m$ and $\hat{f}^{(j)}_m$has a polynomial growth. By iterating, we get for all $t\in(0,T]$ and $x\in\R_+$,
  \begin{equation}\label{formula_derivative}
    \partial^x_j\E[\hat{f}_m(X^x_t)]= \sum_{i=0}^\infty \frac{e^{-d_t x/2}(d_t x/2)^i}{i!} I_{i+j}(\hat{f}_m^{(j)},c_t) e^{-kjt}.
  \end{equation}
  Note that, since for $j\leq m $, $ |\hat{f}_m^{(j)}(z)| \leq \|\hat{f}_m\|_{m,L}(1 + z^L)$ and using the well known formula for the $L$-th raw moment of gamma distribution we have for all $i\in\N$
  \begin{equation}\label{estimate_I_f_m}
    |I_i(\hat{f}_m^{(j)},c_t)| \leq \|\hat{f}_m\|_{m,L} \left(1+ \bigg(\frac{2}{c_t}\bigg)^L \frac{\Gamma(i+L+v)}{\Gamma(i+v)}\right).
  \end{equation}
  Thus, the derivation of the series~\eqref{formula_derivative} is valid, and we get that
  $$|\partial^x_j\E[\hat{f}_m(X^x_t)]|\leq \|\hat{f}_m\|_{m,L} e^{-k jt}\left(1+ \big(\frac{2}{c_t}\big)^L \sum_{i=0}^\infty \frac{e^{-d_t x/2}(d_t x/2)^i}{i!} \frac{\Gamma(i+j+L+v)}{\Gamma(i+j+v)}\right). $$
  The quotient $\frac{\Gamma(i+j+L+v)}{\Gamma(i+j+v)}$ is a polynomial function of degree $L$ with respect to $i$, and we denote $\beta^j_0,\ldots,\beta^j_{L}$ its coefficients in the basis $\{ 1, i, i(i-1),\ldots, i(i-1)\cdots(i-L+1)\}$. Thus, we get that
  \begin{align*}
    |\partial^x_j\E[\hat{f}_m(X^x_t)]| & \leq \|\hat{f}_m\|_{m,L} e^{(-k)^+jT}\left(1+ \bigg(\frac{2}{c_t}\bigg)^L \sum_{i=0}^L |\beta^j_i| \bigg(\frac{d_t}{2}\bigg)^i x^i \right) \\
                                       & \leq \|\hat{f}_m\|_{m,L} e^{(-k)^+ mT}\left(1+ \sum_{i=0}^L |\beta^j_i| \bigg(\frac{2}{c_t}\bigg)^{L-i} e^{-kit} (1+x^L) \right)           \\
                                       & \leq \|\hat{f}_m\|_{m,L} e^{(-k)^+(m+L)T}\left(1+ \sum_{i=0}^L |\beta^j_i| \bigg(\frac{2}{c_{\text{min}}}\bigg)^{L-i}   \right)(1+x^L).
  \end{align*}
  By the triangular inequality and~\eqref{estim_taylor_trunc}, we get $\|\hat{f}_m\|_{m,L}\leq(1+e)\|f\|_{m,L}$, so one has for all $t\in (0,T]$, $j\le m$
  \begin{equation}
    |\partial^x_j\E[\hat{f}_m(X^x_t)]|  \leq (1+e) e^{(-k)^+(m+L)T} \left(1+ \sum_{i=0}^L |\beta^j_i| \bigg(\frac{2}{c_{\text{min}}}\bigg)^{L-i}   \right)\|f\|_{m,L}(1+x^L),
  \end{equation}
  and thus for all $t\in(0,T]$:
  \begin{equation}
    \|\E[\hat{f}_m(X^{\cdot}_t)]\|_{m,L} \leq \hat{C}\|f\|_{m,L},
  \end{equation}
  where $\hat{C}:=  (1+e) e^{(-k)^+(m+L)T} \max_{0\le j\le m}\left(1+ \sum_{i=0}^L |\beta^j_i| \bigg(\frac{2}{c_{\text{min}}}\bigg)^{L-i}   \right)$.
  Finally, we get the desired estimate by the triangular inequality,~\eqref{estim_taylor_trunc} and Lemma~\ref{Moments_Formula_CIR2}:
  \begin{equation*}
    \left\|\E[f(X^{\cdot}_t)]\right\|_{m,L} \leq  \|\E[\hat{f}_m (X^{\cdot}_t)]\|_{m,L} + \|\E[T_m(f) (X^{\cdot}_t)]\|_{m,L} \leq (\hat{C}+C_{cir}(m,T)) \| f\|_{m,L}. \qedhere
  \end{equation*}
\end{proof}

\begin{lemma}\label{Moments_Formula_CIR2}
  Let $P\in \PLRp{m}$ be  a polynomial function of degree $m\in\N^*$ and $L\in\N^*$ such that $L\geq m$. Then, for $t\in [0,T]$ we have the following estimate
  \begin{equation}\label{Pol_Estimate_CIR_2}
    \|\E[P(X^{\cdot}_t)]\|_{m,L} \leq C_{\text{cir}}(m,T) \|P\|_{m,L},
  \end{equation}
  where $C_{\text{cir}}(m,T) = \max_{t\in[0,T]} \sum_{j=0}^m \sum_{i=0}^j |\tilde{u}_{i,j}(t)|$ \blue{ with $\tilde{u}_{i,j}(t)$ defined as in Lemma~\ref{Moments_Formula_CIR} by $\E[(X^x_t)^j]=\sum_{i=0}^j\tilde{u}_{i,j}(t)x^i$}.
\end{lemma}
\begin{proof}
  We consider a polynomial function $P(y)=\sum_{i=0}^m b_i y^i$ of degree $m$ and $L\geq m$. For all $l\in\{0,\ldots,m\}$ one has from Lemma~\ref{Moments_Formula_CIR}
  \begin{align*}
    \frac{|\partial_x^l\E[P(X_t^x)]|}{1+x^L} & =\bigg| \sum_{j=0}^m b_j \frac{\partial_x^l\tilde{u}_j(t,x)}{1+x^L} \bigg|  \leq \bigg| \sum_{j=0}^m b_j \sum_{i=l}^j\tilde{u}_{i,j}(t)\frac{i!}{(i-l)!}\frac{x^{i-l}}{1+x^L} \bigg| \\
                                             & \leq  \sum_{j=0}^m |b_j| j! \sum_{i=l}^j |\tilde{u}_{i,j}(t)|
    \leq \max_{t\in[0,T]} \sum_{j=0}^m \sum_{i=l}^j |\tilde{u}_{i,j}(t)| \max_{j\in\{0,\ldots,L\}}|b_j| j!,
  \end{align*}
  passing to supremum over $x\geq 0$, $l\in\{0,\ldots,m\}$ we get \eqref{Pol_Estimate_CIR_2} observing that $|b_j| j!=|P^{(j)}(0)|\le \|P\|_{m,L}$.
\end{proof}

\subsubsection{Upper bound for the approximation scheme}
We now prove the estimate~\eqref{H2_bar} for the approximation of the CIR process. The main result of this paragraph is the following.
\begin{prop}\label{prop_H2_sch}
  Let $T>0, \sigma^2\le 4a$, $m,M \in \N$,  $Y$ be a symmetric random variable with density $\eta\in \mathcal{C}^M(\R)$ such that for all $i\in\{0,\ldots,M\}$, $|\eta^{(i)}(y)|=o(|y|^{-(2L+i)})$ for $|y|\rightarrow \infty$,  and $\eta^*_m\ge 0$ for all $1 \le m\le M$ (see Lemma~\ref{regular_density} below for the definition of $\eta^*_m$).
  Let $Q_lf(x)=\E[f(\hat{X}^x_{h_l})]$ with $\hat{X}^x_t=\varphi(t,x,\sqrt{t}Y)$, $n\ge 1$, $l\in \N$ and $h_l=T/n^l$. Then, for any $L \in \N$, there exists $C\in \R_+$ such that: 
  $$  \max_{0\le j \le n^l}\|Q_l^{[j]}f\|_{m,L} \le C \|f\|_{m,L}, \ f\in \CpolKL{m}{L}, l \in \N.$$
\end{prop}
Note that by Lemma~\ref{gaussian_eta_positivity} below, the assumptions of Proposition~\ref{prop_H2_sch} are satisfied by~$Y\sim \mathcal{N}(0,1)$. Therefore,~\eqref{H2_bar} holds for  the scheme of Ninomiya and Victoir~\eqref{NV_scheme}.
\begin{proof}
  We have $\hat{X}^x_t=\varphi(t,x,\sqrt{t}Y)=X_0(t/2,X_1(\sqrt{t}Y,X_0(t/2,x)))$. Let $f \in \CpolKL{m}{L}$. We apply Lemma~\ref{X0_inequalities} and Lemma~\ref{regular_density} below to get:
  \begin{align*}  \|\E[f(X_0(t/2,X_1(\sqrt{t}Y,X_0(t/2,\cdot ))))]\|_{m,L} & \le e^{Kt/2} \|\E[f(X_1(\sqrt{t}Y,X_0(t/2,\cdot )))]\|_{m,L} \\&\le e^{Kt/2+Ct} \|f(X_0(t/2,\cdot ))\|_{m,L}\le e^{(C+K)t} \|f\|_{m,L}
  \end{align*}
  This gives $\max_{0\le j \le n^l}\|Q_l^{[j]}f\|_{m,L} \le e^{(C+K)T} \|f\|_{m,L}$.
\end{proof}

\begin{lemma}\label{regular_density}
  Let $M,L\in \N$. Let $Y$ be a symmetric random variable with density $\eta\in \mathcal{C}^M(\R)$ such that for all $i\in\{0,\ldots,M\}$, $|\eta^{(i)}(y)|=o(|y|^{-(2L+i)})$ for $|y|\rightarrow \infty$. Then, for all function $f\in\CpolKL{M}{L}$, $m\in\{1, \ldots,  M\}$ and $t\in [0,T]$ one has the following representation
  \begin{equation}\label{repres_X1_functional}
    \partial^m_x\E[f(X_1(\sqrt{t}Y,x))] = \int_{-\infty}^\infty \int_0^1 (u-u^2)^{m-1} f^{(m)}(w(u,x,y)) \eta^*_m(y) dudy
  \end{equation}
  where  $w(u,x,y)=x+(2u-1)\sigma\sqrt{t}y\sqrt{x}+ \sigma^2ty^2/4$, $\eta^*_m(y)=(-1)^{m-1}  \left(\sum_{j=1}^m c_{j,m}  y^j\eta^{(j)}(y)\right)$, and the coefficients $c_{j,m}$  are defined by induction, starting from $c_{1,1}=-1$, through the following formula
  \begin{equation}\label{recursive_coeff_formula}
    c_{j,m}   = \bigg(\frac{2j}{m-1}-4\bigg)  c_{j,m-1}\mathds{1}_{j<m} + \frac{2}{m-1}  c_{j-1,m-1}\mathds{1}_{j>1},     \  j\in\{1, \ldots, m\}, \, m\in\{2, \ldots,M\}.  \\
  \end{equation}
 \blue{In particular, $c_{m,m}=-\frac{2^{m-1}}{(m-1)!}<0$.} Furthermore, if the density $\eta$ is such that  $\eta^*_{m}(y)\geq0$ for all $y\in\R$, and all $ m\in\{1, \ldots,  M\}$, then there exists $C\in \R_+$ such that
  \begin{equation}\label{estimate_X1_functional}
    	\|\E[f(X_1(\sqrt{t}Y,\cdot))]\|_{m,L} \leq (1+Ct) \|f\|_{m,L},\ t\in [0,T].
  \end{equation}
\end{lemma}
Let us stress here two things that are crucial in~\eqref{estimate_X1_functional}: the same norm is used in both sides, and the sharp time dependence of the multiplicative constant $(1+Ct)$. These properties are used in the proof of Proposition~\ref{prop_H2_sch} to get~\eqref{H2_bar}.

\begin{proof}
  We first consider $m=1$ and $f\in\CpolKL{M}{L}$. From the symmetry of $Y$, we have the equality $\E[f(X_1(\sqrt{t}Y,x))] = \E[f(X_1(\sqrt{t}Y,x))+ f(X_1(-\sqrt{t}Y,x))]/2$ and using the notation $\psi^{\pm}_f(x,y) = f(x+\sigma\sqrt{t}y\sqrt{x}+ \sigma^2ty^2/4) \pm f(x-\sigma\sqrt{t}y\sqrt{x}+ \sigma^2ty^2/4)$ we can write,
  \begin{equation*}
    \partial_x\E[f(X_1(\sqrt{t}Y,x))] = \frac{1}{2} \int_{-\infty}^\infty \partial_x\psi^+_f(x,y) \eta(y)dy.
  \end{equation*}
  One derivation and a little of algebra show that
  \begin{align*}
    \partial_x \psi^+_f(x,y) & = (1+ \frac{\sigma\sqrt{t}y}{2\sqrt{x}}) f'(x+\sigma\sqrt{t}y\sqrt{x}+ \sigma^2ty^2/4)+ (1- \frac{\sigma\sqrt{t}y}{2\sqrt{x}}) f'(x-\sigma\sqrt{t}y\sqrt{x}+ \sigma^2ty^2/4) \\
                             & = \frac{1}{\sigma\sqrt{t}\sqrt{x}}
    \begin{multlined}[t]
      \Big((\sigma^2 t y/2+ \sigma\sqrt{t}\sqrt{x}) f'(x+\sigma\sqrt{t}y\sqrt{x}+ \sigma^2ty^2/4) \\
      - (\sigma^2 t y/2- \sigma\sqrt{t}\sqrt{x}) f'(x-\sigma\sqrt{t}y\sqrt{x}+ \sigma^2ty^2/4)\Big)
    \end{multlined}                                                                                                                                                                              \\
                             & = \frac{1}{\sigma\sqrt{t}\sqrt{x}} \Big(\partial_y [f(x+\sigma\sqrt{t}y\sqrt{x}+ \sigma^2ty^2/4)] -  \partial_y [f(x-\sigma\sqrt{t}y\sqrt{x}+ \sigma^2ty^2/4)] \Big)         \\
                             & = \frac{\partial_y \psi^-_f(x,y)}{\sigma\sqrt{t}\sqrt{x}}.
  \end{align*}
  Integrating by parts in the variable $y$, observing that the boundary term vanishes since $|\eta(y)|=_{|y|\to \infty}o(|y|^{-2L})$ and $f(z)=_{z\to \infty}O(z^L)$, one has
  \begin{align*}
    \partial_x\E[f(X_1(\sqrt{t}Y,x))] & = - \frac{1}{2} \int_{-\infty}^\infty \frac{ \psi^-_f(x,y) \eta'(y)}{\sigma\sqrt{t}\sqrt{x}} dy            \\
                                      & = -  \int_{-\infty}^\infty \int_0^1 f'(x+(2u-1)\sigma\sqrt{t}y\sqrt{x}+ \sigma^2ty^2/4)  \eta'(y)y \,dudy  \\
                                      & =  \int_{-\infty}^\infty \int_0^1 f'(x+(2u-1)\sigma\sqrt{t}y\sqrt{x}+ \sigma^2ty^2/4)  (-\eta'(y)y) \,dudy
  \end{align*}
  since $\partial_u f(x+(2u-1)\sigma\sqrt{t}y\sqrt{x}+ \sigma^2ty^2/4) = 2\sigma\sqrt{t}y\sqrt{x} f'(x+(2u-1)\sigma\sqrt{t}y\sqrt{x}+ \sigma^2ty^2/4) $.
  In order to simplify the notation, we define $w(u,x,y) := x+(2u-1)\sigma\sqrt{t}y\sqrt{x}+ \sigma^2ty^2/4$ and we write explicitly the partial derivatives of $w$
  \begin{equation}\label{deriv_w}
    \begin{cases}
      \partial_u w(u,x,y) = 2\sigma\sqrt{t}y\sqrt{x},                  \\
      \partial_x w(u,x,y) = 1+\frac{(2u-1)\sigma\sqrt{t}y}{2\sqrt{x}}, \\
      \partial_y w(u,x,y) = (2u-1)\sigma\sqrt{t}\sqrt{x}+\frac{\sigma^2 t y}{2},
    \end{cases}
  \end{equation}
  and we define  for $s:[0,1]\times \R \rightarrow \R$
  \begin{equation}
    I^{(l)}_{m,n}(s)= \int_{-\infty}^\infty \int_0^1 s(u,y) (u^2-u)^{m-1}f^{(l)}(w(u,x,y))  \bigg(\sum_{j=1}^n c_{j,n}  y^j\eta^{(j)}(y)\bigg) \,dudy,
  \end{equation}
  so we can rewrite \eqref{repres_X1_functional} as $\partial^m_x\E[f(X_1(\sqrt{t}Y,x))] = I^{(m)}_{m,m}(1)$ where the 1 in the argument has to be intended as the constant map identically equal to 1. So far, we have shown that formula \eqref{repres_X1_functional} is true for $m=1$, we take now $m\geq2$ and we prove it by induction over $m$ assuming that the result holds for $m-1$. We differentiate Eq.~\eqref{repres_X1_functional}
  for $m-1$ and use the second equality of \eqref{deriv_w} to get
  \begin{equation}
    \partial^m_x\E[f(X_1(\sqrt{t}Y,x))]  =I^{(m)}_{m-1,m-1}(1) + I^{(m)}_{m-1,m-1}\left(\frac{(2u-1)\sigma\sqrt{t}y}{2\sqrt{x}}\right).
  \end{equation}
  Then, from the third equality of \eqref{deriv_w}, one has $\frac{(2u-1)\sigma\sqrt{t}y}{2\sqrt{x}}= \frac{(2u-1)}{\sigma\sqrt{t}\sqrt{x}}\partial_yw - (2u-1)^2$ and so
  \begin{align*}
    \partial^m_x\E[f(X_1(\sqrt{t}Y,x))] & =I^{(m)}_{m-1,m-1}(1- (2u-1)^2) + I^{(m)}_{m-1,m-1}\left(\frac{(2u-1)}{\sigma\sqrt{t}\sqrt{x}} \partial_y w(u,x,y)\right) \\
                                        & =-4 I^{(m)}_{m-1,m-1}(u^2-u) + I^{(m)}_{m-1,m-1}\left(\frac{(2u-1)}{\sigma\sqrt{t}\sqrt{x}} \partial_y w(u,x,y)\right)    \\
                                        & =-4 I^{(m)}_{m,m-1}(1) + I^{(m)}_{m-1,m-1}\left(\frac{(2u-1)}{\sigma\sqrt{t}\sqrt{x}} \partial_y w(u,x,y)\right).
  \end{align*}
  We work on the term $I^{(m)}_{m-1,m-1}(\frac{(2u-1)}{\sigma\sqrt{t}\sqrt{x}} \partial_y w(u,x,y))$. We use first an integration by parts in the variable $y$ and subsequently one in the variable $u$.  The boundary terms vanishes by using the hypothesis on $\eta$ since $|f^{(m)}(w(u,x,y))|\le \|f\|_{m,L}(1+w(u,x,y)^L)\underset{|y|\to \infty }=O(y^{2L})$ and to the fact that the function $u^2-u$ vanishes in $0$ and $1$. One gets
    {\small	\begin{multline}
        \int_0^1\int_{-\infty}^\infty   \frac{(2u-1)(u^2-u)^{m-2}}{\sigma\sqrt{t}\sqrt{x}}f^{(m)}(w(u,x,y))\partial_y w(u,x,y)\bigg(  \sum_{j=1}^{m-1} c_{j,m-1}  y^j\eta^{(j)}(y) \bigg)dydu \\
        \begin{aligned}
           & =-\int_{-\infty}^\infty \int_0^1  \frac{(2u-1)(u^2-u)^{m-2}}{\sigma\sqrt{t}\sqrt{x}}f^{(m-1)}(w(u,x,y))  \bigg(\sum_{j=1}^{m-1} c_{j,m-1}  (jy^{j-1}\eta^{(j)}(y)+y^{j}\eta^{(j+1)}(y)) \bigg)dudy  \\
           & =\int_{-\infty}^\infty \int_0^1 (u^2-u)^{m-1} f^{(m)}(w(u,x,y)) \bigg( \sum_{j=1}^{m-1} \frac{2}{m-1}c_{j,m-1}  (jy^{j}\eta^{(j)}(y)+y^{j+1}\eta^{(j+1)}(y)\bigg) \,dudy                            \\
           & =\int_{-\infty}^\infty \int_0^1 (u^2-u)^{m-1} f^{(m)}(w(u,x,y))  \frac{2}{m-1}\bigg(\sum_{j=1}^{m-1} jc_{j,m-1}  y^{j}\eta^{(j)}(y) +\sum_{j=2}^{m} c_{j-1,m-1}y^{j}\eta^{(j)}(y)\bigg) \bigg)dudy.
        \end{aligned}
      \end{multline}}
  Rewriting the last equality for $\partial^m_x\E[f(X_1(\sqrt{t}Y,x))]$, one has
    {\small \begin{multline}
        \partial^m_x\E[f(X_1(\sqrt{t}Y,x))] = \int_{-\infty}^\infty \int_0^1 (u^2-u)^{m-1} f^{(m)}(w(u,x,y))  \bigg(-4\sum_{j=1}^{m-1} c_{j,m-1}  y^{j}\eta^{(j)}(y)\bigg)\,dudy +\\
        + \int_{-\infty}^\infty \int_0^1 (u^2-u)^{m-1} f^{(m)}(w(u,x,y))  \frac{2}{m-1}\bigg(\sum_{j=1}^{m-1} jc_{j,m-1}  y^{j}\eta^{(j)}(y) +\sum_{j=2}^{m} c_{j-1,m-1}y^{j}\eta^{(j)}(y)\bigg) \,dudy\\
        = \begin{aligned}[t]
           & \int_{-\infty}^\infty \int_0^1 (u^2-u)^{m-1} f^{(m)}(w(u,x,y))\bigg( \big(\frac{2}{m-1} -4\big)c_{1,m-1} y\eta^{(1)}(y)+                                                 \\
           & \sum_{j=2}^{m-1}\Big(\big(\frac{2j}{m-1} -4\big)c_{j,m-1} +\frac{2}{m-1}c_{j-1,m-1}\Big) y^{j}\eta^{(j)}(y) + \frac{2}{m-1} c_{m-1,m-1} y^{m}\eta^{(m)}(y)\bigg) \,dudy,
        \end{aligned}
      \end{multline}}
  which proves the representation \eqref{repres_X1_functional}. \blue{Since  $c_{1,1}=-1$ and $c_{m,m}=-\frac{2}{m-1} c_{m-1,m-1}$ for $m\ge 2$, we get $c_{m,m}=-\frac{2^{m-1}}{(m-1)!}$ for $m\ge 1$.}

  We are now able to prove the estimate using this representation. Defining $\eta^*_m(y) = (-1)^{m-1} \sum_{j=0}^m c_{j,m}y^j\eta{(j)}(y)$, that is nonnegative for all $y$ by hypothesis, one has
  \begin{align*}
    |\partial^m_x\E[f(X_1(\sqrt{t}Y,x))]| & \leq \int_0^1 (u-u^2)^{m-1}  \int_{-\infty}^\infty   |f^{(m)}(w(u,x,y))| \eta^*_m(y)dydu                     \\
                                          & \leq \|f\|_{m,L}\int_0^1 (u-u^2)^{m-1}  \int_{-\infty}^\infty (1+w(u,x,y)^L) \eta^*_m(y)dydu                 \\
                                          & =  \|f\|_{m,L}\underbrace{\int_0^1 (u-u^2)^{m-1}  \int_{-\infty}^\infty  \eta^*_m(y)dydu}_{A}                \\
                                          & \quad+ \|f\|_{m,L}\underbrace{\int_0^1 (u-u^2)^{m-1}  \int_{-\infty}^\infty w(u,x,y)^L \eta^*_m(y)dydu}_{B}.
  \end{align*}
  The double integral $A$ can be seen by means of representation \eqref{repres_X1_functional} with $f(x)=\frac{x^m}{m!}$ ($f^{(m)}\equiv 1$) as
  \begin{align*}
    A = \partial^m_x\E\left[\frac{X_1(\sqrt{t}Y,x)^m}{m!}\right] & =\frac{1}{m!} \partial^m_x\sum_{j=0}^{m}{2m\choose 2j} x^{m-j}\left(\frac{\sigma\sqrt{t}}{2}\right)^j \E[Y^{2j}]=1,
  \end{align*}
  by using  the symmetry of the density $\eta$.	  In the same way, $B$ can be seen by means of the representation as
  \begin{align*}
    B & = \partial^m_x\E\left[\frac{L!}{(L+m)!}X_1(\sqrt{t}Y,x)^{L+m}\right]                                                               \\
      & = \partial^m_x\sum_{j=0}^{L+m}{2(L+m)\choose 2j} \frac{L!x^{L+m-j}}{(L+m)!}\left(\frac{\sigma\sqrt{t}}{2}\right)^{2j} \E[Y^{2j}]   \\
      & = \sum_{j=0}^{L}{2(L+m)\choose 2j} \frac{L!(L+m-j)!}{(L-j)!(L+m)!}x^{L-j}\left(\frac{\sigma\sqrt{t}}{2}\right)^{2j} \E[Y^{2j}]     \\
      & =x^L +t\sum_{j=1}^{L}{2(L+m)\choose 2j} \frac{L!(L+m-j)!}{(L-j)!(L+m)!}x^{L-j}\left(\frac{\sigma}{2}\right)^{2j}t^{j-1} \E[Y^{2j}] \\
      & \leq x^L +t(1+x^L)(1+\E[Y^{2L}])\sum_{j=1}^{L}{2(L+m)\choose 2j} \left(\frac{\sigma c_T}{2}\right)^{2j}                            \\
      & \leq x^L +Ct(1+x^L)
  \end{align*}
  where $c_T= \max(1,T)$ and $C = \frac{1}{2}\left((1+\frac{\sigma c_T}{2})^{2(L+k)}+(1-\frac{\sigma c_T}{2})^{2(L+k)}\right)(1+\E[Y^{2L}])$. Putting parts $A$ and $B$ back together one has
  \begin{equation}
    \partial^m_x\E[f(X_1(\sqrt{t}Y,x))] \leq \|f\|_{m,L} (1 +x^L + Ct(1+x^L)) = (1 +x^L)(1 +Ct)\|f\|_{m,L}  ,
  \end{equation}
  and this proves the desired norm inequality.
\end{proof}

\begin{lemma}\label{gaussian_eta_positivity}
  Let $\eta(y)=\frac 1 {\sqrt{2\pi}} e^{-y^2/2}$ be the density of a standard normal variable. Then, we have for $m\ge 1$:
  \begin{equation}\label{gauss_identity}
    \eta^*_m(y):= (-1)^{m-1}\sum_{j=1}^m  c_{j,m} y^j\eta^{(j)}(y) = -c_{m,m}y^{2m}\eta(y),
  \end{equation}
  so, in particular $\eta^*_m(y)\ge 0$ for all $y\in\R$.
\end{lemma}

\begin{proof}
  For $m=1$, \eqref{gauss_identity} is clearly true since $\eta'(y)=-y\eta(y)$.
  We now take $m\geq2$, $M\geq m$ and we suppose \eqref{gauss_identity} true for $m-1$:  for all $f\in \CpolKL{M}{L}$ and $x\in\R_+$, we have
  \begin{equation*}
    \int_{-\infty}^\infty \int_0^1 (u-u^2)^{m-2} f^{(m-1)}(w(u,x,y)) \big(\eta^*_{m-1}(y)+c_{m-1,m-1}y^{2m-2}\eta(y)\big)dudy=0.
  \end{equation*}
  Doing one differentiation step with respect to~$x$ like in the proof of Lemma~\ref{regular_density} and using that $\eta'(y)=-y\eta(y)$, we obtain
  \begin{equation*}
    \int_{-\infty}^\infty \int_0^1 (u-u^2)^{m-1} f^{(m)}(w(u,x,y)) \big(\eta^*_m(y)+c_{m,m}y^{2m}\eta(y)\big)dudy=0.
  \end{equation*}
  By choosing $f_L(x):= \frac{L!}{(L+m)!}x^{L+m}$ for $L\in \N$, we get for all $L\in \N$, $x \in \R_+$,
  \begin{equation*}
      \int_{-\infty}^\infty \int_0^1 (u-u^2)^{m-1} w(u,x,y)^L \big(\eta^*_m(y)+c_{m,m}y^{2m}\eta(y)\big)dudy=0.
  \end{equation*}
  We now take $x=0$ so that $w(u,0,y)=\frac{\sigma^2t}4 y^2$ and then
  $$ \int_{-\infty}^\infty y^{2L} \big(\eta^*_m(y)+c_{m,m}y^{2m}\eta(y)\big)dy=0, \ L\in \N.$$ We remark also that
  $\eta^*_m(y)=\big(\sum_{j=1}^m (-1)^{m+j-1}c_{j,m}y^jH_j(y)\big)\eta(y)$, where $H_j$ is the $j^{\text{th}}$ Hermite polynomial function (defined by $\eta^{(j)}(y)=(-1)^j H_j(y)\eta(y)$). Thus, $\eta^*_m(y)+c_{m,m}y^{2m}\eta(y)=P_{2m}(y)\eta(y)$ where $P_{2m}$ is  an even polynomial function of degree $2m$. We therefore obtain $\int_{-\infty}^\infty  y^{l} P_{2m}(y) \eta(y) dy=0$ for all $l\in \N$, which gives $P_{2m}=0$ and thus the claim.
\end{proof}

\begin{remark}
  Lemma~\ref{gaussian_eta_positivity} gives a remarkable formula of the monomial of order $2m$ $m\in\N^*$ in terms of the first $m$ Hermite polynomials multiplied respectively by the first $m$ monomials
  \begin{equation}\label{Hermite_inv_formula}
    y^{2m} = \sum_{j=1}^m (-1)^{m+j}\frac{c_{j,m}}{c_{m,m}}y^jH_j(y).
  \end{equation}
\end{remark}

The next result gives a kind of reciprocal result to Lemma~\ref{gaussian_eta_positivity}. It explains why we consider a normal random variable for~$Y$ in Theorem~\ref{thm_main}, since we use Proposition~\ref{prop_H2_sch} for any $M\in \N$.
\begin{theorem}\label{thm_carac_gauss}
  Let $Y$ be a symmetric random variable with a $\mathcal{C}^\infty$ probability density function~$\eta$ such that $\E[Y^2]=1$, $\E[Y^4]=3$ and $\eta_m^*\ge 0$ for all $m\ge 1$. Then, $Y\sim \mathcal{N}(0,1)$.
\end{theorem}
\begin{proof}
  By Corollary~\ref{cor_density_etam}, there exists a positive Borel measure~$\mu$ such that
  $\eta(x)=\int_0^\infty e^{-tx^2}\mu(dt)$. Since $\int_{\R}\eta=1$, we get $\int_0^\infty \sqrt{\pi/t}\mu(dt)=1$ and then   $\eta(x)=\int_0^\infty \frac{e^{-tx^2}}{\sqrt{\pi/t}} \tilde{\mu}(dt)$ with $\tilde{\mu}(dt)=\sqrt{\pi/t} \mu(dt)$ being a probability measure on~$\R_+$. We have $\E[Y^2]=\int_0^\infty \int_{\R} x^2\frac{e^{-tx^2}}{\sqrt{\pi/t}} dx \tilde{\mu}(dt)=\int_0^\infty \frac{1}{2t} \tilde{\mu}(dt)$ and $\E[Y^4]=\int_0^\infty 3\left(\frac{1}{2t}\right)^2 \tilde{\mu}(dt)$. Therefore, we have
  $$\int_0^\infty \frac{1}{2t} \tilde{\mu}(dt)=\int_0^\infty \left(\frac{1}{2t}\right)^2 \tilde{\mu}(dt) =1.$$
  The equality condition in the Cauchy-Schwarz inequality implies that $ \tilde{\mu}(dt) =\delta_{1/2}(dt)$, i.e. $Y$ is a standard normal variable.
\end{proof}

\subsection{Proof of Theorem~\ref{thm_main}}\label{Subsec_thm_main}
We prove the result for $\hat{P}^{2,n}$. By assumption, $f\in  \CpolKL{18}{L}$, for $L\ge 18$ sufficiently large. From~\eqref{devt_erreur}, we have
\begin{align*}
  P_Tf-\hat{P}^{2,n}f & =  \sum_{k=0}^{n-1}Q_1^{[n-(k+1)]}[P_{h_1}-Q_2^{[n]}]Q_1^{[k]}f \\ &\quad+\sum_{k=0}^{n-1}\sum_{k'=0}^{n-(k+2)}P_{(n-(k+k'+2))h_1}[P_{h_1}-Q_1] Q_1^{[k']}[P_{h_1}-Q_1] Q_1^{[k]},
\end{align*}
with $h_l=T/n^l$.
Using Proposition~\ref{prop_H2_NV} three times and Proposition~\ref{H1_bar_mL} twice, we get for $k\in \{0,\dots,n-1\},k' \in \{0,\dots,n-(k+2)\}$:
\begin{align*}
  \|P_{(n-(k+k'+2))h_1}[P_{h_1}-Q_1] Q_1^{[k']}[P_{h_1}-Q_1] Q_1^{[k]} f\|_{0,L+6} & \le C\|[P_{h_1}-Q_1] Q_1^{[k']}[P_{h_1}-Q_1]  Q_1^{[k]}f\|_{0,L+6} \\
                                                                                   & \le C h_1^3\| Q_1^{[k']}[P_{h_1}-Q_1] Q_1^{[k]}f\|_{6,L+3}         \\
                                                                                   & \le  C h_1^3\|[P_{h_1}-Q_1] Q_1^{[k]}f \|_{6,L+3}                  \\
                                                                                   & \le Ch_1^6\|Q_1^{[k]}f \|_{18,L} \le Ch_1^6\|f \|_{18,L}.
\end{align*}
For the other term, we write $P_{h_1}-Q_2^{[n]}=\sum_{k'=0}^{n-1}P_{(n-(k'+1))h_2}[P_{h_2}-Q_2]Q_2^{[k']}$ and get for $k,k'\in \{0,\dots,n-1\}$ by using  Proposition~\ref{prop_H2_NV}, Proposition~\ref{H1_bar_mL} and Lemma~\ref{lem_estimnorm}:
\begin{align*}
  \|Q_1^{[n-(k+1)]}P_{(n-(k'+1))h_2}[P_{h_2}-Q_2]Q_2^{[k']}Q_1^{[k]}f\|_{0,L+6} & \le C\|[P_{h_2}-Q_2]Q_2^{[k']}Q_1^{[k]}f\|_{0,L+6}   \\
                                                                                & \le C h_2^3 \|Q_2^{[k']}Q_1^{[k]}f\|_{6,L+3}         \\
                                                                                & \le  C h_2^3 \|f\|_{6,L+3}\le  C h_2^3 \|f\|_{18,L}.
\end{align*}
This gives
$$\|P_Tf-\hat{P}^{2,n}f\|_{0,L+6}\le C \|f\|_{18,L} n^2(h_1^6+h_2^3) \le C \|f\|_{18,L} n^{-4},$$
and in particular $P_Tf(x)-\hat{P}^{2,n}f(x)=O(n^{-4})$ for any $x\ge 0$.

We now consider $f\in \mathcal{C}^\infty$ with derivatives of polynomial growth. Therefore, for any $m\in \N$, it exists $L \ge m$ sufficiently large, such that $f\in  \CpolKL{m}{L}$. We can then apply \cite[Theorem~3.10]{AB} to get that for some functions ${\bf m},\ell:\N^*\to \N^*$, we have $\|P_Tf-\hat{P}^{\nu,n}f\|_{0,L+\ell(\nu)}\le C\|f\|_{{\bf m}(\nu),L}n^{-2\nu}$ \blue{for $L\ge {\bf m}(\nu)$}, which gives the claim.

%%%%%%%%%%%%%%%%%%%%%%%%%%%%%%%%%%%%%%%%%%%%%%%%%%%%%%%%%%%%%%%%%%%%%%%%

\section{Simulations results}\label{Simulations}
In order to present some numerical test, we  first explain how to implement the approximations $\hat{P}^{2,n}$ and $\hat{P}^{3,n}$ defined respectively by~\eqref{def_P_2} and~\eqref{def_P_3} (let us recall here that $\hat{P}^{1,n}$ is the approximation obtained on the regular time grid $\Pi^0 = \{kT/n, 0 \leq k \leq n\}$). We consider a general case of a scheme that can be written as a function of the starting point, the time step, the Brownian increment and an independent random variable, i.e.
$$Q_lf(x)=\E[\varphi(x,h_l,W_{h_l},V)].$$
The second order scheme for the CIR~\eqref{NV_scheme}  falls into this framework as well as the second order scheme for the Heston model~\eqref{H2S} that we introduce below. As illustrated in \cite{AB} the approximation~$\hat{P}^{2,n}$ is the simplest case for the implementation. It consists in the simulation of two starting schemes on the uniform time grid $\Pi^0$ and on the random grid : $\Pi^1 = \Pi^0 \cup  \{ \kappa T/n + k'T/n^2 , 1 \leq k' \leq n-1 \}$, where $\kappa$ is an independent uniform random variable on $\{0,\ldots,n-1\}$. We denote by $\hat{X}^{n,0}$ the scheme on $\Pi^0$
\begin{align}
  \hat{X}^{n,0}_0          & = x, \notag                                                                                \\
  \hat{X}^{n,0}_{(k+1)h_1} & = \varphi(\hat{X}^{n,0}_{kh_1}, h_1, W_{(k+1)h_1} - W_{kh_1}, V_k ),\quad 0\leq k\leq n-1, \label{Scheme0}
\end{align}
and by  $\hat{X}^{n,1}$ the scheme on $\Pi^1$:
\begin{align*}
  \hat{X}^{n,1}_{kh_1}                 & = \hat{X}^{n,0}_{kh_1},                                                                            & \quad   0\leq k \leq \kappa,   \\
  \hat{X}^{n,1}_{\kappa h_1+(k'+1)h_2} & = \varphi(\hat{X}^{n,1}_{\kappa h_1+k'h_2}, h_2, W_{\kappa h_1+(k'+1)h_2} - W_{\kappa h_1+k'h_2},V_{n+k'}), & \quad 0 \leq k' \leq n-1,      \\
  \hat{X}^{n,1}_{(k+1)h_1}             & = \varphi(\hat{X}^{n,1}_{kh_1}, h_1, W_{(k+1)h_1} - W_{kh_1},V_k),                                     & \quad \kappa+1 \leq k\leq n-1.
\end{align*}
Here, $(V_k)_{k\ge 0}$ is an i.i.d. sequence with the same law as $V$. Finally, we can give the following probabilistic representation
\begin{align}
  \hat{P}^{2,n}f & = Q^{[n]}_1f+n\E[Q_1^{[n-(\kappa+1)]}[Q_2^{[n]}-Q_1]Q_1^{[\kappa]}f] \nonumber            \\
           & =\E[f(\hat{X}^{n,0}_T)] + n\E[f(\hat{X}^{n,1}_T) - f(\hat{X}^{n,0}_T)]. \label{P2n_indep} 
\end{align}
\blue{Let us stress here that it is crucial for the Monte-Carlo method to use the same underlying Brownian motion for $\hat{X}^{n,0}$ and $\hat{X}^{n,1}$. Thus, the variance of $n\left(f(\hat{X}^{n,1}_T) - f(\hat{X}^{n,0}_T)\right)$ is quite moderate. It is shown in~\cite[Appendix A]{AB} that this variance is bounded when using the Euler scheme for an SDE with smooth coefficients. The theoretical analysis of the variance in our framework is beyond the scope of the paper. We only check numerically how it evolves with respect to~$n$ on our experiments, see Table~\ref{Table_VAR_CIR1} below. 
}

The approximation $\hat{P}^{3,n}$ is more involved. Let $\kappa'$ be an independent uniform random variable on $\{0,\ldots,n-1\}$. We 
define the scheme  $\hat{X}^{n,2}$:
\begin{align*}
  \hat{X}^{n,2}_{kh_1}                 & = \hat{X}^{n,1}_{kh_1}, \quad \hat{X}^{n,2}_{\kappa h_1+k' h_2} = \hat{X}^{n,1}_{\kappa h_1+k'h_2},                                                                                0\leq k \leq \kappa,\ 0\leq k' \leq \kappa',   \\
  \hat{X}^{n,2}_{\kappa h_1+\kappa' h_2+(k''+1) h_3} & = \varphi(\hat{X}^{n,2}_{\kappa h_1+\kappa' h_2+k'' h_3}, h_3, W_{\kappa h_1+\kappa'h_2 +(k''+1)h_3} - W_{\kappa h_1+\kappa'h_2 +k''h_3} ,V_{2n+k''}),   \\
  &\phantom{= \varphi(\hat{X}^{n,2}_{\kappa h_1+\kappa' h_2+k'' h_3}, h_3, W_{\kappa h_1+\kappa'h_2 +(k''+1)h_3}) ********}0 \leq k'' \leq n-1,      \\
  \hat{X}^{n,2}_{\kappa h_1 +(k'+1)h_2}             & = \varphi(\hat{X}^{n,2}_{\kappa h_1 +k' h_2}, h_2, W_{\kappa h_1+(k'+1)h_2} - W_{\kappa h_1+k'h_2},V_{n+k'}),                                    \,  \kappa+1 \leq k'\leq n-1.\\
  \hat{X}^{n,2}_{(k+1)h_1}             & = \varphi(\hat{X}^{n,2}_{kh_1}, h_1, W_{(k+1)h_1} - W_{kh_1},V_k),                                       \kappa+1 \leq k\leq n-1.
\end{align*}
This is the scheme obtained on the time grid $\Pi^1 \cup  \{ \kappa T/n + \kappa' T/n^2 + k''T/n^3 , 1 \leq k'' \leq n-1 \}$.
We have $$\sum_{k=0}^{n-1}Q_1^{[n-(k+1)]}\left[\sum_{k'=0}^{n-1}Q_2^{[n-(k'+1)]}[Q_3^{[n]}-Q_2]Q_2^{[k']} \right] Q_1^{[k]}f =n^2\E[f(\hat{X}^{n,2}_T) - f(\hat{X}^{n,1}_T)].$$
We now explain how to calculate the second term in~\eqref{def_P_3}. Let $(\kappa_1,\kappa_2)$ be an independent random variable uniformly distributed on the set $\{(k_1,k_2):0\le k_1<k_2<n\}$. We define: 
\begin{align*}
  \hat{X}^{n,3}_{kh_1}                 & = \hat{X}^{n,0}_{kh_1},                                                                            & \quad   0\leq k \leq \kappa_1,   \\
  \hat{X}^{n,3}_{\kappa_1 h_1+(k'+1)h_2} & = \varphi(\hat{X}^{n,3}_{\kappa_1 h_1+k'h_2}, h_2, W_{\kappa h_1+(k'+1)h_2} - W_{\kappa h_1+k'h_2},V_{3n+k'}), & \quad 0 \leq k' \leq n-1,      \\
  \hat{X}^{n,3}_{(k+1)h_1}             & = \varphi(\hat{X}^{n,3}_{kh_1}, h_1, W_{(k+1)h_1} - W_{kh_1},V_k),                                     & \quad \kappa_1+1 \leq k\leq n-1,
\end{align*}
\begin{align*}
  \hat{X}^{n,4}_{kh_1}                 & = \hat{X}^{n,0}_{kh_1},                                                                            & \quad   0\leq k \leq \kappa_2,   \\
  \hat{X}^{n,4}_{\kappa_2 h_1+(k'+1)h_2} & = \varphi(\hat{X}^{n,4}_{\kappa_2 h_1+k'h_2}, h_2, W_{\kappa_2 h_1+(k'+1)h_2} - W_{\kappa_2 h_1+k'h_2},V_{4n+k'}), & \quad 0 \leq k' \leq n-1,      \\
  \hat{X}^{n,4}_{(k+1)h_1}             & = \varphi(\hat{X}^{n,4}_{kh_1}, h_1, W_{(k+1)h_1} - W_{kh_1},V_k),                                     & \quad \kappa_2+1 \leq k\leq n-1,
\end{align*}
and
\begin{align*}
  \hat{X}^{n,5}_{kh_1}                 & = \hat{X}^{n,3}_{kh_1},                                                                            & \quad   0\leq k \leq \kappa_2,   \\
  \hat{X}^{n,5}_{\kappa_2 h_1+(k'+1)h_2} & = \varphi(\hat{X}^{n,5}_{\kappa_2 h_1+k'h_2}, h_2, W_{\kappa_2 h_1+(k'+1)h_2} - W_{\kappa_2 h_1+k'h_2},V_{4n+k'}), & \quad 0 \leq k' \leq n-1,      \\
  \hat{X}^{n,5}_{(k+1)h_1}             & = \varphi(\hat{X}^{n,5}_{kh_1}, h_1, W_{(k+1)h_1} - W_{kh_1},V_k),                                     & \quad \kappa_2+1 \leq k\leq n-1,
\end{align*}
These schemes correspond respectively to the time grids $\Pi^0 \cup  \{ \kappa_1 T/n + k'T/n^2 , 1 \leq k' \leq n-1 \}$, $\Pi^0 \cup  \{ \kappa_2 T/n + k'T/n^2 , 1 \leq k' \leq n-1 \}$ and $\Pi^0 \cup  \{ \kappa_1 T/n + k'T/n^2 , 1 \leq k' \leq n-1 \}\cup  \{ \kappa_2 T/n + k'T/n^2 , 1 \leq k' \leq n-1 \}$. We then get \begin{align}
  \hat{P}^{3,n}f  =& \E[f(\hat{X}^{n,0}_T)] + n\E[f(\hat{X}^{n,1}_T) - f(\hat{X}^{n,0}_T)] +n^2 \E[f(\hat{X}^{n,2}_T) - f(\hat{X}^{n,1}_T)] \label{P3n} \\
  &+\frac{n(n-1)}{2}\E[f(\hat{X}^{n,5}_T) - f(\hat{X}^{n,4}_T)- f(\hat{X}^{n,3}_T)+ f(\hat{X}^{n,0}_T)]. \notag 
\end{align}

\subsection{Simulations result for the CIR process}\label{Sim_CIR}
In this subsection, we want to illustrate the convergence of the approximations $\hat{P}^{2,n}$ and $\hat{P}^{3,n}$, which together with the use of the second order scheme \eqref{NV_scheme} guarantee respectively approximations of order four and six by Theorem~\ref{thm_main}. In order to calculate these approximations, we use Monte-Carlo estimators of~\eqref{P2n_indep} and~\eqref{P3n}, using independent samples for each expectation. The number of samples (up to $10^{11}$)  is such that we can neglect the statistical error. In Figures \ref{CIR_Plot1}, \ref{CIR_Plot2} and \ref{CIR_Plot3} we plot the convergence in function of the time step for different parameters choices, taking advantage of the closed formula for the Laplace transform of the CIR process, see e.g.~\cite[Proposition 1.2.4]{AA_book}. The three numerical experiments test different levels of the ratio $\sigma^2 / 4a$ in decreasing order. We observe that the slopes estimated on the log-log plots are close to 2, 4 and 6 respectively, so that they are in accordance with Theorem~\ref{thm_main}. Note however that Theorem~\ref{thm_main} gives an asymptotic result for $n\to \infty$, while we are restricted here to rather small values of $n$ since we are using a large number of samples to kill the statistical error.  
In all the cases shown, the approximations of higher order outperform the one built with the simple second order scheme~\eqref{NV_scheme}. Talking about accuracies, the fourth order approximation for $n=3$ shows an absolute relative error of about $0.17\%$ in the tests in Figures \ref{CIR_Plot1}, and \ref{CIR_Plot2} and $0.02\%$ in the one in Figure \ref{CIR_Plot3}; the sixth order approximation  already for $n=3$ exhibits a relative error of $0.002\%$ in each case studied.

\begin{figure}[h]
  \centering
  \begin{subfigure}[h]{0.49\textwidth}
    \centering
    \includegraphics[width=\textwidth]{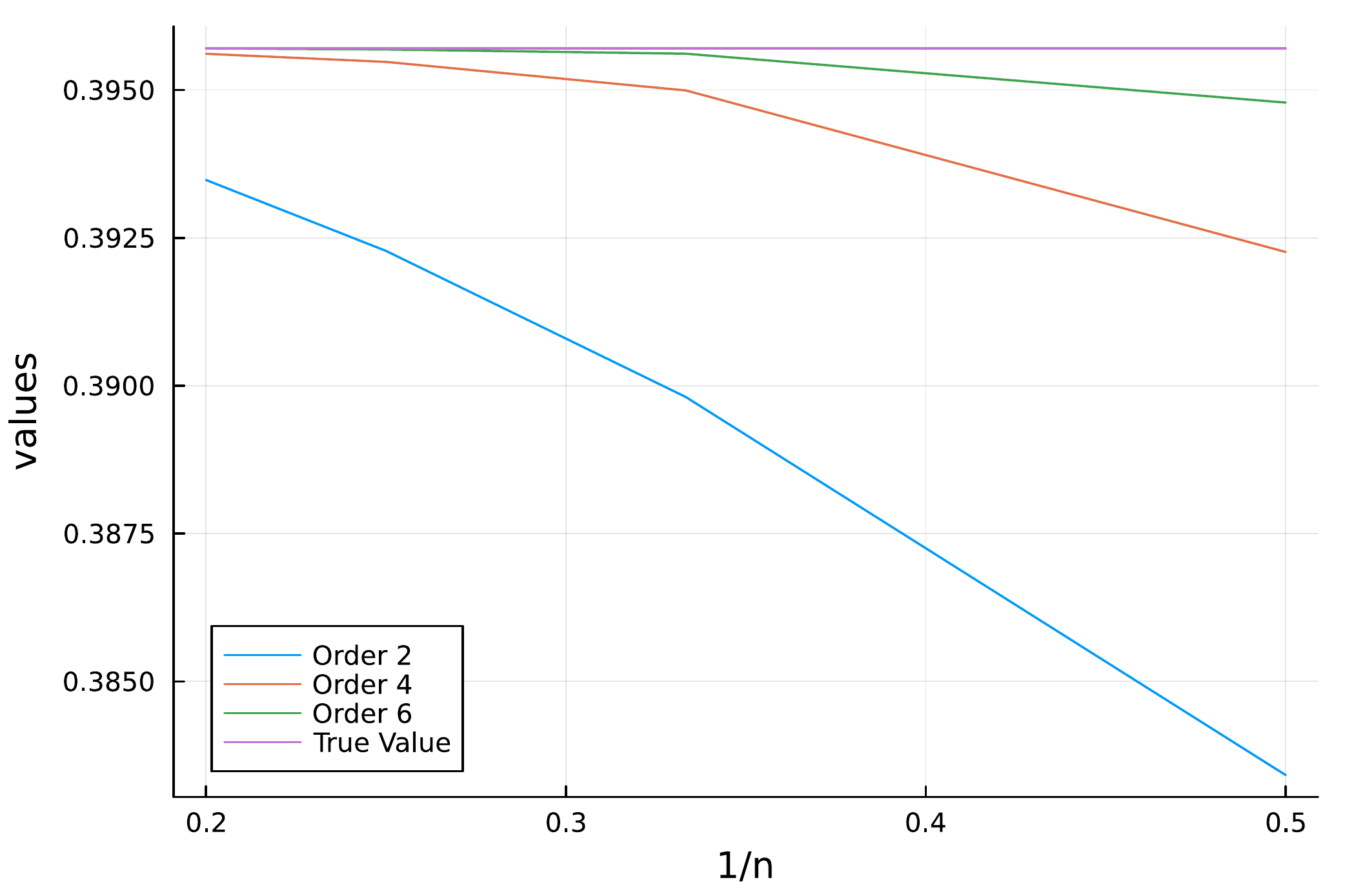}
    \caption{Values plot}
    \label{fig:values_plot13}
  \end{subfigure}
  \hfill
  \begin{subfigure}[h]{0.49\textwidth}
    \centering
    \includegraphics[width=\textwidth]{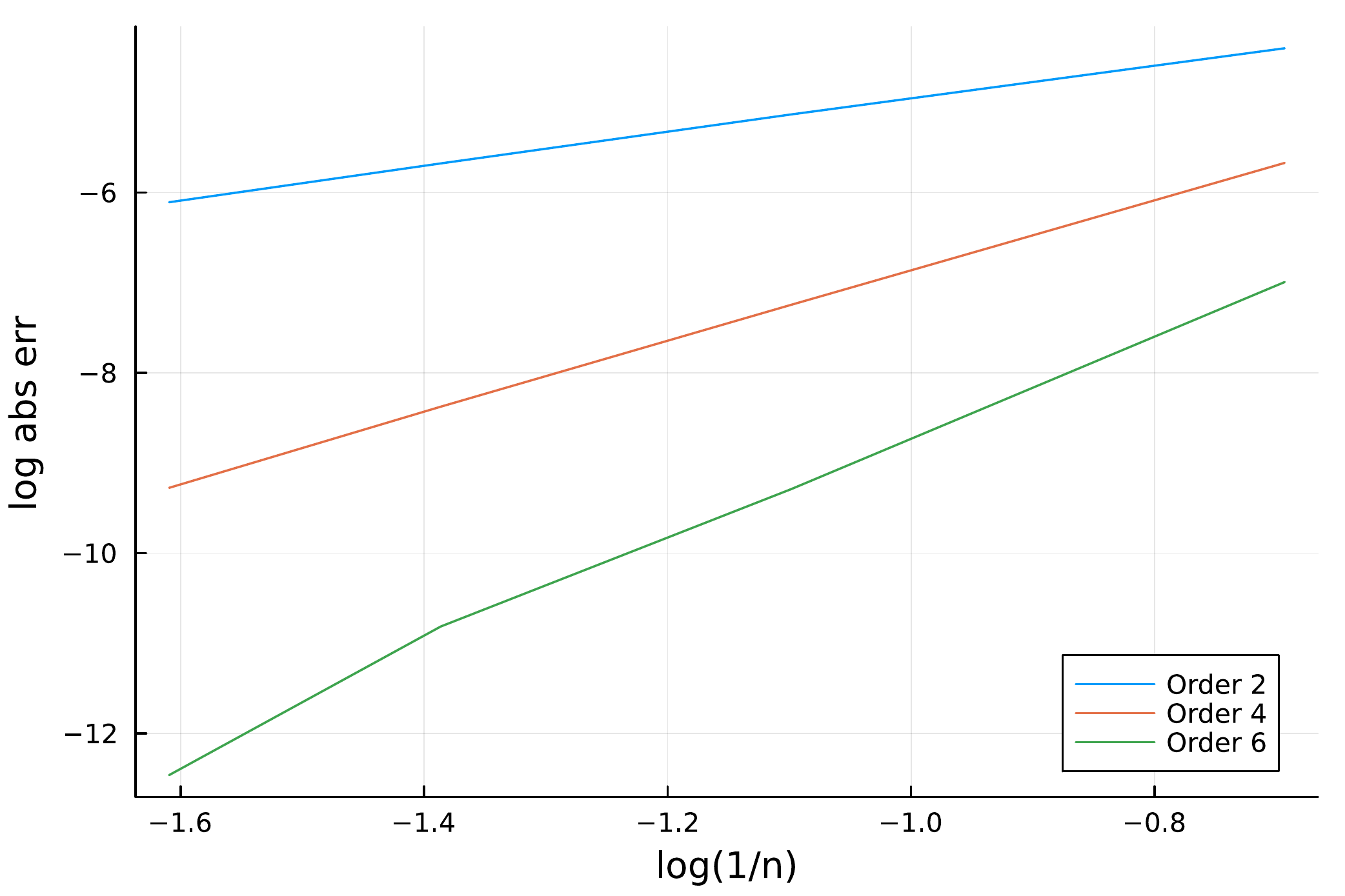}
    \caption{Log-log plot}
    \label{fig:log-log_plot13}
  \end{subfigure}
  \caption{Parameters: $x=0.0$, $a=0.2$, $k=0.5$, $\sigma=0.65$, $f(z)=\exp(-10 z)$ and $T=1$ \blue{($\frac{\sigma^2}{2a}\approx 1.06$)}. Graphic~({\sc a}) shows the values of $\hat{P}^{1,n}f$, $\hat{P}^{2,n}f$, $\hat{P}^{3,n}f$ as a function of the time step $1/n$  and the exact value. Graphic~({\sc b}) draws $\log(|\hat{P}^{i,n}f-P_Tf|)$ in function of $\log(1/n)$: the regressed slopes are 1.86, 3.93 and 5.87 for the second, fourth and sixth order respectively.}\label{CIR_Plot1}
\end{figure}

\begin{figure}[h!]
  \centering
  \begin{subfigure}[h]{0.49\textwidth}
    \centering
    \includegraphics[width=\textwidth]{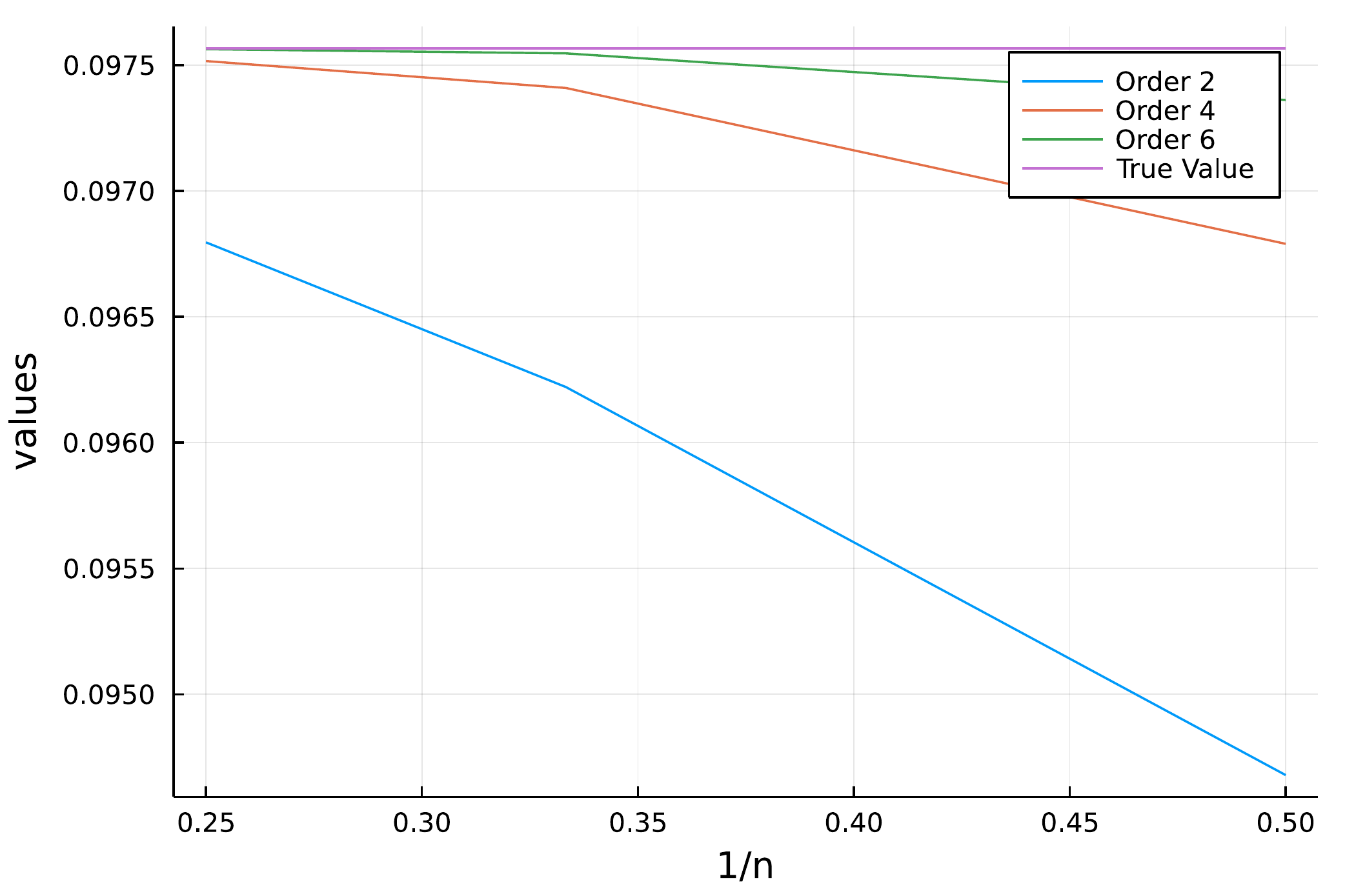}
    \caption{Values plot}
    \label{fig:values_plot9}
  \end{subfigure}
  \hfill
  \begin{subfigure}[h]{0.49\textwidth}
    \centering
    \includegraphics[width=\textwidth]{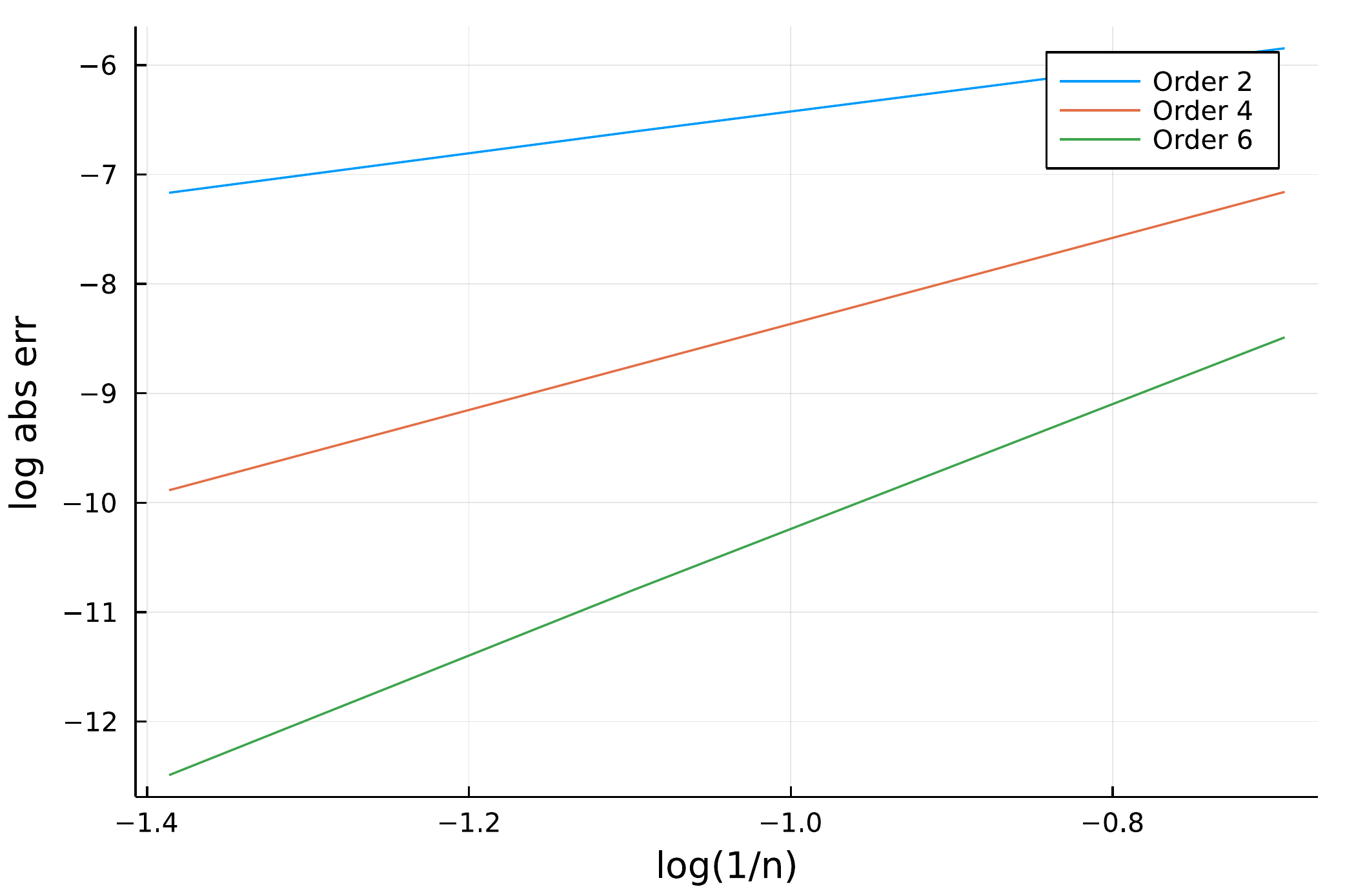}
    \caption{Log-log plot}
    \label{fig:log-log_plot9}
  \end{subfigure}
  \caption{Parameters: $x=0.3$, $a=0.4$, $k=1$, $\sigma=0.4$, $f(z)=\exp(-8 z)$ and $T=1$ \blue{($\frac{\sigma^2}{2a}= 0.2$)}. Graphic~({\sc a}) shows the values of $\hat{P}^{1,n}f$, $\hat{P}^{2,n}f$, $\hat{P}^{3,n}f$ as a function of the time step $1/n$  and the exact value. Graphic~({\sc b}) draws $\log(|\hat{P}^{i,n}f-P_Tf|)$ in function of $\log(1/n)$:  the regressed slopes are 1.90, 3.93 and 5.77 for the second, fourth and sixth order respectively.}\label{CIR_Plot2}
\end{figure}

\begin{figure}[h]
  \centering
  \begin{subfigure}[h]{0.49\textwidth}
    \centering
    \includegraphics[width=\textwidth]{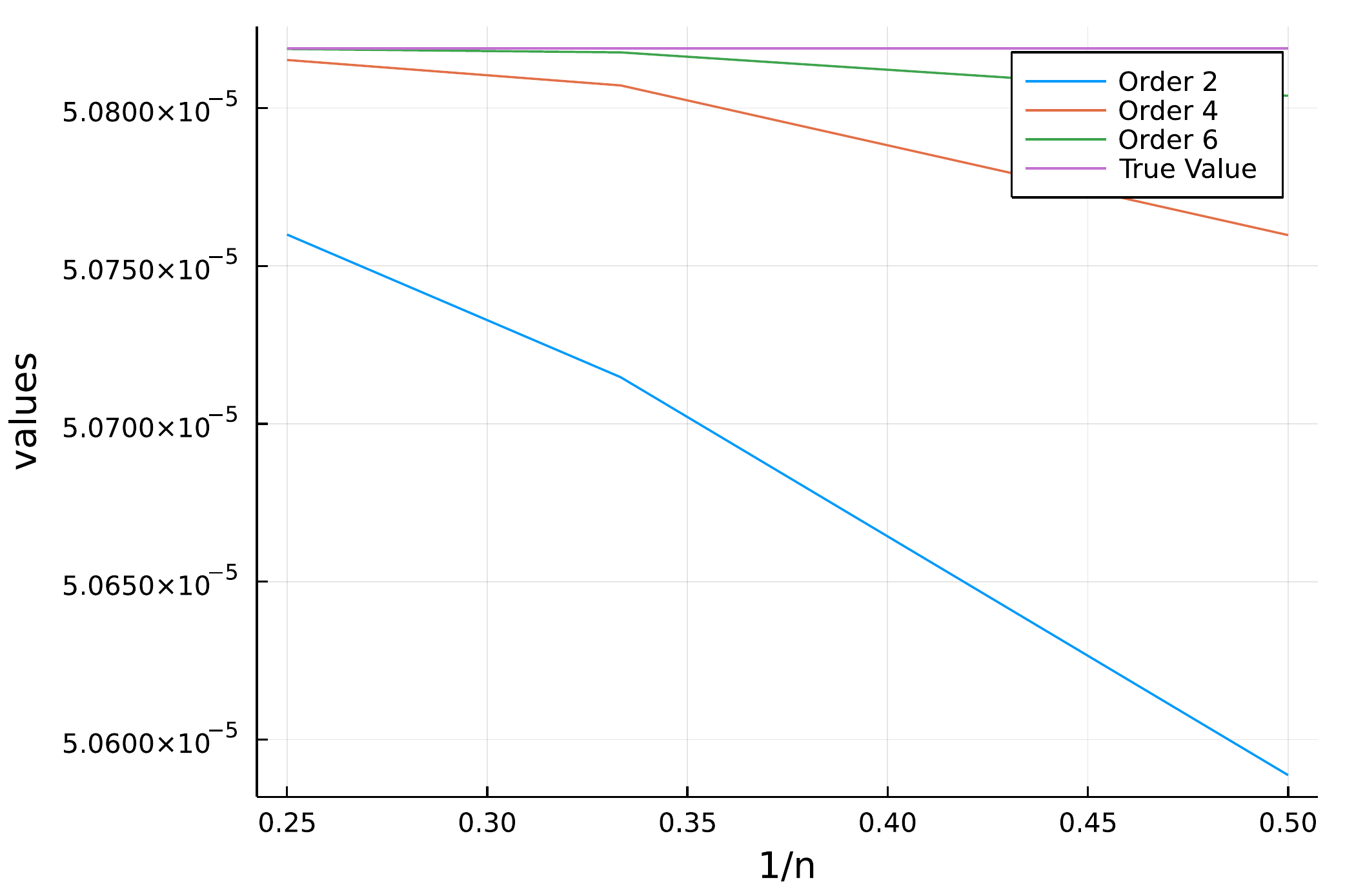}
    \caption{Values plot}
    \label{fig:values_plot15}
  \end{subfigure}
  \hfill
  \begin{subfigure}[h]{0.49\textwidth}
    \centering
    \includegraphics[width=\textwidth]{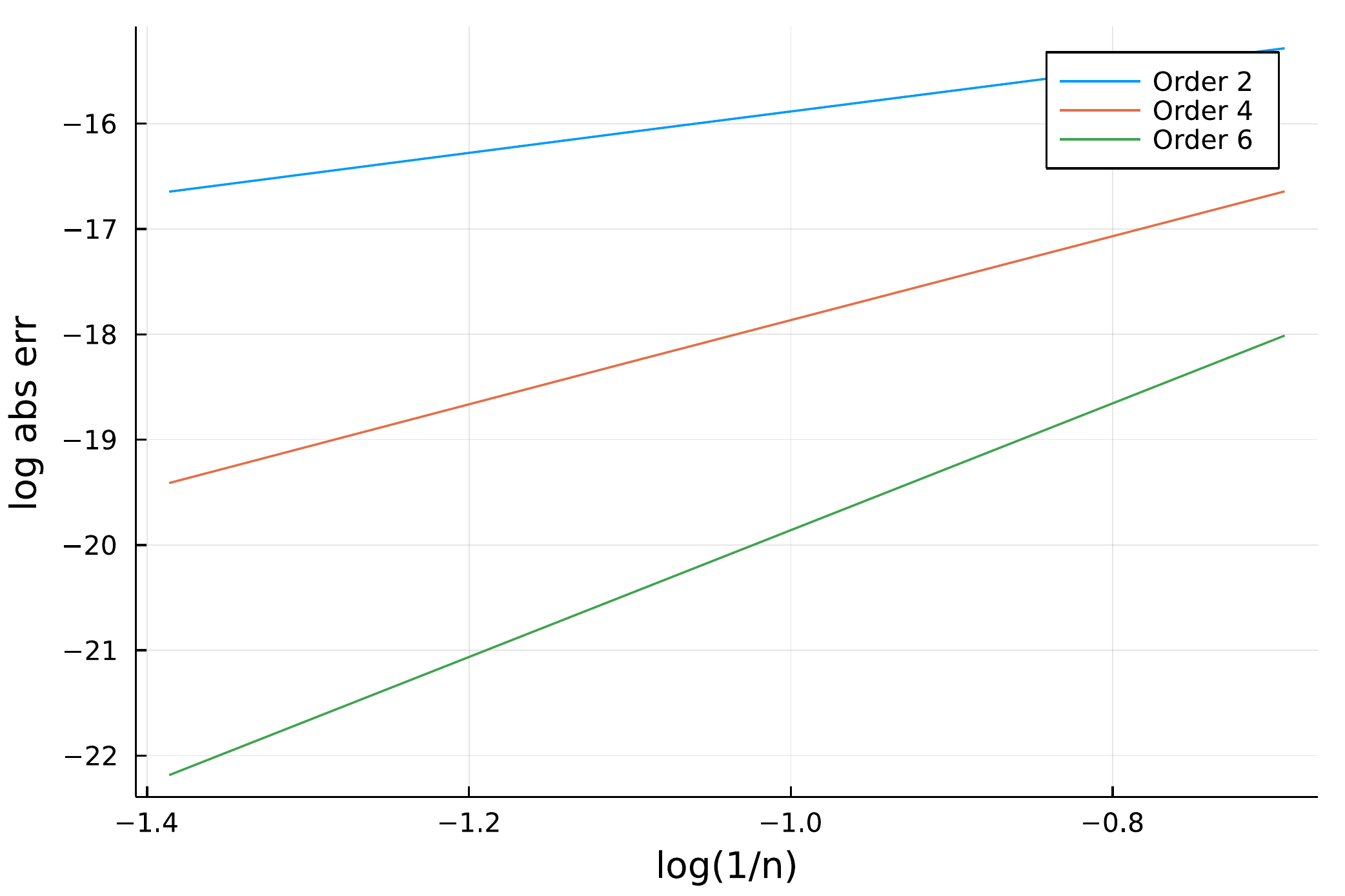}
    \caption{Log-log plot}
    \label{fig:log-log_plot15}
  \end{subfigure}
  \caption{Parameters: $x=10$, $a=10$, $k=1$, $\sigma=0.23$, $f(z)=\exp(- z)$ and $T=1$ \blue{($\frac{\sigma^2}{2a}\approx 0.0026$)}.
  Graphic~({\sc a}) shows the values of $\hat{P}^{1,n}f$, $\hat{P}^{2,n}f$, $\hat{P}^{3,n}f$ as a function of the time step $1/n$  and the exact value. Graphic~({\sc b}) draws $\log(|\hat{P}^{i,n}f-P_Tf|)$ in function of $\log(1/n)$:  the regressed slopes are 1.96, 4.00 and 6.02 for the second, fourth and sixth order respectively.}\label{CIR_Plot3}
\end{figure}

\subsection{Simulations result for the Heston model}\label{Sim_Heston}

In this subsection, we want to test the second order scheme for the Heston model proposed by Alfonsi in~\cite{AA_MCOM} along with the approximations of order 4 and 6 obtained with combination of random grids. 
First, we recall the couple of stochastic differential equations describing this model
\begin{equation}\label{Heston SDEs}
  \begin{cases}
    dS^{(x,s)}_t = rS^{(x,s)}_t dt + \sqrt{X_t}S^{(x,s)}_t (\rho dW_t + \sqrt{1-\rho^2} dZ_t), \ S^{(x,s)}_0=s,\\
    dX^x_t = (a-kX^x_t) dt +\sigma \sqrt{X^x_t} dW_t, \ X^x_0=x,
  \end{cases}
\end{equation}
where $W$ and $Z$ are two independent Brownian motions. We define the two following random variables
\begin{align*}
  S_1\big((x,s),h,Z_h \big) & = \left(x, s\exp\Big(\sqrt{x(1-\rho^2)}Z_h\Big)                            \right)                                                                               \\
  S_2\big((x,s),h,W_h \big) & =\bigg( \varphi(x,h,W_h),\\& s\exp\left((r-\frac{\rho}{\sigma}a)h + (\frac{\rho}{\sigma}k-\frac{1}{2})\frac{x+\varphi(x,h,W_h)}{2}h + \frac{\rho}{\sigma}(\varphi(x,h,W_h)-x)\right)\bigg)
\end{align*}
where $\varphi$ is defined by~\eqref{def_varphi} anf corresponds to the second order scheme for the CIR process.
We define as in \cite{AA_MCOM} the second order scheme for~\eqref{Heston SDEs} as follows
\begin{equation}\label{H2S}
  \Phi\big((x,s),h,(W_h,Z_h),B\big) = \begin{cases} S_2\left(S_1\big((x,s),h,Z_h \big),h,W_h\right), \text{ if } B=1, \\
    S_1\left(S_2\big((x,s),h,W_h \big),h,Z_h\right), \text{ if } B=0.\end{cases} 
\end{equation}
where $B$ is an independent Bernoulli random variable of parameter 1/2.

To test the order of the approximations $\hat{P}^{2,n}$ and $\hat{P}^{3,n}$ boosting the second order scheme~\eqref{H2S}, we have calculated European put prices taking advantage of the existence of a semi closed formula for this option, see~\cite{Heston}. In Figure~\ref{Heston_orders} we draw the convergence in function of the time step. Again, we noticed that the slopes obtained on the log-log plot are in line with the expected order of convergence. 
More importantly, we see that the correction terms of the approximations $\hat{P}^{2,n}$ and $\hat{P}^{3,n}$ really improves the precision. They respectively give relative errors of a 0.035\% and 0.0023\%, already for $n=3$.

\begin{figure}[h!]
  \centering
  \begin{subfigure}[h]{0.49\textwidth}
    \centering
    \includegraphics[width=\textwidth]{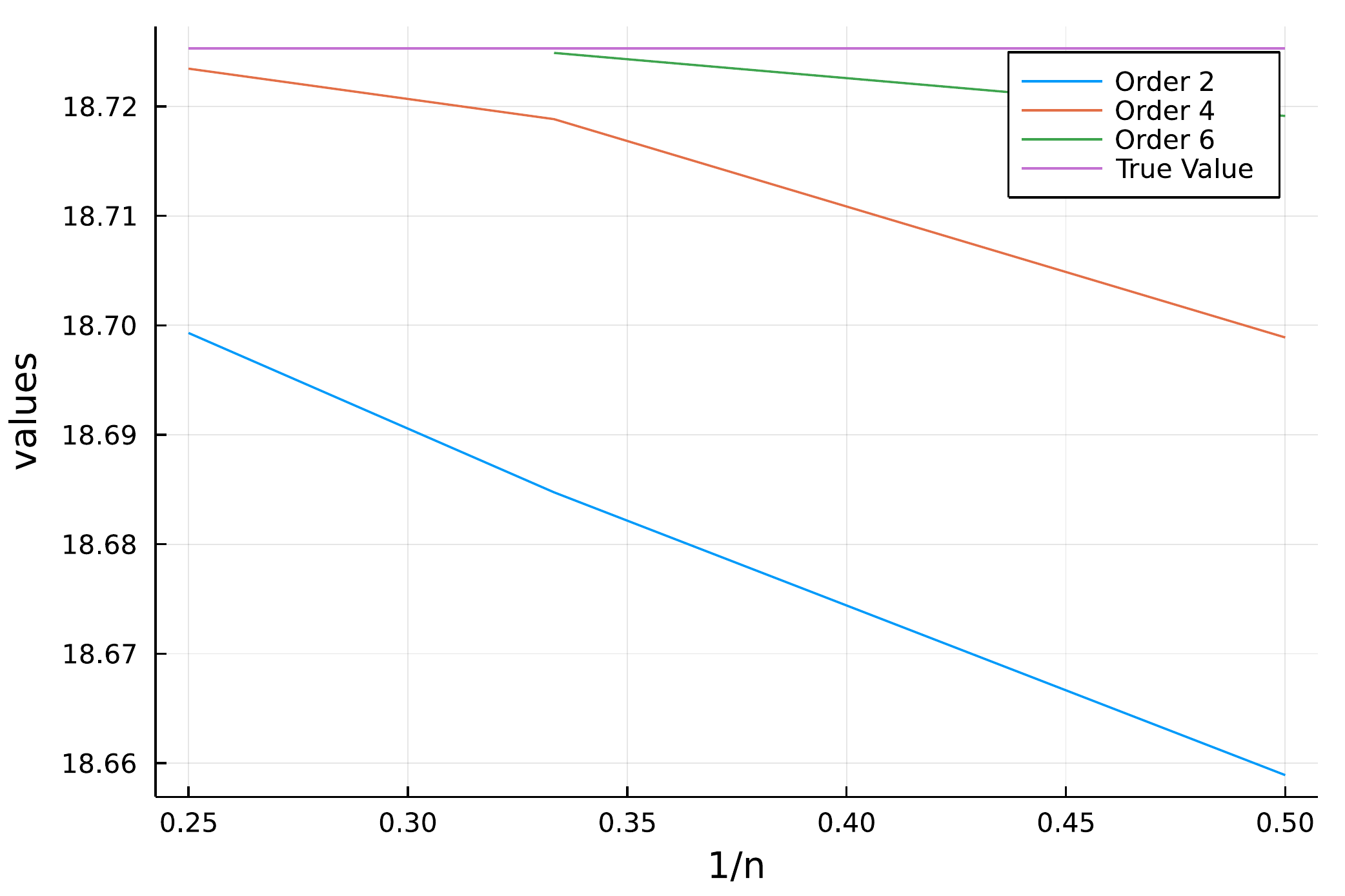}
    \caption{Values plot}
    \label{fig:values_plot_heston2}
  \end{subfigure}
  \hfill
  \begin{subfigure}[h]{0.49\textwidth}
    \centering
    \includegraphics[width=\textwidth]{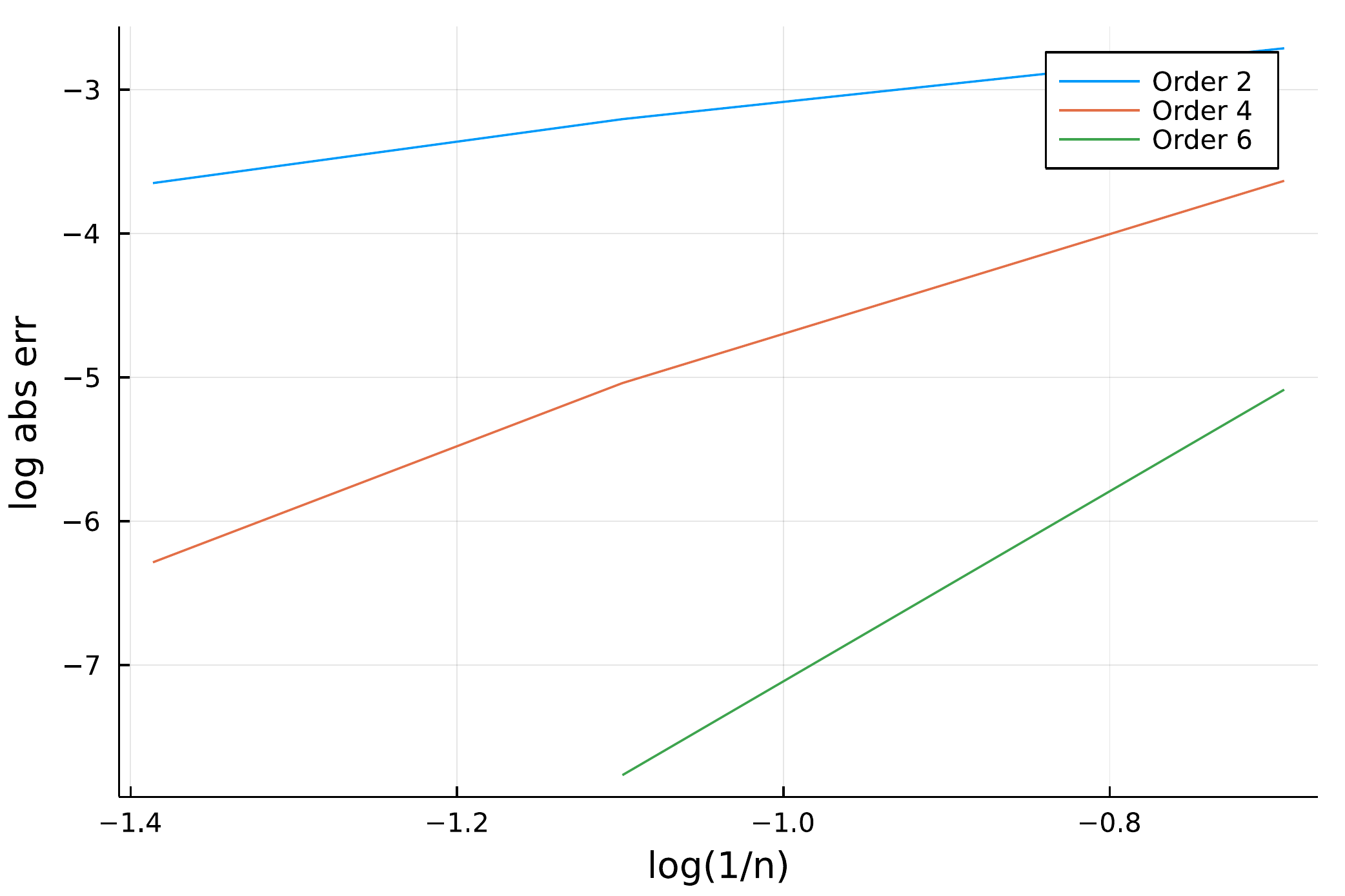}
    \caption{Log-log plot}
    \label{fig:log-log_plot_heston2}
  \end{subfigure}
  \caption{Test function: $f(x,s)=(K-s)^+$. Parameters: $S_0=100$, $r=0$, $x=0.25$, $a=0.25$, $k=1$, $\sigma=0.65$, $\rho=-0.3$, $T=1$, $K=100$ \blue{($\frac{\sigma^2}{2a}= 0.845$)}.
  Graphic~({\sc a}) shows the values of $\hat{P}^{1,n}f$, $\hat{P}^{2,n}f$, $\hat{P}^{3,n}f$ as a function of the time step $1/n$  and the exact value. Graphic~({\sc b}) draws $\log(|\hat{P}^{i,n}f-P_Tf|)$ in function of $\log(1/n)$:  the regressed slopes are 1.34, 4.00 and 6.02 for the second, fourth and sixth order respectively.}\label{Heston_orders}
\end{figure}

\subsection{Optimized implementation of~$\hat{P}^{2,n}$}\label{Optim_time/var}

The approximations $\hat{P}^{2,n}$ and $\hat{P}^{3,n}$ defined respectively by~\eqref{P2n_indep} and~\eqref{P3n} involve respectively two and four expectations. The larger is $\nu$ the more  expectations are involved in $\hat{P}^{\nu,n}$. Thus,  for simplicity, independent samples were used  by Alfonsi and Bally~\cite{AB} to compute each term. However,  it may be interesting to reuse some samples in order to spare computation time. This is what we investigate in this subsection. 

Namely, Equation \eqref{P2n_indep} leads naturally to the  two following estimators  of $P^{2,n}f$:
\begin{align}\label{Theta_i}
  \Theta_{\i}(M_1,M_2,n) &= \frac{1}{M_1} \sum_{j=1}^{M_1} f\big((\hat{X}^{n,0}_T)^{(j)} \big) +  \frac{1}{M_2} \sum_{i=M_1+1}^{M_1+M_2} n\left(f\big((\hat{X}^{n,1}_T)^{(i)}\big) -   f\big((\hat{X}^{n,0}_T)^{(i)}\big) \right), \\
  \Theta_{\d}(M_1,M_2,n) &= \frac{1}{M_1} \sum_{j=1}^{M_1} f\big((\hat{X}^{n,0}_T)^{(j)} \big) +  \frac{1}{M_2} \sum_{i=1}^{M_2} n\left(f\big((\hat{X}^{n,1}_T)^{(i)}\big) -   f\big((\hat{X}^{n,0}_T)^{(i)}\big) \right).\label{Theta_d}
\end{align}
The first one takes  independent samples  and we call this estimator $\Theta_{\i}$. This approach is the one used in~\cite{AB}. In the second case, we reuse the first $M_1\wedge M_2$ simulations of $f\big((\hat{X}^{n,0}_T)^{(i)}\big)$ in both sums. We call this estimator $\Theta_{\d}$ to indicate the dependence between samples. In terms of variance, we have
\begin{align}
   & \Var\left(\Theta_{\i}(M_1,M_2,n)\right) = \frac{\Var\big(f(\hat{X}^{n,0}_T)\big)}{M_1} + \frac{\Var\big(n(f(\hat{X}^{n,1}_T) - f(\hat{X}^{n,0}_T))\big)}{M_2}, \\
   & \begin{multlined}
    \Var\left(\Theta_{\d}(M_1,M_2,n)\right) = \frac{\Var\big(f(\hat{X}^{n,0}_T)\big) }{M_1}+ 2 \frac{\Cov\big(f(\hat{X}^{n,0}_T),n(f(\hat{X}^{n,1}_T) - f(\hat{X}^{n,0}_T))\big) }{M_1\vee M_2} \\
    + \frac{\Var\big(n(f(\hat{X}^{n,1}_T) - f(\hat{X}^{n,0}_T))\big)}{M_2}.
  \end{multlined}\label{var_thetad}
\end{align}
Let us define $t_1$ as the time to generate one sample $f(\hat{X}^{n,0}_T)$ and $t_2$ as the one needed for one sample of the correction $n(f(\hat{X}^{n,1}_T) - f(\hat{X}^{n,0}_T))$. The computation time needed to compute $\Theta_{\i}$ is given by $g_{\i}(M_1,M_2)=M_1t_1+M_2t_2$, while the one needed to compute $\Theta_{\d}$ is $g_{\d}(M_1,M_2)=\mathbf{1}_{M_1\ge M_2} [(M_1-M_2)t_1+M_2t_2] +\mathbf{1}_{M_1< M_2}M_2t_2$. We note $\zeta=\frac{t_2}{t_1}$. From the definition of schemes $\hat{X}^{n,0}$ and $\hat{X}^{n,1}$ in~\eqref{Scheme0}, we observe that  $2\le \zeta \le 3 $ and that $\zeta \approx 2.5$ in average since these schemes are equal up to $\kappa h_1$.  
The advantage of $\Theta_{\d}$ is not necessarily in reducing the variance, but in decreasing the number of simulations needed, making it more efficient from a computational time point of view.\\

We want to find the optimal numbers of simulations $M_1$ and $M_2$ for our estimators in order to minimize the execution time for a given variance $\varepsilon^2$. Let us define 
$\sigma^2_2(n) = \Var\big(f(\hat{X}^{n,0}_T)\big)$, $\sigma^2_{4}(n) = \Var\big(n(f(\hat{X}^{n,1}_T) - f(\hat{X}^{n,0}_T))\big)$, $\Gamma(n)=\Cov\big(f(\hat{X}^{n,0}_T), n(f(\hat{X}^{n,1}_T) - f(\hat{X}^{n,0}_T))\big)$. For~$\Theta_{\i}$, the minimization of $g_{\i}$ given that $\sigma^2_2(n)/M_1+\sigma^4_2(n)/M_2=\varepsilon^2$ leads to $M_1=\sqrt{\zeta}\frac{\sigma_2(n)}{\sigma_4(n)}M_2$ and then to:
\begin{equation}\label{optim_i}
M_{1,\i}=\left\lceil \frac{1}{\varepsilon^2}\left( \sigma^2_2(n) +\sqrt{\zeta} \sigma_2(n)\sigma_4(n)   \right)  \right\rceil, \ M_{2,\i}= \left\lceil \frac{1}{\varepsilon^2}\left( \sigma^2_4(n) +\frac{\sigma_2(n)\sigma_4(n)}{\sqrt{\zeta}}   \right)  \right\rceil.
\end{equation}
To minimize the execution time $g_{\d}$, one has first to decide whether we take $M_1\ge M_2$ or $M_1<M_2$. From~\eqref{var_thetad}, this amounts to compare $\frac{\sigma^2_2(n)+2\Gamma(n)}{m+\tilde{m}\zeta}$ with $\frac{\sigma^2_4(n)+2\Gamma(n)}{m+\tilde{m}}$ where $m=M_1\wedge M_2$ and $\tilde{m}\ge 0$ ($\tilde{m}$ simulations of the correction term takes the same time as $\zeta \tilde{m}$ simulations of~$f(\hat{X}^{n,0}_T)$). Taking the derivative at~$\tilde{m}=0$, we get that $M_1\ge M_2$ if $\zeta\frac{\sigma^2_2(n)+2\Gamma(n)}{\sigma^2_4(n)+2\Gamma(n)}\ge 1$, and $M_1<M_2$ otherwise. When $M_1\ge M_2$, the minimisation of $g_{\d}$ given $ \Var\left(\Theta_{\d}(M_1,M_2,n)\right) =\varepsilon^2$ leads to
\begin{align}
 \begin{cases} M_{1,\d}  &=\left\lceil \frac{1}{\varepsilon^2}\left( \sigma^2_2(n) + 2\Gamma(n) +\sqrt{\big(\sigma^2_2(n)+2\Gamma(n)\big)\sigma^2_4(n)(\zeta-1) }  \right)  \right\rceil, \\ M_{2,\d} & =\left\lceil \frac{1}{\varepsilon^2}\left( \sigma^2_4(n) +\sqrt{\frac{\big(\sigma^2_2(n)+2\Gamma(n)\big)\sigma^2_4(n)}{\zeta-1} }  \right)  \right\rceil.\end{cases} \label{optim_d}
\end{align}
We have similar formulas when $M_1<M_2$. In all our numerical experiments below, we are in the case where  $\zeta\frac{\sigma^2_2(n)+2\Gamma(n)}{\sigma^2_4(n)+2\Gamma(n)}\ge 1$ and thus taking $M_1\ge M_2$ is optimal.

Now, we show the performance of the two estimators~\eqref{Theta_i} and \eqref{Theta_d}. To do this, we calculate the empirical variances $\sigma_2^2(n)$, $\sigma_4^2(n)$ and the  empirical covariance $\Gamma(n)$ on a small sampling, fix a desired precision $\varepsilon=1.96\sqrt{\Var(\Theta(M_1,M_2,n))}$ for both the estimators, so that all the terms have roughly the same statistical error with a 95\% confidence interval half-width equal to $\varepsilon$. We show two tables in which we set the precision $\varepsilon$ to $10^{-3}$. In Table~\ref{Table_T/V_I_VS_D_1}, we have   $\sigma_2^2(n) \gg \sigma_4^2(n) $, while in Table~\ref{Table_T/V_I_VS_D_2}, $\sigma_2^2(n)$ is still larger than $\sigma_4^2(n)$, but of the same order of magnitude.

\begin{table}[h!]
  \begin{tabular}{ c|c|c|c|c|  }
    \cline{2-5}
        & $n=2$     & $n=3$     & $n=4$      & $n=5$      \\
    \hline
    \multicolumn{1}{|c|}{$\Theta_i$}  & 63.04 & 96.15 & 131.84 & 165.80 \\
    \multicolumn{1}{|c|}{$\Theta_d$}  & 51.61 & 87.24 & 122.76 & 152.32 \\
    \hline
  \end{tabular}
  \caption{ Computation time (in seconds) needed by the Estimators $\Theta_i$ and $\Theta_d$ for a precision $\varepsilon=10^{-3}$.
  Test function: $f(x,s)=(K-s)^+$. Parameters: $S_0=100$, $r=0$, $x=0.4$, $a=0.4$, $k=1$, $\sigma=0.2$, $\rho=-0.3$, $T=1$, $K=100$ \blue{($\frac{\sigma^2}{2a}=0.05$)}.}\label{Table_T/V_I_VS_D_1}
\end{table}

\begin{table}[h!]
  \begin{tabular}{ c|c|c|c|c|  }
    \cline{2-5}
        & $n=2$     & $n=3$     & $n=4$      & $n=5$      \\
    \hline
    \multicolumn{1}{|c|}{$\Theta_i$}  & 59.50 & 102.13 & 148.45 & 193.41 \\
    \multicolumn{1}{|c|}{$\Theta_d$}  & 37.59 & 70.43  & 100.14 & 136.16 \\	
    \hline
  \end{tabular}
  \caption{Computation time (in seconds) needed by the Estimators $\Theta_i$ and $\Theta_d$ for a precision $\varepsilon=10^{-3}$.  Test function: $f(x,s)=(K-s)^+$. Parameters: $S_0=100$, $r=0$, $x=0.1$, $a=0.1$, $k=1$, $\sigma=0.63$, $\rho=-0.3$, $T=1$, $K=100$ \blue{($\frac{\sigma^2}{2a}\approx 1.98$)}.}\label{Table_T/V_I_VS_D_2}
\end{table}

We observe that we do not have a great gain in using $\Theta_{\i}$ when $\sigma_2^2(n)   \gg \sigma_4^2(n) $ (Table \ref{Table_T/V_I_VS_D_1}), while we save up to $30\%$ of execution time, using $\Theta_{\d}$ instead of $\Theta_{\i}$, when $\sigma_2^2 (n) $ is of the same order of magnitude $\sigma_4^2 (n)$ (Table \ref{Table_T/V_I_VS_D_2}).  Heuristically, this can be understood as follows: when $\sigma^2_2(n)$ is of the same magnitude as~$\sigma^2_4(n)$, so are $M_{1,\i}$ and $M_{2,\i}$, which gives an important gain in reusing the simulation of the correction term.  
In any case, $\Theta_{\d}$ turns out to be faster for each choice of parameters, and therefore we recommend it at the expense of $ \Theta_{\i}$.

\subsection{Comparison between the second and the fourth order approximation}
Subsections~\ref{Sim_CIR} and~\ref{Sim_Heston} have confirmed  numerically the theoretical results obtained in this paper. However, they do not compare directly the computation time required by the different methods. We now present numerical tests that allow us to prove  the real advantage of using the fourth order approximation $\hat{P}^{2,n}$ instead of the simple second order scheme. Namely, we compare the squared $L^2$ distance of the estimator $\Theta_d$ from the true value with the same distance between the estimator of $\hat{P}^{1, n^2}$ with the true value. We plot these quantities in function of the computation time needed. Note that $\hat{P}^{2, n}$ and $\hat{P}^{1, n^2}$ converges at a  rate of $O(n^{-4})$ so that their bias have the same order of magnitude.

\begin{figure}[h!]
  \centering
  \begin{subfigure}[h]{0.49\textwidth}
    \centering
    \includegraphics[width=\textwidth]{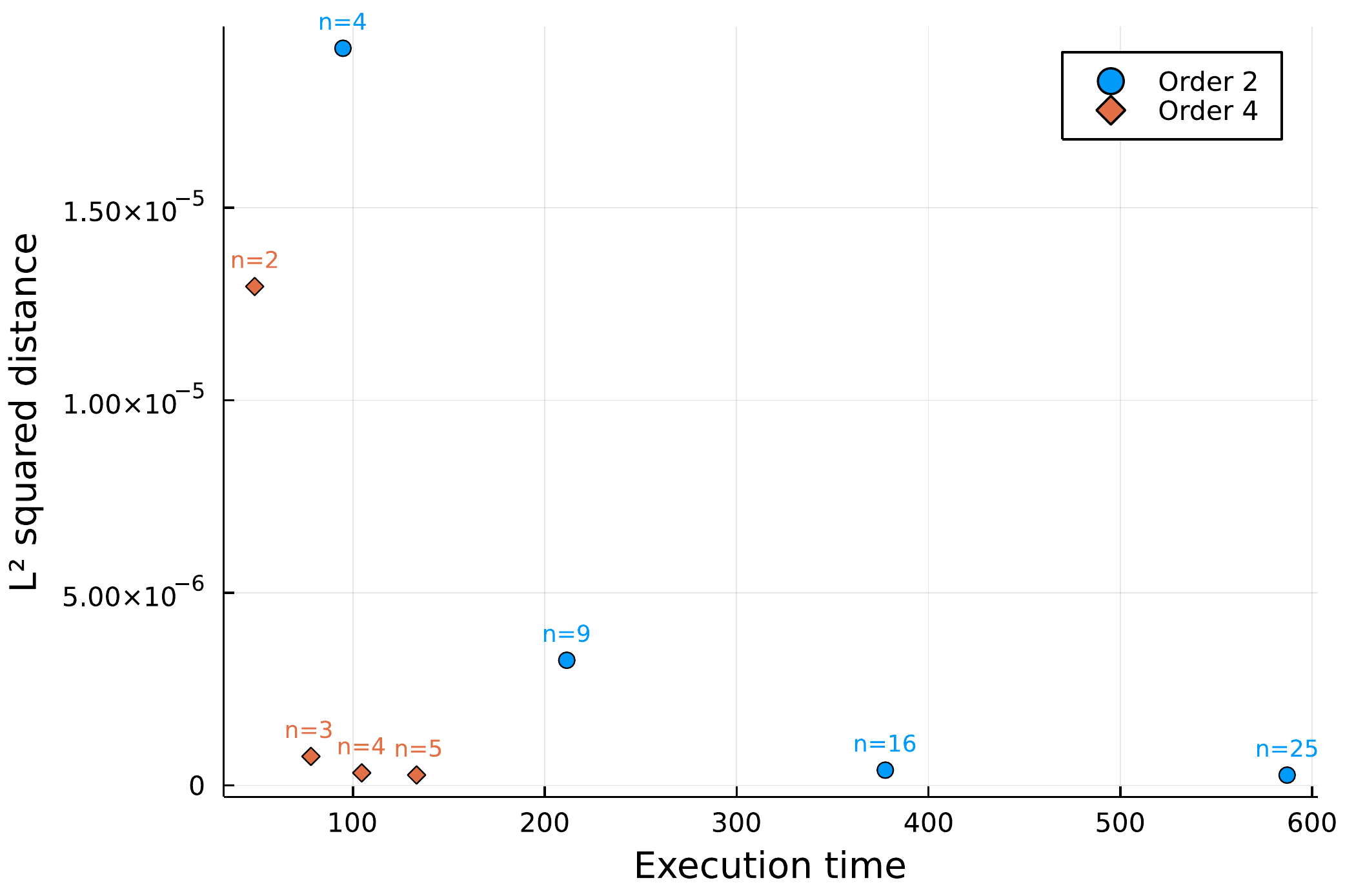}
    \caption{\blue{($\frac{\sigma^2}{2a}=0.0125$)}} %Small $\sigma^2/(4a)$
    \label{fig: L2distance_H4S_2M_Dep_VS_H2S6}
  \end{subfigure}
  \hfill
  \begin{subfigure}[h]{0.49\textwidth}
    \centering
    \includegraphics[width=\textwidth]{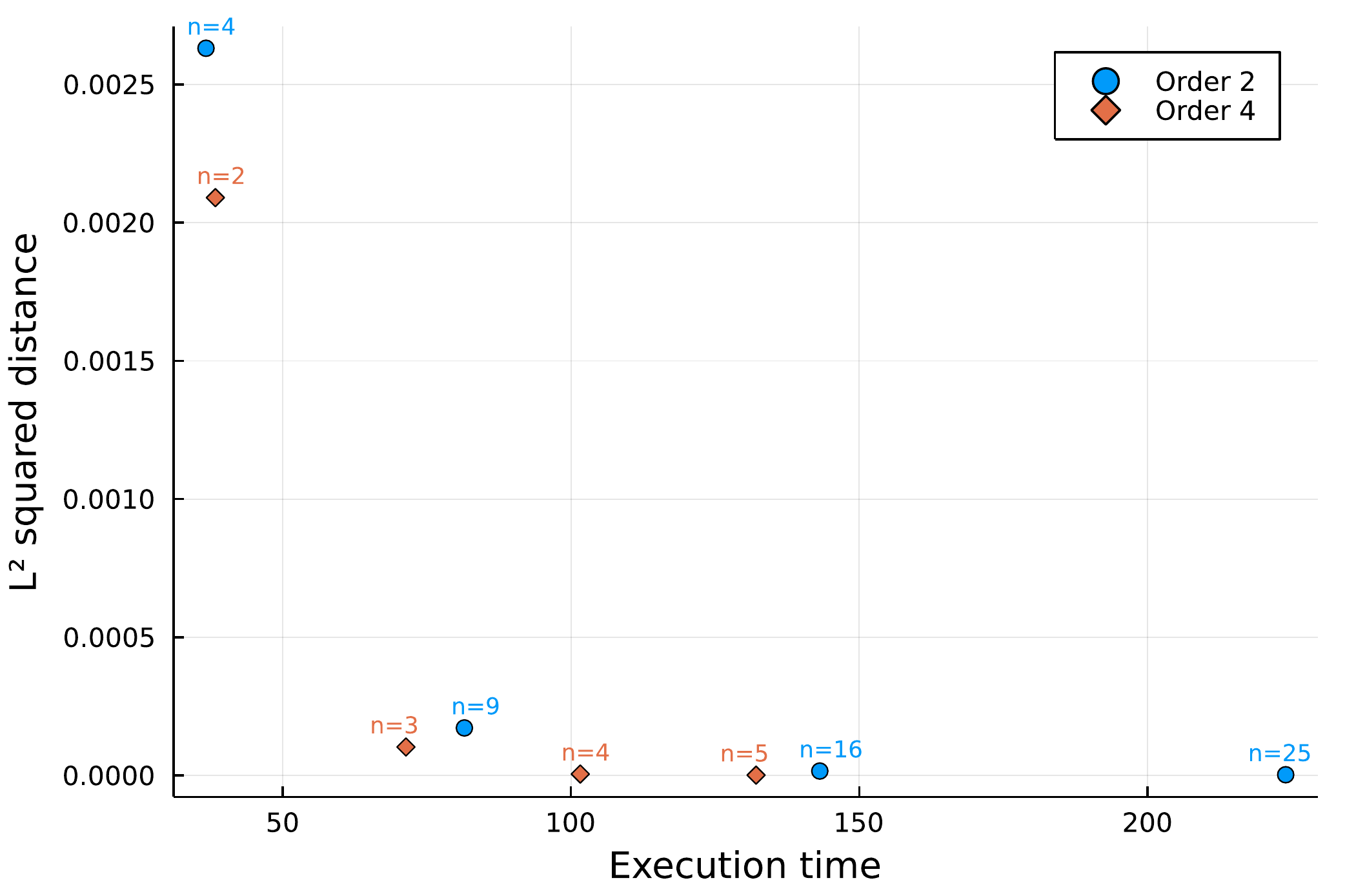}
    \caption{\blue{($\frac{\sigma^2}{2a}\approx 1.98$)}} %Large $\sigma^2/(4a)$
    \label{fig: L2distance_H4S_2M_Dep_VS_H2S13}
  \end{subfigure}
  \caption{
    $L^2$-square error in function of the execution time in seconds. Test function: $f(x,s)=(K-s)^+$. Parameters in graphic~({\sc a}) : $S_0=100$, $r=0$, $x=0.4$, $a=0.4$, $k=1$, $\sigma=0.1$, $\rho=-0.3$, $T=1$, $K=100$.\\Parameters in graphic~({\sc b}) : 
   $S_0=100$, $r=0$, $x=0.1$, $a=0.1$, $k=1$, $\sigma=0.63$, $\rho=-0.3$, $T=1$, $K=100$.}\label{H4S_dep_VS_H2S} 
\end{figure}

Figure~\ref{H4S_dep_VS_H2S} shows the results for the calculation of the price of a European put option in the Heston model with two different sets of parameters. In this numerical experience we set a precision $\varepsilon$ equal to $10^{-3}$. The empirical evidences show that the fourth order estimator $\Theta_d$ is the best choice, especially when the ratio $\frac{\sigma^2}{2a}\ll 1$ (Figure \ref{H4S_dep_VS_H2S} ({\sc a})) where the performance of the fourth order estimator is unparalleled. For example, $\hat{P}^{2, 3}$ is twice more accurate and more than twice faster than $\hat{P}^{1,9}$.
Even in Figure \ref{H4S_dep_VS_H2S} ({\sc b}), where  the ratio $\frac{\sigma^2}{2a}$ is larger and close to~2,  the fourth order estimator~$\Theta_d$ is more precise than the second order estimator and is faster from $n=3$ onward. These experiments illustrate the outperformance of the boosted estimator $\hat{P}^{2, n}$ with respect to $\hat{P}^{1, n}$.

%%%%%%%%%%%%%%%%%%%%%%%%%%%%%%%%%%%%%%%%%%%%%%%%%%%%%%% NEW %%%%%%%%%%%%%%%%%%%%%%%%%%%%%%%%%%%%%%%%%%%%%%%%%%%%%%%%%

\blue{\subsection{Numerical experiments for $\sigma^2>4a$}
In the previous subsections, we have presented analyses to confirm numerically the theoretical rates of convergence of our approximations, and to assess their computational time. This is why we have only considered parameters such that $\sigma^2\le 4a$, since this condition is required in Theorem~\ref{thm_main}. However, it is possible to test numerically the relevance of the boosting technique using random grids when $\sigma^2> 4a$. This is the purpose of this subsection. We first present the different schemes and then analyse numerically the variance of the correcting term. Then, we present the numerical bias of the approximation $\hat{P}^{2, n}$ for the CIR and Heston models.}

\blue{\subsubsection{The approximation schemes}
In order to perform the numerical tests for $\sigma^2>4a$, we consider two different second order schemes for the CIR process. The first one is the second order scheme~\eqref{Alfonsi_scheme} presented in \cite{AA_MCOM}. More precisely, we define
\begin{equation}\label{Alfonsi_scheme_detailed}
  \varphi_A(x,t,\sqrt{t}N) = \varphi_A^u(x,t,\sqrt{t}N)\mathds{1}_{x \ge K^A_2(t)}  + \varphi_A^d(x,t,\sqrt{t}N)\mathds{1}_{x<K^A_2(t)},
\end{equation}
with 
\begin{align*}
  \varphi_A^u(x,t,\sqrt{t}N) &=\varphi(x,t,-\sqrt{3t})\mathds{1}_{N<\cN^{-1}(1/6)}  + \varphi(x,t,0)\mathds{1}_{\cN^{-1}(1/6)\le N<\Phi_N^{-1}(5/6)} \nonumber \\
  &+\varphi(x,t,\sqrt{3t})\mathds{1}_{N\ge\cN^{-1}(5/6)}, \\
  \varphi_A^d(x,t,\sqrt{t}N) &=\frac{\E[X^x_t]}{2(1-\pi(t,x))}\mathds{1}_{N<\cN^{-1}(1-\pi(t,x))}  + \frac{\E[X^x_t]}{2\pi(t,x)}\mathds{1}_{ N\ge\cN^{-1}(1-\pi(t,x))} ,
\end{align*}
where $\cN$ is the  cumulative distribution function of the standard normal distribution, $\pi(t,x)=\frac{1-\sqrt{1-\frac{\E[X^x_t]^2}{\E[(X^x_t)^2]}}}{2}$ and $K^A_2(t)$ is the function given by~\eqref{threshold_gen} with $A_Y=\sqrt{3}$. Here, we have written the scheme $\varphi_A$ as a function of the starting point $x$, the time step $t$ and the Brownian increment $\sqrt{t} N$. When computing $n\E[\left(f(\hat{X}^{n,1}_T) - f(\hat{X}^{n,0}_T)\right)]$ by Monte-Carlo, we use the same Brownian path to sample $\hat{X}^{n,0}_T$ and $\hat{X}^{n,1}_T$, as explained at the beginning of Section~\ref{Simulations}. Thus, there is a strong dependence between these schemes.
}

\blue{We present also another scheme that corresponds to other choices of $Y$ and $\hat{X}^{x,d}$ in~\eqref{Alfonsi_scheme}. We use a distribution that is pretty similar to a Gaussian distribution over the threshold, and a scaled beta distribution below. Thus, we define
\begin{equation}\label{def_sch_B}
  \varphi_B(x,t,\sqrt{t}N) = \varphi_B^u(x,t,\sqrt{t}N)\mathds{1}_{x\ge K^B_2(t)}  + \varphi_B^d(x,t,\sqrt{t}N)\mathds{1}_{x<K^B_2(t)},
\end{equation}
with 
\begin{align*}
  \varphi_B^u(x,t,\sqrt{t}N) & =\varphi(x,t,-z_2)\mathds{1}_{N\le-c_2} + \varphi(x,t,-z_1)\mathds{1}_{-c_2<N\le-c_1}  + \varphi(x,t,N)\mathds{1}_{-c_1\le N<c_1}  \nonumber \\
  &+\varphi(x,t,z_1)\mathds{1}_{c_1<N\le c_2} + \varphi(x,t,z_2)\mathds{1}_{N>c_2}, \\
  \varphi_B^d(x,t,\sqrt{t}N) & = \frac{\E[X^x_t]}{2\pi(t,x)} (\cN(N))^{\frac{1}{2\pi(t,x)}-1},
\end{align*}
where $z_1=2.7523451704710586$, $z_2 = 3.5$, $c_1=2.58$, $c_2= 3.106520327375868$, and  $K^B_2(t)$ is the function given by~\eqref{threshold_gen} with $A_Y=3.5$. Here, we have fixed the values of $c_1$ and $z_2$, and we have numerically calculated $c_2$ and $z_1$ to have $\E[Y^2]=\E[N^2]$ and $\E[Y^4]=\E[N^4]$ with 
$$Y= -z_2 \mathds{1}_{N\le-c_2} -z_1 \mathds{1}_{-c_2<N\le -c_1} +N \mathds{1}_{-c_1<N\le c_1}+z_1 \mathds{1}_{c_1<N\le c_2}+ z_2\mathds{1}_{c_2<N}. $$
The random variable $\varphi_B^d(x,t,\sqrt{t}N)$ has the same two first moments as $X^x_t$, and we can prove following the same arguments as~\cite[Theorem 2.8]{AA_MCOM} that $\varphi_B(x,t,\sqrt{t}N)$ is a second order scheme for the weak error.}

\blue{\subsubsection{Numerical study of the variance of the correcting term $n\left(f(\hat{X}^{n,1}_T) - f(\hat{X}^{n,0}_T)\right)$}
We now analyse the variance of the corrections terms of the correcting term $n\left(f(\hat{X}^{n,1}_T) - f(\hat{X}^{n,0}_T)\right)$ in function of the number $n$ of discretization steps, when we use the different schemes~\eqref{Alfonsi_scheme_detailed} and~\eqref{def_sch_B}. We start with an example with $\sigma^2<4a$ for which $\varphi$ is still defined and $\varphi_A$ (resp. $\varphi_B$) does not use the auxiliary scheme $\varphi_A^d$ (resp. $\varphi_B^d$) since $K_2^{A}(t)=K_2^{B}(t)=0$ in this case. We observe in Table \ref{Table_VAR_CIR1} that the scheme~$\varphi_A$ leads to a value of $\Var(n(f(\hat{X}^{n,1}_T) - f(\hat{X}^{n,0}_T)))$ that is more than 20 times as large as that the one obtained using $\varphi$. Besides, the variance given by the scheme  $\varphi_A$ increases quite linearly with $n$, while the one obtained with $\varphi$ seems to be bounded and to decrease with $n$. One heuristic explanation for this is that  $\varphi_A$ is discrete scheme, which increases the strong error between the scheme on the fine grid~$\Pi^1$ and the scheme on the coarse grid~$\Pi^0$. Considering the scheme $\varphi_B$ that mixes Gaussian and discrete distributions leads to a much smaller variance that is rather close to the one of the scheme $\varphi$. However, as $n$ gets large, we see that the variance does not decrease in contrast to the scheme $\varphi$.
\begin{table}[h!]
  \begin{tabular}{ c c|c|c|c|c|c|  }
    %\hline
    %\multicolumn{5}{|c|}{Execution Time}            \\
    \cline{3-7}
        & & $n=2$     & $n=4$      & $n=8$       & $n=16$    & $n=32$    \\
    \hline
    \multicolumn{1}{|c|}{\multirow{2}{*}{$\varphi$} } & $\sigma^2_{4}(n)$    & 23.86e-4 & 17.43e-4 & 9.35e-4 & 4.85e-4  & 2.49e-4 \\
    \multicolumn{1}{|c|}{}& 95\% prec.& 3.2e-6 & 3.7e-6 & 2.8e-6 & 2.1e-6 & 1.5e-6 \\ 
    \hline
    \multicolumn{1}{|c|}{\multirow{2}{*}{$\varphi_A$}} & $\sigma^2_{4}(n)$     & 4.807e-2 & 10.870e-2 & 22.493e-2 & 45.437e-2 & 91.219e-2 \\ 
    \multicolumn{1}{|c|}{}& 95\% prec.     & 2.4-5 & 5.2e-5 & 11.1e-5 & 22.9e-5 & 46.3e-5 \\
    \hline
    \multicolumn{1}{|c|}{\multirow{2}{*}{$\varphi_B$} } & $\sigma^2_{4}(n)$    & 24.17e-4 & 18.37e-4 & 11.78e-4 & 10.27e-4  & 13.85e-4 \\
    \multicolumn{1}{|c|}{}& 95\% prec.& 3.2e-6 & 3.7e-6 & 2.9e-6 & 3.0e-6 & 4.5e-6 \\ 
    \hline
  \end{tabular}
  \caption{$\sigma^2_{4}(n) = \Var\big(n(f(\hat{X}^{n,1}_T) - f(\hat{X}^{n,0}_T))\big)$ for the different schemes, with $10^8$ samples and 95\% confidence interval precision.
  Test function: $f(x)=\exp(-10x)$. Parameters: $x=0.2$, $a=0.2$, $k=0.5$, $\sigma=0.5$, $T=1$ ($\frac{\sigma^2}{2a}=0.625$).}\label{Table_VAR_CIR1}
\end{table}
}

\blue{We now consider a case with $\sigma^2>4a$ so that the schemes $\varphi_A$ and $\varphi_B$ switch around their threshold. The scheme $\varphi$ is no longer defined. 
In Table~\ref{Table_VAR_CIR2}, we observe a huge increase of the variance in time steps with respect to Table~\ref{Table_VAR_CIR1}. We now observe that the variances grow almost linearly with respect to~$n$. Again, this can be explained heuristically by the switching that increases the strong error between the schemes on the fine grid~$\Pi^1$ and the coarse grid~$\Pi^0$. The rather high values of the variance obtained with the scheme $\varphi_A$ makes the boosting technique using random grids less interesting in practice from a computational point of view. In contrast, the scheme $\varphi_B$ produces much lower variances and the Monte-Carlo estimator of $\hat{P}^{2,n}f$ is more competitive. 
\begin{table}[h!]
  \begin{tabular}{ c c|c|c|c|c|c|  }
    %\hline
    %\multicolumn{5}{|c|}{Execution Time}            \\
    \cline{3-7}
      & & $n=2$     & $n=4$      & $n=8$       & $n=16$    & $n=32$    \\
      \hline
      \multicolumn{1}{|c|}{\multirow{2}{*}{$\varphi_A$} } 
      & $\sigma^2_{4}(n)$   & 0.0927 & 0.8742 & 2.7966 & 7.9095  & 21.6793 \\ 
      \multicolumn{1}{|c|}{}& 95\% prec.& 5.3e-5 & 3.3e-4 & 1.6e-3 & 6.1e-3 & 2.1e-2 \\ 
      \hline
      \multicolumn{1}{|c|}{\multirow{2}{*}{$\varphi_B$} } 
      & $\sigma^2_{4}(n)$   & 0.0757 & 0.2184 & 0.5145 & 1.1892  & 2.6600 \\ 
      \multicolumn{1}{|c|}{}& 95\% prec.& 6.4e-5 & 1.8e-4 & 5.5e-4 & 1.9e-3 & 6.2e-3 \\ 
      \hline
  \end{tabular}
  \caption{$\sigma^2_{4}(n) = \Var\big(n(f(\hat{X}^{n,1}_T) - f(\hat{X}^{n,0}_T))\big)$ with $10^8$ samples and 95\% confidence interval precision.
  Test function: $f(x)=\exp(-10x)$. Parameters: $x=0.2$, $a=0.2$, $k=0.5$, $\sigma=1.5$, $T=1$ ($\frac{\sigma^2}{2a}=5.625$).}\label{Table_VAR_CIR2}
\end{table}
}

\blue{\subsubsection{Numerical Convergence for the CIR}
\begin{figure}[h!]
  \centering
  \begin{subfigure}[h]{0.49\textwidth}
    \centering
    \includegraphics[width=\textwidth]{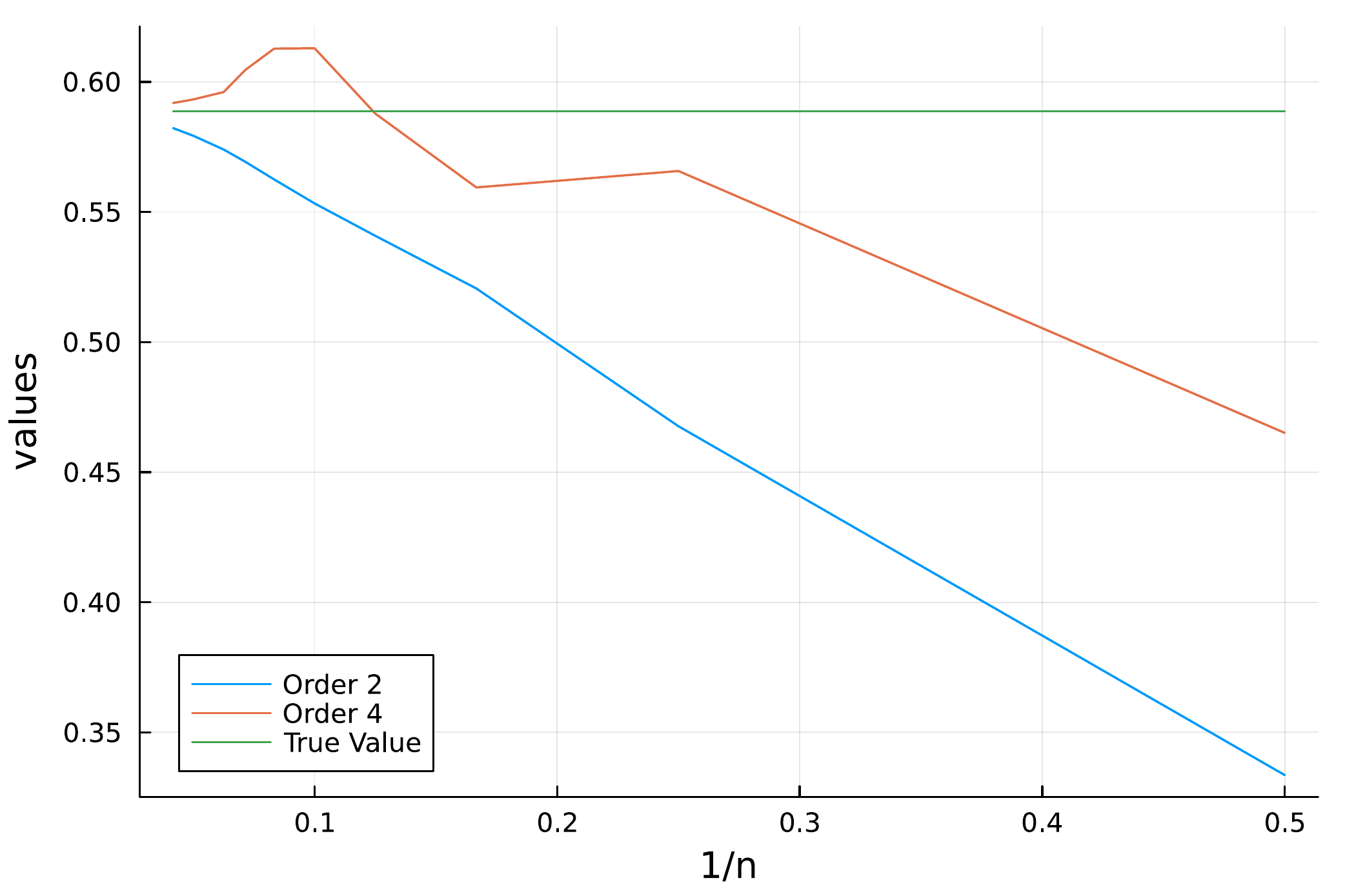}
    \caption{Values plot, scheme $\varphi_A$}
    \label{fig:values_plot_cirAA1}
  \end{subfigure}
  \hfill
  \begin{subfigure}[h]{0.49\textwidth}
    \centering
    \includegraphics[width=\textwidth]{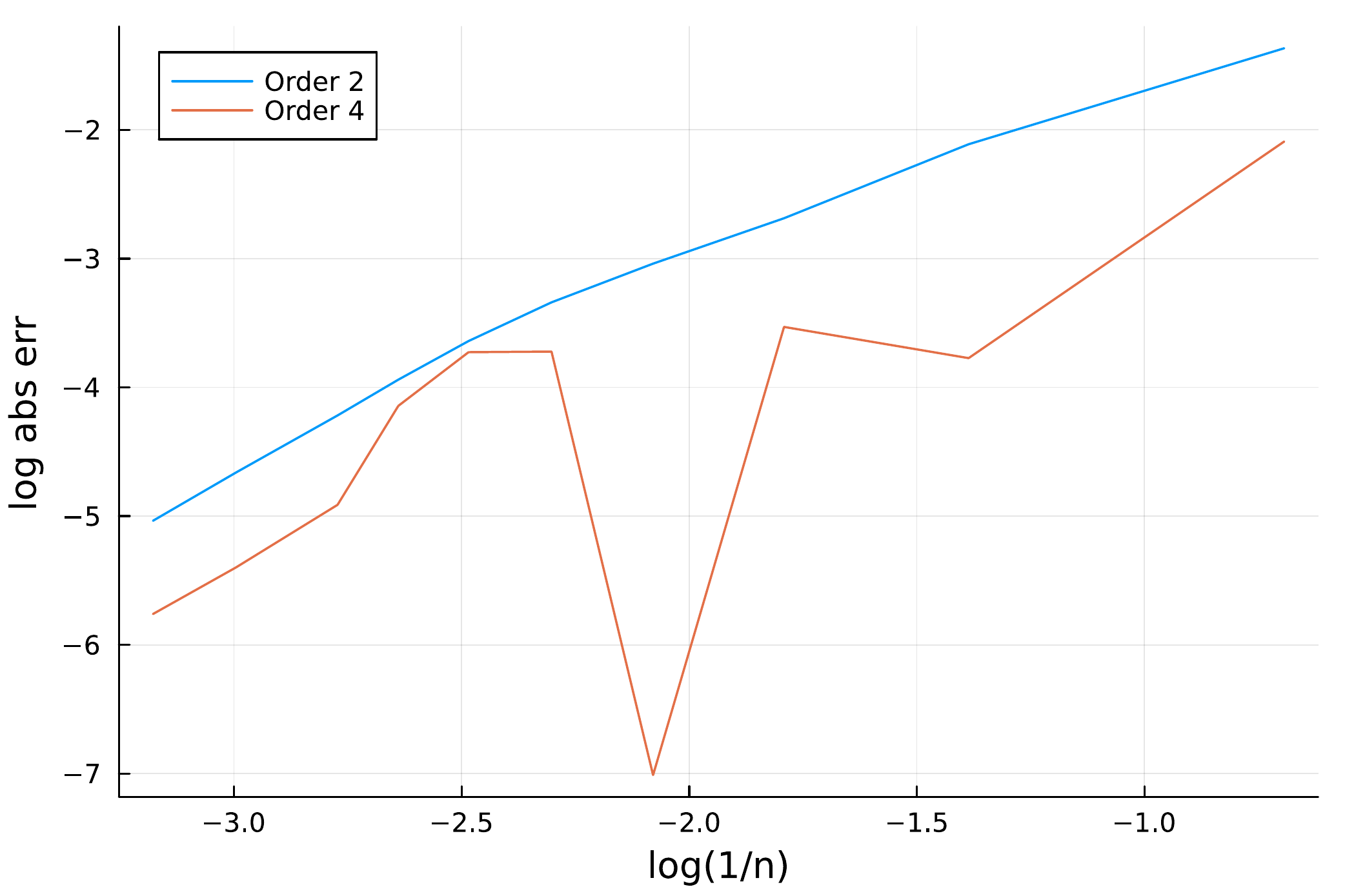}
    \caption{Log-log plot, scheme $\varphi_A$}
    \label{fig:log-log_plot_cirAA1}
  \end{subfigure}
  \begin{subfigure}[h]{0.49\textwidth}
    \centering
    \includegraphics[width=\textwidth]{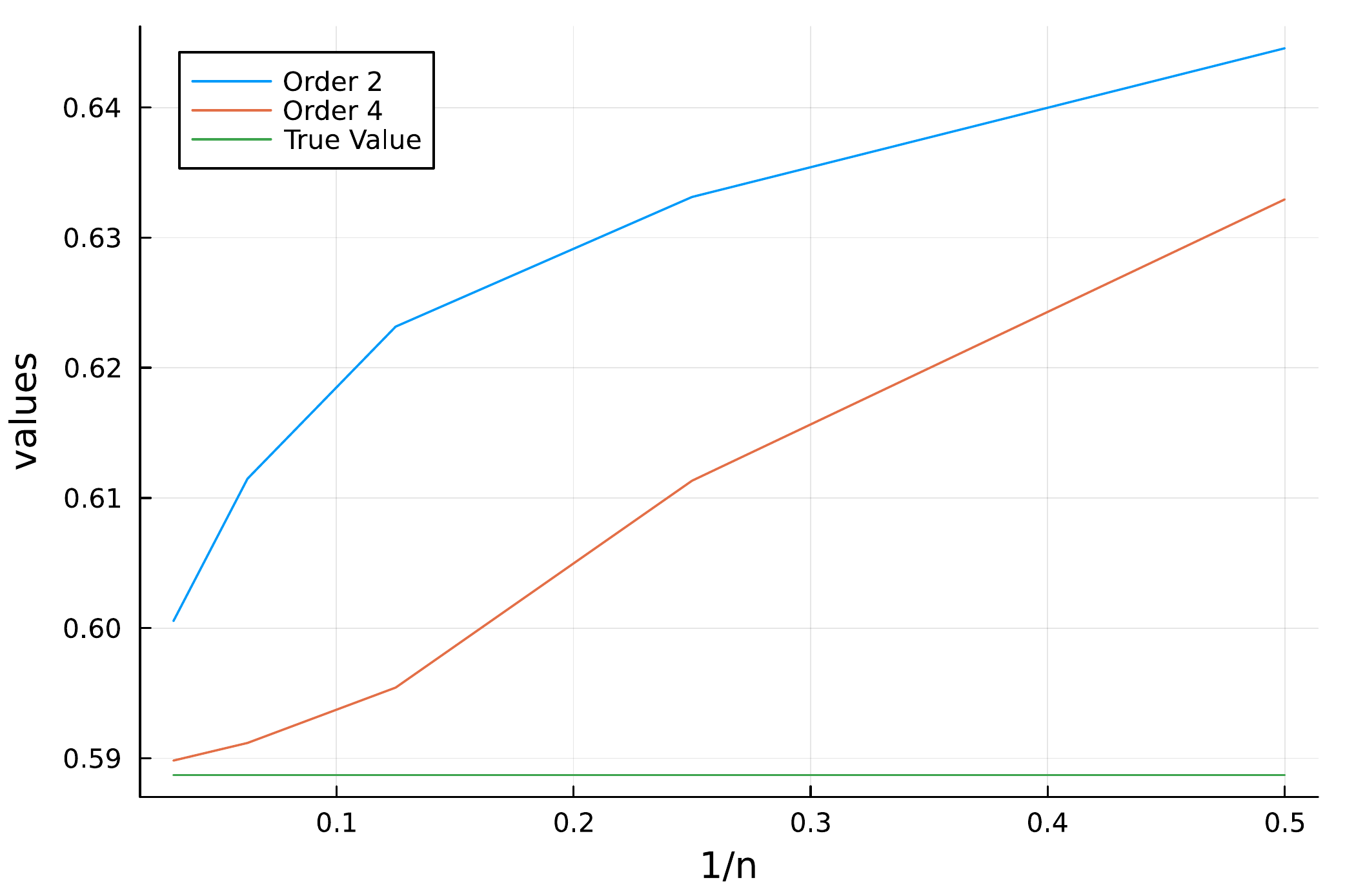}
    \caption{Values plot, scheme $\varphi_B$}
    \label{fig:values_plot_cirAE1}
  \end{subfigure}
  \hfill
  \begin{subfigure}[h]{0.49\textwidth}
    \centering
    \includegraphics[width=\textwidth]{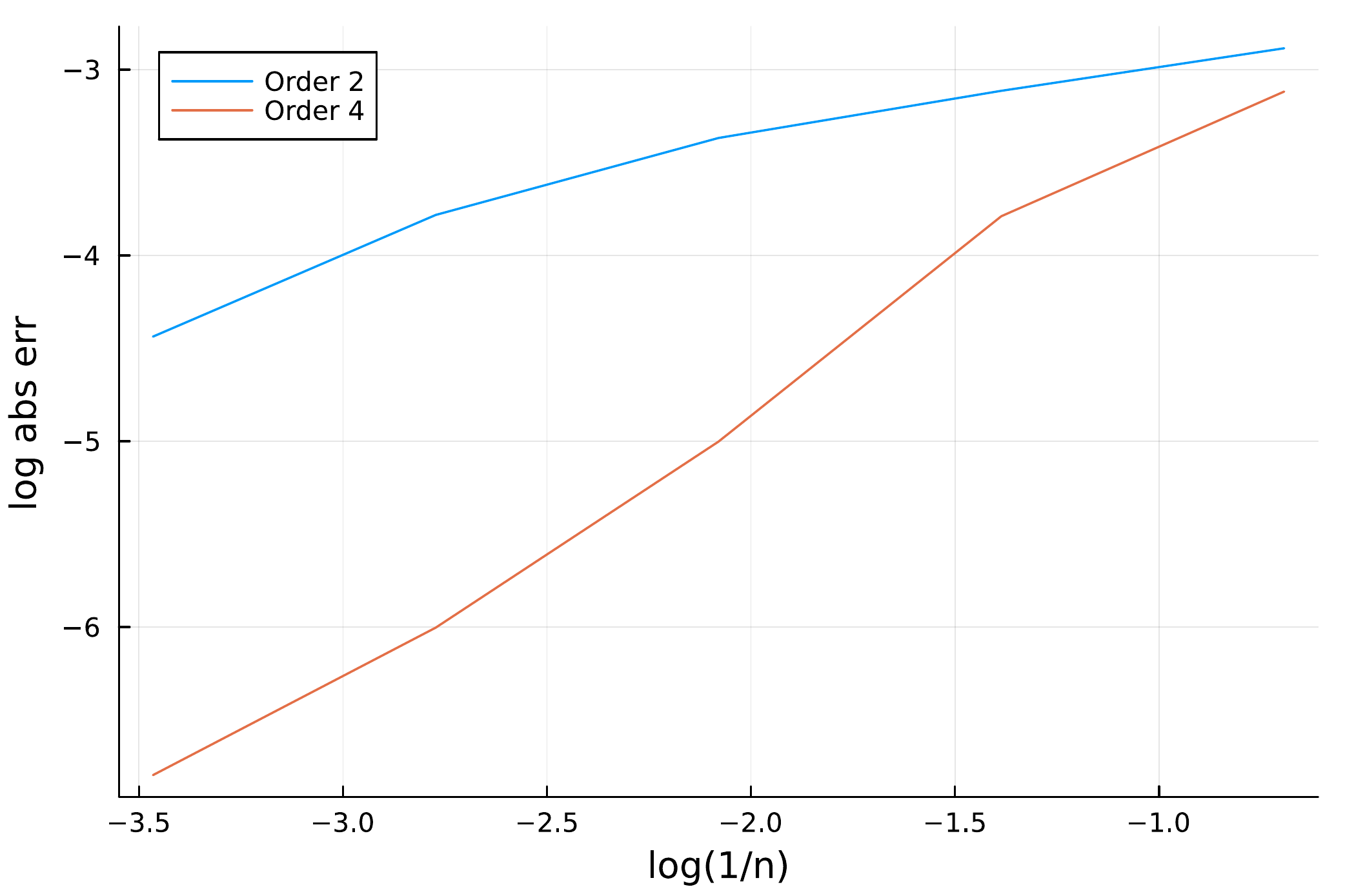}
    \caption{Log-log plot, scheme $\varphi_B$}
    \label{fig:log-log_plot_cirAE1}
  \end{subfigure}
  \caption{Test function: $f(x)=\exp(-10x)$. Parameters: $x=0.2$, $a=0.2$, $k=0.5$, $\sigma=1.5$, $T=1$ ($\frac{\sigma^2}{2a}=5.625$). Statistical precision $\varepsilon=5$e-5.
  Left graphics show the values of $\hat{P}^{1,n}f$, $\hat{P}^{2,n}f$ as a function of the time step $1/n$  and the exact value. Right graphics draw $\log(|\hat{P}^{i,n}f-P_Tf|)$ in function of $\log(1/n)$:  for the scheme $\varphi_A$ (resp. $\varphi_B$) the regressed slopes are 1.47 (resp. 0.54) and 1.14 (resp. 1.38) for the second and fourth order respectively.}\label{CIRAA_orders}
\end{figure}
}

\blue{We have plotted in Figure \ref{CIRAA_orders} the convergence of the estimators of the Monte-Carlo estimators $\hat{P}^{1,n}f$ and $\hat{P}^{2,n}f$ for the schemes $\varphi_A$ and $\varphi_B$.  We note that in all our experiments, $\hat{P}^{2,n}f$ gives a better approximation than $\hat{P}^{1,n}f$, though there is no theoretical guarantee of that. However, the improvement is not as good as for $\sigma^2\le 4a$. 
We know that $\hat{P}^{1,n}f$ leads to an asymptotic weak error of order~2: the estimated rate of convergence obtained by regression are below since we consider rather small values of~$n$ and are not in the asymptotic regime. We have instead no theoretical guarantee that $\hat{P}^{2,n}f$ gives an asymptotic weak error of order~4. The estimated rates are quite far from this value, indicating that a fourth order of convergence may not hold. To sum up,   even if $\hat{P}^{2,n}f$ is still more accurate than $\hat{P}^{1,n}f$ for $\sigma^2>4a$, it does not lead to obvious computational gains. 
}

\blue{\subsubsection{Simulations in the Heston model}
We present now some numerical tests for Heston model and consider three different schemes that are well defined for any $\sigma\ge 0$:
\begin{itemize}
\item $\Phi_A$ is the scheme~\eqref{H2S} where $\varphi_A(x,h,W_h)$ is used instead of 
$\varphi(x,h,W_h)$,
\item $\Phi_B$ is the scheme~\eqref{H2S} where $\varphi_B(x,h,W_h)$ is used instead of $\varphi(x,h,W_h)$,
\item $\Phi_E$ is the scheme~\eqref{H2S} where the exact scheme $X^x_h$ (see, e.g.~\cite[Proposition 3.1.1]{AA_book}) is used instead of $\varphi(x,h,W_h)$.
\end{itemize}
We start by comparing the variance of the correcting terms with the different schemes.
 In Table~\ref{Table_VAR_Heston1}, we consider a case with $\sigma^2<4a$ and also include the variance for the scheme~$\Phi$ given by~\eqref{H2S}. We remark that the variances of the correction term for the standard scheme $\Phi$ and for the scheme $\Phi_E$ appear to be bounded. In contrast, the variance for the schemes~$\Phi_A$ and $\Phi_B$ tends to increase with~$n$: the variance is very high for $\Phi_A$ while the one produced by $\Phi_B$ remains close to the one of $\Phi$ and $\Phi_E$.  Table~\ref{Table_VAR_Heston2} deals with a case with $\sigma^2>4a$ for which variances are much higher. We observe an approximately linear growth of the variance of the correction term for the schemes $\Phi_A$ and $\Phi_B$. The variance produced by the scheme $\Phi_E$ also increases, but in a much moderate way.  
\begin{table}[h!]
  \begin{tabular}{ c c|c|c|c|c|c|  }
    %\hline
    %\multicolumn{5}{|c|}{Execution Time}            \\
    \cline{3-7}
        & & $n=2$     & $n=4$      & $n=8$       & $n=16$    & $n=32$    \\
        \hline
        \multicolumn{1}{|c|}{\multirow{2}{*}{$\Phi$} } & $\sigma^2_{4}(n)$   & 33.252 & 41.962 & 46.159 & 48.273  & 49.385 \\ 
        \multicolumn{1}{|c|}{}& 95\% prec.& 0.024 & 0.029 & 0.033 & 0.035 & 0.037 \\ 
    \hline
    \multicolumn{1}{|c|}{\multirow{2}{*}{$\Phi_A$}} & $\sigma^2_{4}(n)$     & 450.95 & 973.82 & 1976.53 & 3984.64 & 8014.19 \\ 
    \multicolumn{1}{|c|}{}& 95\% prec.     & 0.20 & 0.40 & 0.83 & 1.70 & 3.47 \\
    \hline
        \multicolumn{1}{|c|}{\multirow{2}{*}{$\Phi_B$} } & $\sigma^2_{4}(n)$   & 33.702 & 43.116 & 48.606 & 53.373  & 59.760 \\ 
        \multicolumn{1}{|c|}{}& 95\% prec.& 0.025 & 0.031 & 0.037 & 0.044 & 0.059 \\ \hline
        \multicolumn{1}{|c|}{\multirow{2}{*}{$\Phi_E$}} & $\sigma^2_{4}(n)$     & 51.99 & 53.93 & 52.46 & 51.47 &  50.99 \\     
        \multicolumn{1}{|c|}{}& 95\% prec.     & 0.032 & 0.034 & 0.036 & 0.037 & 0.037 \\
        \hline  
      \end{tabular}
  \caption{$\sigma^2_{4}(n) = \Var\big(n(f(\hat{X}^{n,1}_T,\hat{S}^{n,1}_T) - f(\hat{X}^{n,0}_T,\hat{S}^{n,0}_T))\big)$ with $10^8$ samples and 95\% confidence interval precision.
  Test function: $f(x,s)=(K-s)^+$. Parameters: $S_0=100$, $r=0$, $x=0.2$, $a=0.2$, $k=1.0$, $\sigma=0.5$, $\rho=-0.7$, $T=1$, $K=105$ ($\frac{\sigma^2}{2a}=0.625$).}\label{Table_VAR_Heston1}
\end{table}
\begin{table}[h!]
  \begin{tabular}{ c c|c|c|c|c|c|  }
    \cline{3-7}
        & & $n=2$     & $n=4$      & $n=8$       & $n=16$    & $n=32$    \\
    \hline
    \multicolumn{1}{|c|}{\multirow{2}{*}{$\Phi_A$}} & $\sigma^2_{4}(n)$     & 799.93 & 2568.43 & 6384.48 & 14588.23 & 29798.4266 \\ 
    \multicolumn{1}{|c|}{}& 95\% prec.     & 0.58 & 1.93 & 5.88 & 16.63 & 42.38 \\
    \hline
    \multicolumn{1}{|c|}{\multirow{2}{*}{$\Phi_B$}} & $\sigma^2_{4}(n)$     & 306.87 & 581.70 & 958.06 & 1729.18 & 3185.83 \\ 
    \multicolumn{1}{|c|}{}& 95\% prec.     & 0.18 & 0.38 & 0.90 & 2.65 & 8.25 \\
    \hline
    \multicolumn{1}{|c|}{\multirow{2}{*}{$\Phi_{E}$} } & $\sigma^2_{4}(n)$    & 233.89 & 287.50 & 314.03 & 331.31 & 344.20 \\ 
    \multicolumn{1}{|c|}{}& 95\% prec.& 0.14 & 0.20 & 0.24 & 0.27 & 0.29 \\ 
\hline
  \end{tabular}
  \caption{$\sigma^2_{4}(n) = \Var\big(n(f(\hat{X}^{n,1}_T,\hat{S}^{n,1}_T) - f(\hat{X}^{n,0}_T,\hat{S}^{n,0}_T))\big)$ with $10^8$ samples and 95\% confidence interval precision.
  Test function: $f(x,s)=(K-s)^+$. Parameters: $S_0=100$, $r=0$, $x=0.2$, $a=0.2$, $k=1.0$, $\sigma=1.5$, $\rho=-0.7$, $T=1$, $K=105$ ($\frac{\sigma^2}{2a}=5.625$).}\label{Table_VAR_Heston2}
\end{table}
}

\blue{We now turn to the convergence of the Monte-Carlo estimators. We have plotted in Figure~\ref{HestonAA_orders2}, for the same set of parameters as in Table~\ref{Table_VAR_Heston2}, the behavior of $\hat{P}^{1,n}f$ and $\hat{P}^{2,n}f$ for the schemes $\Phi_B$ and $\Phi_E$. We have discarded the scheme $\Phi_A$ that produces a too large variance for the correcting term. As for the CIR diffusion, we  note that $\hat{P}^{2,n}f$ gives a better approximation than $\hat{P}^{1,n}f$ but the bias does not seem to be of order~$4$. For the scheme $\Phi_B$, the improvement is moderate, and do not really compensate the computational effort of calculating the correcting term. Instead, for the scheme $\Phi_E$, the improvement is rather significant, making the approximation $\hat{P}^{2,n}f$ interesting from a computational point of view with respect to~$\hat{P}^{1,n}f$. Also, the estimated rate of convergence is much higher and not so far from~$4$. A dedicated theoretical study of $\hat{P}^{2,n}f$ with the scheme $\Phi_E$ is left for further research. 
\begin{figure}[h!]
  \centering
  \begin{subfigure}[h]{0.49\textwidth}
    \centering
    \includegraphics[width=\textwidth]{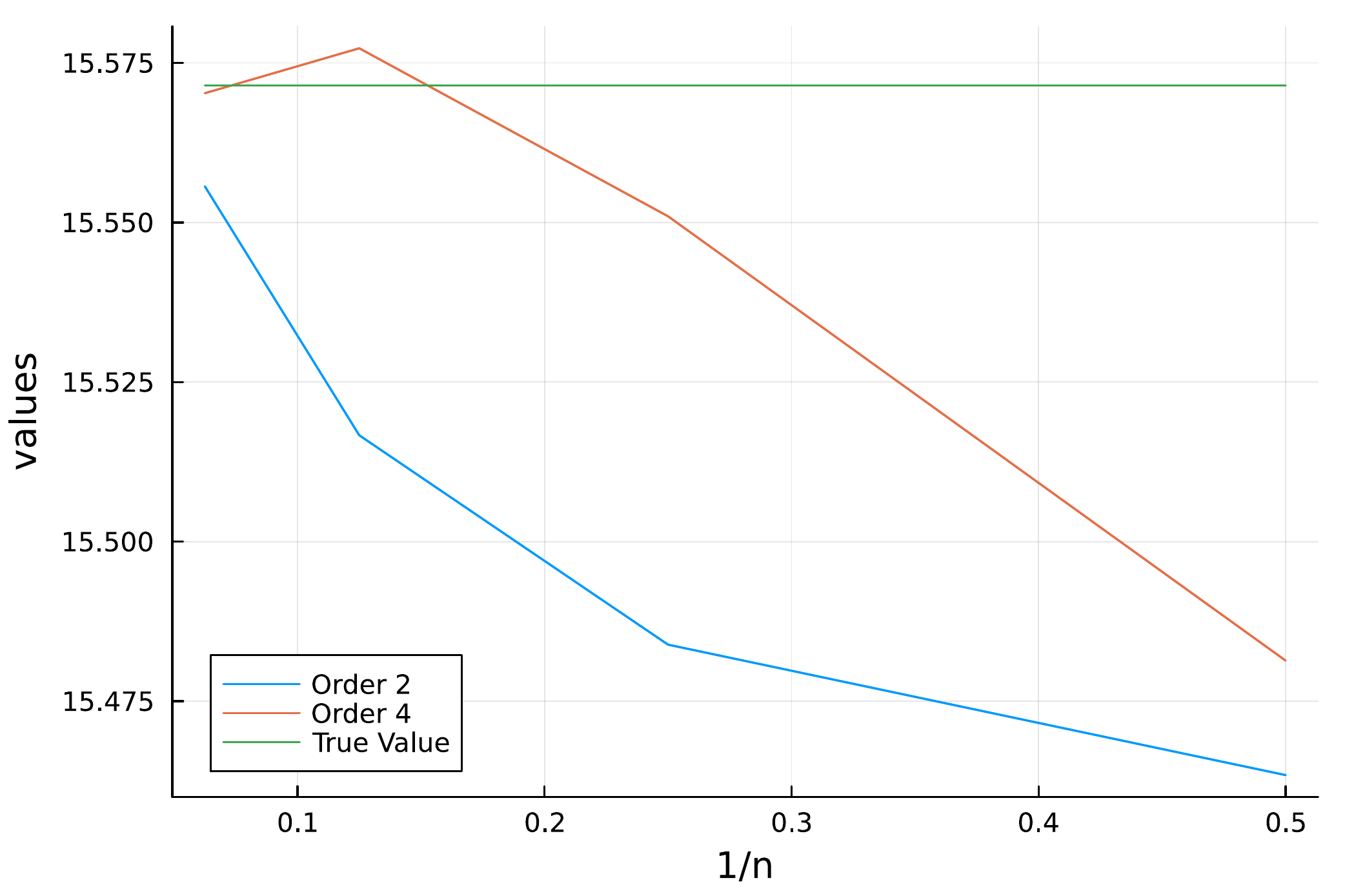}
    \caption{Values plot, scheme $\Phi_B$}
    \label{fig:values_plot_hestonAA2}
  \end{subfigure}
  \hfill
  \begin{subfigure}[h]{0.49\textwidth}
    \centering
    \includegraphics[width=\textwidth]{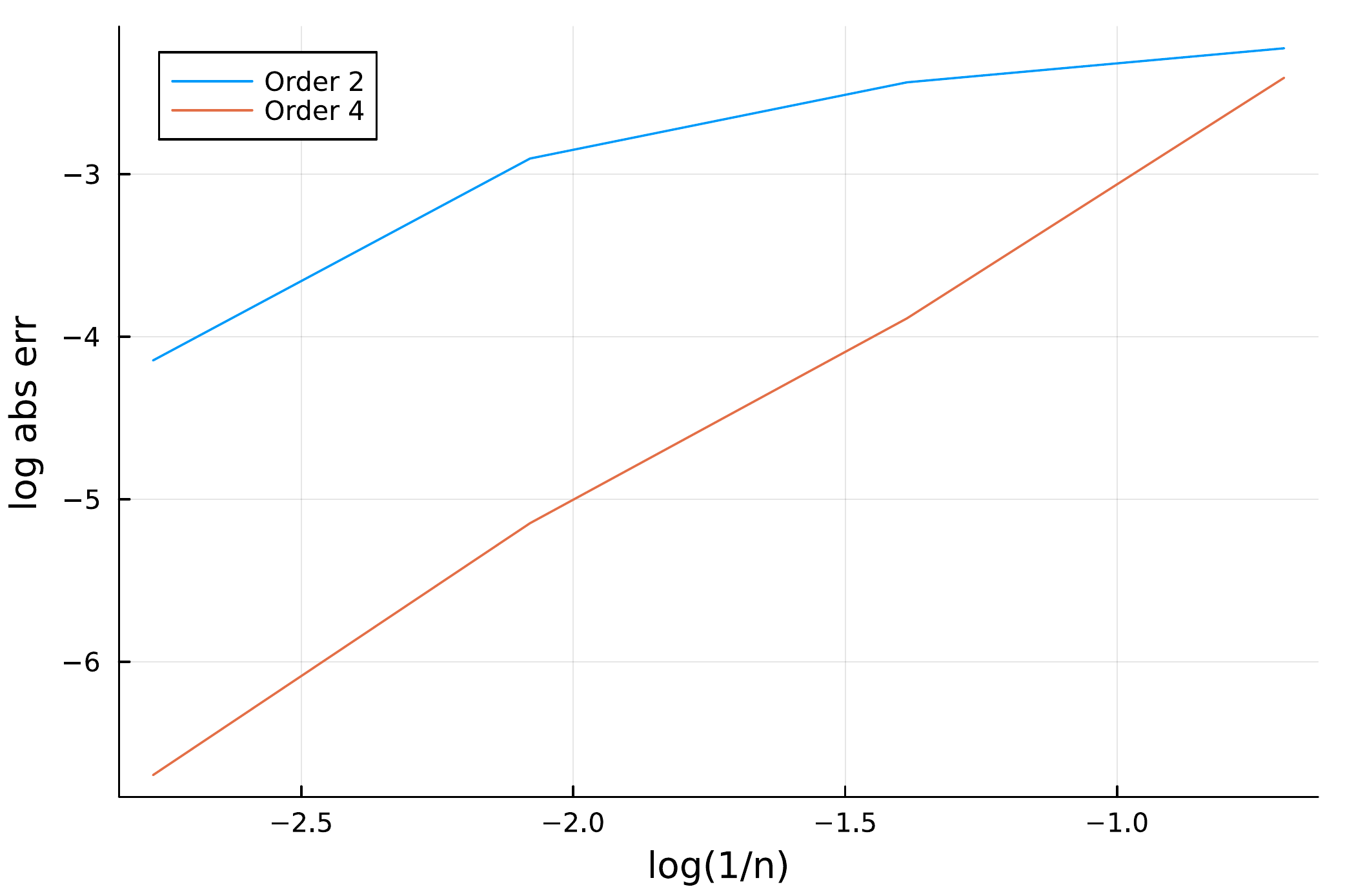}
    \caption{Log-log plot, scheme $\Phi_B$}
    \label{fig:log-log_plot_hestonAA2}
  \end{subfigure}
  \begin{subfigure}[h]{0.49\textwidth}
    \centering
    \includegraphics[width=\textwidth]{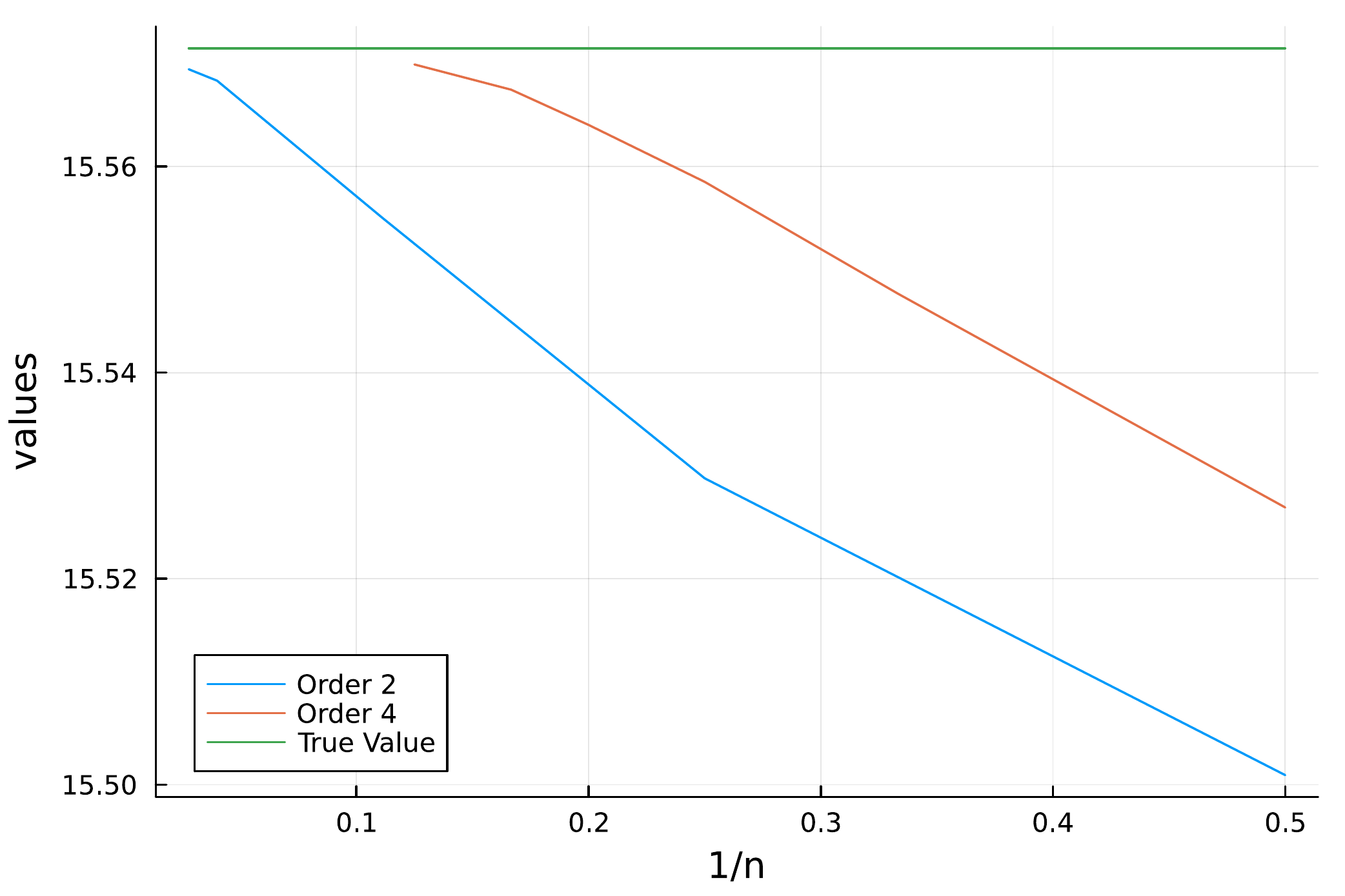}
    \caption{Values plot, scheme $\Phi_E$}
    \label{fig:values_plot_hestonCF1}
  \end{subfigure}
  \hfill
  \begin{subfigure}[h]{0.49\textwidth}
    \centering
    \includegraphics[width=\textwidth]{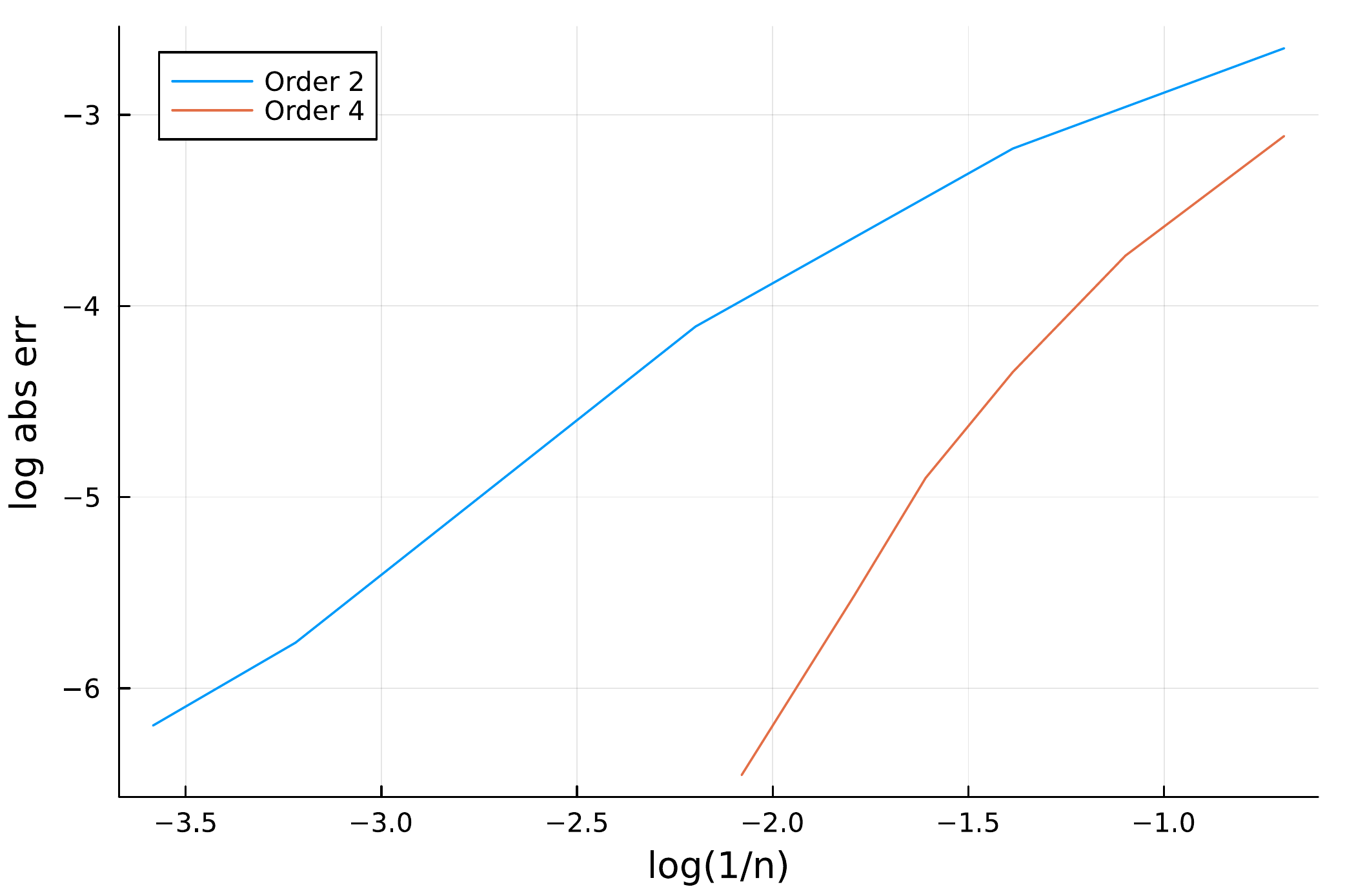}
    \caption{Log-log plot, scheme $\Phi_E$}
    \label{fig:log-log_plot_hestonCF1}
  \end{subfigure}
  \caption{Test function: $f(x,s)=(K-s)^+$. Parameters: $S_0=100$, $r=0$, $x=0.2$, $a=0.2$, $k=1$, $\sigma=1.5$, $\rho=-0.7$, $T=1$, $K=105$ ($\frac{\sigma^2}{2a}=5.625$). Statistical precision $\varepsilon=5$e-4.
  Left graphics show the values of $\hat{P}^{1,n}f$, $\hat{P}^{2,n}f$ as a function of the time step $1/n$  and the exact value. Right graphics draw $\log(|\hat{P}^{i,n}f-P_Tf|)$ in function of $\log(1/n)$: for the scheme $\Phi_B$ (resp. $\Phi_E$) the regressed slopes  are 0.90 (resp. 1.28) and 2.04 (resp. 2.40) for the second and fourth order respectively.}\label{HestonAA_orders2}
\end{figure} 
}

\appendix

\section{Proofs of Section~\ref{Sec_pol_fct}}\label{App_proof_sec_pol}
\begin{proof}[Proof of Lemma~\ref{estimates_pol}.]
  \textit{(1)} Let $f\in \PLRp{L}$.  We have $X_0(t,x)^j=\sum_{i=0}^j {j\choose i}((a-\sigma^2/4)\psi_k(t))^{j-i}e^{-kti}x^i$ and thus
  \begin{align*}
    f(X_0(t,x)) & =\sum_{j=0}^L a_j \sum_{i=0}^j {j\choose i}((a-\sigma^2/4)\psi_k(t))^{j-i}e^{-kti} x^i. 
  \end{align*}
  Therefore, $f(X_0(t,\cdot)) \in \PLRp{L}$ and we have
  $$\|f(X_0(t,\cdot)) \|\le \sum_{j=0}^L |a_j| \sum_{i=0}^j {j\choose i}(|a-\sigma^2/4|\psi_k(t))^{j-i}e^{-kti}   =\sum_{j=0}^L |a_j| \tilde{X}_0(t)^j, $$
  with $ \tilde{X}_0(t)=e^{-kt}+  |a-\sigma^2/4|\psi_k(t)$. For $k\ge 0$, we have $0\leq\psi_k(t)\leq t$ and thus $\tilde{X}_0(t)\le (1+|a-\sigma^2/4|t)$. For $k<0$, we have $\tilde{X}_0(t)=e^{-kt}(1+  |a-\sigma^2/4|\psi_{-k}(t))\le e^{-kt}(1+  |a-\sigma^2/4|t )$. Since $(1+  |a-\sigma^2/4|t )^L\le 1+t\sioj {j\choose i}  |a-\sigma^2/4|^i  (1\vee T)^i\le 1+t(1+  |a-\sigma^2/4|(1\vee T) )^L$, we get $\tilde{X}_0(t)^j\le (1\vee e^{-k Lt})[ 1+t(1+  |a-\sigma^2/4|(1\vee T) )^L]$ for $j\in\{0,\dots,L\}$ and then
  $$\|f(X_0(t,\cdot))\|\leq (1\vee e^{-kLt})(1+(1+|a-\sigma^2/4|(1\vee T))^Lt) \|f \|,$$ which gives the claim with $C_{X_0}=1+|a-\sigma^2/4|(1\vee T)$.\\

  \textit{(2)} Since $Y$ is a symmetric random variable, we have
  \begin{align*}
    \E[f(X_1(\sqrt{t}Y,x ))] & = \sjzL a_j \E[X_1(\sqrt{t}Y,x)^j] = \sjzL a_j \sum_{i=0}^{2j} {2j\choose i}  \bigg(\frac{\sigma \sqrt{t}}{2}\bigg)^{2j-i}\E[Y^{2j-i}]x^{i/2} \\
                             & =\sjzL a_j \sum_{i=0}^j {2j\choose 2i}  \bigg(\frac{\sigma^2 t}{4}\bigg)^{j-i}\E[Y^{2(j-i)}]x^i                                               \\
                             & =\sjzL a_j x^j + t\sjzL a_j \sum_{i=0}^{j-1} {2j\choose 2i}  \bigg(\frac{\sigma^2 }{4}\bigg)^{j-i} t^{j-i-1}\E[Y^{2(j-i)}]x^i.
  \end{align*}
  This proves that $\E[f(X_1(\sqrt{t}Y,\cdot))]\in \PLRp{L}$. We note that $\E[Y^{2j}]\ge 1$ by Hölder inequality since $\E[Y^2]=1$, and thus $\E[Y^{2j}]\le \E[Y^{2L}]$ for $j\in \{0,\dots,L\}$.  We get
  \begin{align*}
    \| f(X_1(\sqrt{t}Y, \cdot ))\| & \le \|f\| +t \E[Y^{2L}]\sjzL |a_j| \sum_{i=0}^{j-1} {2j\choose 2i}  \bigg(\frac{\sigma^2 }{4}\bigg)^{j-i} (1\vee T)^{j-i} \\
                                   & \le \|f\|\left(1+t  \E[Y^{2L}]\left(1+\frac \sigma 2 \sqrt{1\vee T}\right)^{2L}\right),
  \end{align*}
  since $\sum_{i=0}^{j-1} {2j\choose 2i}  \left(\frac{\sigma^2 }{4}\right)^{j-i} (1\vee T)^{j-i}\le \left(1+\frac \sigma 2 \sqrt{1\vee T}\right)^{2j} \le \left(1+\frac \sigma 2 \sqrt{1\vee T}\right)^{2L} $. This gives the claim with $C_{X_1}=\left(1+\frac \sigma 2 \sqrt{1\vee T}\right)^{2}$.
\end{proof}

\begin{proof}[Proof of Lemma~\ref{Moments_Formula_CIR}.]
  We have $\tilde{u}_0(t,x)=1$, and in the case $m=1$, we have $\tilde{u}_1(t,x)=x +\int_0^t(a-k \tilde{u}_1(s,x))ds$ that has the solution:
  \begin{equation*}
    \tilde{u}_1(t,x) = x e^{-kt} + a \psi_k(t)
  \end{equation*}
  where $\psi_k(t)=\frac{1-e^{-kt}}{k}$ if $k\neq 0$ and $\psi_k(t)=t$ otherwise.
  This gives the claim for $m=1$ with $\tilde{u}_{0,1}=a \psi_k(t)$ and  $\tilde{u}_{1,1}=e^{-kt}$. We then prove the result by induction and consider $m\geq 2$. Using It\^o formula and taking the expected value, one has $\partial_t \tilde{u}_m(t,x) = (am+\sigma^2 m(m-1)/2) \tilde{u}_{m-1}(t,x)-km\tilde{u}_m(t,x)$. Hence, we have
  \begin{equation*}
    \tilde{u}_m(t,x) = (e^{-kt})^m\left( x^m + \int_0^t (am+\sigma^2 m(m-1)/2)(e^{ks})^m  \tilde{u}_{m-1}(s,x) ds\right),
  \end{equation*}
  and we get the following induction relations that give us the representation \eqref{Pol_Estimate_CIR}
  \begin{align*}
    \begin{cases}\tilde{u}_{j,m}(t) = (e^{-kt})^m \int_0^t (am+\sigma^2 m(m-1)/2)(e^{ks})^m  \tilde{u}_{j,m-1}(s) ds, \  0\le j\le m-1, \\
      \tilde{u}_{m,m}(t) = (e^{-kt})^m.
    \end{cases}
  \end{align*}

  Let $f\in \PLRp{L}$. We clearly get from the preceding result that $\E[f(X^\cdot_t)]\in \PLRp{L}$ and
  $$\|\E[f(X^\cdot_t)]\|\le \sum_{m=0}^L |a_m|\sum_{j=0}^m|\tilde{u}_{j,m}(t)| \le C_{\text{cir}}(L,T) \|f\|. \qedhere $$
\end{proof}

\section{Proofs of Section~\ref{Sec_main}}\label{App_proof_sec_5}

\begin{proof}[Proof of Lemma~\ref{lem_estimnorm}]
    Properties (1)--(3) are straightforward, and we prove only (4)--(6).
    \item[$(4)$] We use the fact that $1+x^L\leq 2(1+x^{L+1})$ for $x\geq 0$ , hence
    $$ \max_{j\in\{0,\ldots, m\}} \sup_{x\geq 0}\frac{|f^{(j)}(x)|}{1+x^{L+1}}\leq 2 \max_{j\in\{0,\ldots, m\}} \sup_{x\geq 0}\frac{|f^{(j)}(x)|}{1+x^L}.$$
    \item[$(5)$] Let $f\in \CpolKL{m}{L}$. We will use the fact that for all $ x\geq 0$, $(1+x)(1+x^L)\leq 3(1+x^{L+1})$ so
    $$\sup_{x\geq 0}\frac{x|f^{(j)}(x)|}{1+x^{L+1}} \leq 3\sup_{x\geq 0}\frac{x}{1+x}\sup_{x\geq 0}\frac{|f^{(j)}(x)|}{1+x^L}= 3\sup_{x\geq 0}\frac{|f^{(j)}(x)|}{1+x^L}.$$  Now, we use the Leibniz rule on $\mathcal{M}_1 f$ and get $(xf(x))^{(j)}=jf^{(j-1)}(x)+xf^{(j)}(x)$, so
    $$\sup_{x\geq 0}\frac{|(xf(x))^{(j)}|}{1+x^{L+1}}\leq j\sup_{x\geq 0}\frac{|f^{(j-1)}(x)|}{1+x^{L+1}} +  \sup_{x\geq 0}\frac{x|f^{(j)}(x)|}{1+x^{L+1}}.$$
    Maximizing both sides on $j\in\{0,\ldots, m\}$ and using the previous inequality gives $\|\mathcal{M}_1 f\|_{m,L+1}\le  m\| f\|_{m-1,L+1} + 3\|f\|_{m,L}$. We get the bound by using properties $(2)$ and $(4)$.

    \item[$(6)$] We have $\|\cL f\|_{m,L+1}\le a \| f' \|_{m,L+1} + (2m+3)[|k|\| f' \|_{m,L}+\frac {\sigma^2} 2 \|f''\|_{m,L}]$ by using the property~(5). We get the estimate by using~(3),~(4) and~(2). The other estimate for $V_1^2/2$ is obtained by taking $a=\sigma^2/4$ and $k=0$, while the one for $V_0$ follows by using the same arguments.
\end{proof}

\begin{proof}[Proof of Lemma~\ref{regular_rep}]
  For $x>0$, we have
  \begin{align*}
    \psi'_g(x) & =\left(1+\frac{\beta}{2\sqrt{x}}\right) g'(x+\beta \sqrt{x}+\gamma)+ \left(1-\frac{\beta}{2\sqrt{x}}\right) g'(x-\beta \sqrt{x}+\gamma)             \\
               & =\underset{\psi_{g'}(x)}{\underbrace{g'(x+\beta \sqrt{x}+\gamma)+g'(x-\beta \sqrt{x}+\gamma)}}+\beta^2 \int_0^1g''(x+\beta(2u-1)\sqrt{x}+\gamma)du,
  \end{align*}
  since $\frac{d}{du}g'(x+\beta(2u-1)\sqrt{x}+\gamma)=2\beta\sqrt{x} g''(x+\beta(2u-1)\sqrt{x}+\gamma)$.
  Clearly, this derivative is continuous at~$0$ which shows that $\psi_g$ is $C^1$.

  We are now in position to prove~\eqref{formula_psign} by induction on~$n$. It is true for $n=0,1$.  We assume that it is true for~$n$. Then, we get by using the case $n=1$, differentiating~\eqref{formula_psign} and an integration by parts for the fourth term:
  \begin{align*}
    \psi_g^{(n+1)}(x)= & \, \psi_{g^{(n+1)}}(x)+\beta^2 \int_0^1g^{(n+2)}(x+\beta(2u-1)\sqrt{x}+\gamma)du                                                                   \\
                       & +     \sum_{j=1}^n\binom{n}{j}\beta^{2j} \left(\int_0^1g^{(n+j+1)}(x+\beta(2u-1)\sqrt{x}+\gamma) \frac{(u-u^2)^{j-1}}{(j-1)!}du \right.            \\
                       & \phantom{+   \sum_{k=1}^n\binom{n}{j}\beta^{2j}} \left.+\beta^2\int_0^1g^{(n+j+2)}(x+\beta(2u-1)\sqrt{x}+\gamma) \frac{(u-u^2)^{j}}{j!}du \right).
  \end{align*}
  We then reorganize the terms as follows
  \begin{align*}
    \psi_g^{(n+1)}(x)   = & \, \psi_{g^{(n+1)}}(x)+(n+1)\beta^2 \int_0^1g^{(n+2)}(x+\beta(2u-1)\sqrt{x}+\gamma)du                                               \\
                          & +\beta^{2n+2}\int_0^1g^{(2n+2)}(x+\beta(2u-1)\sqrt{x}+\gamma) \frac{(u-u^2)^{n}}{n!}du                                              \\
                          & + \sum_{j=2}^n\binom{n}{j}\beta^{2j} \left(\int_0^1g^{(n+j+1)}(x+\beta(2u-1)\sqrt{x}+\gamma) \frac{(u-u^2)^{j-1}}{(j-1)!}du \right) \\
                          & +\sum_{j=1}^{n-1}\binom{n}{j}\beta^{2j+2} \left(\int_0^1g^{(n+j+2)}(x+\beta(2u-1)\sqrt{x}+\gamma) \frac{(u-u^2)^{j}}{j!}du \right).
  \end{align*}
  The last sum is equal to $\sum_{j=2}^{n}\binom{n}{j-1}\beta^{2j} \left(\int_0^1g^{(n+j+1)}(x+\beta(2u-1)\sqrt{x}+\gamma) \frac{(u-u^2)^{j-1}}{(j-1)!}du \right)$ by changing $j$ to $j-1$, and we conclude by using that $\binom{n}{j}+\binom{n}{j-1}=\binom{n+1}{j}$.
\end{proof}

\begin{proof}[Proof of Corollary~\ref{cor_psign}]
  We use~\eqref{formula_psign} with $\gamma=\beta^2/4$. We first notice that
  \begin{align*}
    |\psi_{g^{(n)}}(x)| & \le \|g\|_{n,L}(2+(\sqrt{x}+\beta/2)^{2L}+(\sqrt{x}-\beta/2)^{2L})                      \\
                        & =\|g\|_{n,L}\left(2+2x^L +2 \sum_{i=1}^{L}\binom{2L}{2i}(\beta/2)^{2i}  x^{L-i}\right).
  \end{align*}
  Using that $x^{i}\le 1+x^L$ for $0\le i\le L-1$, we get
  $$ |\psi_{g^{(n)}}(x)|\le  2 \|g\|_{n,L}(1+x^L)  \sum_{i=0}^{L}\binom{2L}{2i}(\beta/2)^{2i}=\|g\|_{n,L}(1+x^L)\big((1+\beta/2)^{2L}+(1-\beta/2)^{2L}\big).$$
  For the other terms, we use that for $u\in[0,1]$, $x\ge 0$ and $j\in\{1,\dots,n\}$,
  \begin{align*}
    |g^{(n+j)}(x+\beta(2u-1)\sqrt{x}+\beta^2/4)| & \le \|g\|_{2n,L}(1+(x+\beta(2u-1)\sqrt{x}+\beta^2/4)^L) \\
                                                 & \le \|g\|_{2n,L}(1+( \sqrt{x}+\beta/2)^{2L}).
  \end{align*}
  We again expand $( \sqrt{x}+\beta/2)^L=x^L+ \sum_{i=1}^{2L}\binom{2L}{i}(\beta/2)^{i}x^{(L-i)/2}$ and use that $x^{(L-i)/2}\le 1+x^L$ to get
  $$|g^{(n+j)}(x+\beta(2u-1)\sqrt{x}+\beta^2/4)|\le \|g\|_{2n,L}(1+\beta/2)^{2L}(1+x^L).$$
  Besides, we have $u-u^2\le 1/4$ for $u\in[0,1]$ and thus $\int_0^1(u-u^2)^jdu\leq \frac{1}{4^{j}}$, which gives $\sup_{x\ge 0}\frac{|\psi_{g^{(n)}}(x)|}{1+x^L}\le \tilde{C}(\beta)$ with
  \begin{align*}
    \tilde{C}(\beta) & =\|g\|_{n,L}\big((1+\beta/2)^{2L}+(1-\beta/2)^{2L}\big)+\|g\|_{2n,L}(1+\beta/2)^{2L}\sum_{j=1}^n \binom{n}{j}\bigg(\frac{\beta^2}{4}\bigg)^j \\
                     & =\|g\|_{n,L}\big((1+\beta/2)^{2L}+(1-\beta/2)^{2L}\big)+\|g\|_{2n,L}(1+\beta/2)^{2L}(1+\beta^2/4)^{n}                                        \\
                     & \leq\|g\|_{2n,L}\big((1+\beta/2)^{2L}+(1-\beta/2)^{2L}+(1+\beta/2)^{2L}(1+\beta^2/4)^{m}\big)=C_{\beta,m,L}\|g\|_{2n,L},
  \end{align*}
  which gives the claim.
\end{proof}

\section{Assumption~\eqref{H1_bar} for symmetric random variables }

\begin{theorem}
  Let $\eta:\R\to \R_+$ be a $C^\infty$ even function. Then, $\eta^*_m\geq 0$ for all $m\in\N^*$ if and only if $\eta(\sqrt{\cdot})$ is the Laplace transform of a finite positive Borel measure $\mu$ on $[0,\infty)$, i.e. $\eta(\sqrt{x})=\int_0^{\infty}e^{-tx}\mu(dt)$ for all $x\in\R_+$.
\end{theorem}
\begin{proof}
  We start to prove that $\eta^*_m\geq 0$ for all $m\in\N^*$ implies $\eta(\sqrt{x})=\int_0^{\infty}e^{-tx}\mu(dt)$ for all $x\in\R$. To prove this, we use Bernstein's Theorem for completely monotone functions (see e.g. \cite[Theorem 12a p.~160]{Widder}) and show that  for all $m\in \N$ and $x\in\R_+^*$, $(-1)^m \partial_x^m [\eta(\sqrt{x})]\geq 0$. 
  To do so, we prove by induction on~$m$ the representation $$\partial^m_x[\eta(\sqrt{x})]=-\frac{(m-1)!}{2^{2m-1}}x^{-m}\sum_{j=1}^m c_{j,m}x^{\frac{j}{2}}\eta^{(j)}(\sqrt{x})=(-1)^m\frac{(m-1)!}{2^{2m-1}}x^{-m}\eta^*_m(\sqrt{x}).$$ For $m=1$, we have $\eta^*_1(\sqrt{x})=c_{1,1} \sqrt{x} \eta'(\sqrt{x})$ and the representation holds from $\partial_x[\eta(\sqrt{x})]=\frac{1}{2\sqrt{x}}\eta'(\sqrt{x})=-\frac{1}{2x}\eta^*_1(\sqrt{x})$ using that $c_{1,1}=-1$. Now, let $m\geq 2$ and suppose the representation is true for $m-1$, so
  \begin{equation*}
    \partial_x^{m}[\eta(\sqrt{x})]=\partial_x(\partial_x^{m-1}[\eta(\sqrt{x})])=\partial_x\Bigg(-\frac{(m-2)!}{2^{2m-3}}x^{-(m-1)}\sum_{j=1}^{m-1}c_{j,m-1}x^{\frac{j}{2}}\eta^{(j)}(\sqrt{x}) \Bigg).
  \end{equation*}
  Differentiating and using that $\partial_x \big(x^{\frac{j}{2}}\eta^{(j)}(\sqrt{x})\big)= \frac{1}{2x}\big(j x^\frac{j}{2}\eta^{(j)}(\sqrt{x})+x^{\frac{j+1}{2}}\eta^{(j+1)}(\sqrt{x})\big)$, we get
  \begin{align*}
    \partial_x^{m}[\eta(\sqrt{x})] & =
    \begin{multlined}[t] -\frac{(m-2)!}{2^{2m-3}}\Bigg(-\frac{m-1}{x^{m}} \sum_{j=1}^{m-1}c_{j,m-1}		x^{\frac{j}{2}}\eta^{(j)}(\sqrt{x}) \\
      +\frac{1}{2x^{m}}\sum_{j=1}^{m-1} c_{j,m-1}\bigg(j x^\frac{j}{2}	\eta^{(j)}(\sqrt{x})+x^{\frac{j+1}{2}}\eta^{(j+1)}(\sqrt{x})\bigg) \Bigg)\end{multlined}                                 \\
                                 & =\begin{multlined}[t] -\frac{(m-2)!}{2^{2m-3}}x^{-m}\Bigg( \Big(\frac{1}{2}- m-1\Big)  c_{1,m-1} x^{\frac{1}{2}}\eta^{(1)}(\sqrt{x})\\
      +\sum_{j=1}^{m-1}\bigg(\Big(\frac{j}{2}- m+1\Big)  c_{j,m-1} +\frac{1}{2} 	c_{j-1,m-1}\Big)	x^{\frac{j}{2}}\eta^{(j)}(\sqrt{x}) \\
      +\frac{1}{2}c_{m-1,m-1}x^\frac{m}{2}	\eta^{(m)}(\sqrt{x})\Bigg)\end{multlined} \\
                                 & =\begin{multlined}[t] -\frac{(m-1)!}{2^{2m-1}}x^{-m}\Bigg( \Big(\frac{2}{m-1}- 4\Big)  c_{1,m-1} x^{\frac{1}{2}}\eta^{(1)}(\sqrt{x})\\
      +\sum_{j=1}^{m-1}\bigg(\Big(\frac{2j}{m-1}- 4\Big)  c_{j,m-1} +\frac{2}{m-1} 	c_{j-1,m-1}\Big)	x^{\frac{j}{2}}\eta^{(j)}(\sqrt{x}) \\
      +\frac{2}{m-1}c_{m-1,m-1}x^\frac{m}{2}	\eta^{(m)}(\sqrt{x})\Bigg)\end{multlined}
  \end{align*}
  and we conclude using the recursion formula \eqref{recursive_coeff_formula} for $c_{j,m}$.

  We now assume that $\eta(\sqrt{x})=\int_0^\infty e^{-tx} \mu(dt)$ and show that $\eta^*_m\ge 0$ for all $m\ge 1$. We define $\eta_g(x)=e^{-\frac{x^2}{2}}$ and consider for all $t>0$ the function $\eta_t(x)=e^{-tx^2}$. We remark that for all $t>0$, $\eta_t(x)=\eta_g(h_t(x))$ with $h_t(x)=\sqrt{2t}x$ and so we can write by Lemma~\ref{gaussian_eta_positivity}
  $$(\eta_t)^*_m(x)=(-1)^{m-1}\sum_{j=1}^m c_{j,m}x^j\eta_t^{(j)}(x)=(-1)^{m-1}\sum_{j=1}^m c_{j,m}(\sqrt{2t}x)^j\eta_g^{(j)}(\sqrt{2t}x)=(\eta_g)^*_m(\sqrt{2t}x).
  $$
  Therefore, $(\eta_t)^*_m(x)\geq 0$ for all $t>0$ and $x\in\R$.
  We now consider an even function $\eta:\R \to \R_+$ such that $\eta(\sqrt{x})=\int_0^{\infty}e^{-tx}\mu(dt)$ for some Borel measure~$\mu$ on $[0,\infty)$. We then have for all $x\in\R$, $\eta(x)=\int_0^{\infty} e^{-tx^2}\mu(dt)=\int_0^{\infty}\eta_t(x) \mu(dt)$ and thus $\eta^{(j)}(x)=\int_0^{\infty}\eta^{(j)}_t(x) \mu(dt)$. This gives, for all $m\in\N^*$,
  \begin{align*}
    \eta^*_m(x) = (-1)^{m-1}\sum_{j=1}^m c_{j,m}x^j\eta^{(j)}(x) & =\int_0^\infty(-1)^{m-1}\sum_{j=1}^m c_{j,m}x^j\eta_t^{(j)}(x)\mu(dt) \\
                                                                 & =\int_0^\infty(\eta_t)^*_m(x)\mu(dt)\geq 0
  \end{align*}
  where the last integral is positive for all $x\in\R$ because is an integral of a positive function against a positive measure.
\end{proof}

\begin{corollary}\label{cor_density_etam}
  All the densities that satisfy the hypothesis of the representation Lemma \ref{regular_density} for all $m\in\N^*$ are such that $\eta(\sqrt{\cdot})$ is the Laplace transform of a finite positive Borel measure $\mu$ over $[0,\infty)$.
\end{corollary}

%%%%%%%%%%%%%%%%%%%%%%%%%%%%%%%%%%%%%%%%%%%%%%%%%%%%%%%%%%%%%%%
%%%%%%%%%%%%%%%%%%%%%%%%%%%%%%%%%%%%%%%%%%%%%%%%%%%%%%%%%%%%%%%
%%%%%%%%%%%%%%%%%%%%%%%%%%%%%%%%%%%%%%%%%%%%%%%%%%%%%%%%%%%%%%%
\bibliographystyle{abbrv}
\bibliography{Biblio_Ordre4}

\end{document}